\documentclass[a4paper,12pt,intlimits,oneside]{amsart}
\usepackage{amsmath}
\usepackage{amsthm}
\usepackage{latexsym}
\usepackage{amssymb}
\usepackage{xcolor}
\usepackage[colorlinks=true]{hyperref}
\numberwithin{equation}{section}
\numberwithin{figure}{section}
\def\expo_#1{{\rm e}^{#1}}
\def\rh{\rho}
\def\si{\sigma}
\def\f{\varphi}
\def\t{\theta}
\def\R{{\mathbb R}}
\def\C{{\mathbb C}}
\def\D{{\mathbb D}}
\def\T{{\mathbb T}}
\def\Z{{\mathbb Z}}
\def\S{{\mathbb S}}

\def\s{\vskip 0.25cm\noindent}
\def\pa{\partial}
\def\build#1_#2^#3{\mathrel{\mathop{\kern 0pt#1}\limits_{#2}^{#3}}}
\def\td_#1,#2{\mathrel{\mathop{\build\longrightarrow_{#1\rightarrow #2}^{}}}}
\def\e{\varepsilon}
\newcommand{\ben}{\begin{equation}}
\newcommand{\een}{\end{equation}}

\newcommand{\beno}{\begin{eqnarray*}}
\newcommand{\eeno}{\end{eqnarray*}}
\newtheorem*{theo}{Theorem}
\newtheorem{theorem}{Theorem}
\newtheorem{corollary}{Corollary}
\newtheorem{proposition}{Proposition}
\newtheorem{lemma}{Lemma}
\newtheorem{remark}{Remark}
\newtheorem{definition}{Definition}
\date{February 7, 2014}
\begin{document}
\title[Multiple singular values of Hankel operators]{Multiple singular values\\ of Hankel operators}
\author{Patrick G\'erard}
\address{Universit\'e Paris-Sud XI, Laboratoire de Math\'ematiques
d'Orsay, CNRS, UMR 8628, et Institut Universitaire de France} \email{{\tt Patrick.Gerard@math.u-psud.fr}}

\author[S. Grellier]{Sandrine Grellier}
\address{F\'ed\'eration Denis Poisson, MAPMO-UMR 6628,
D\'epartement de Math\'ematiques, Universit\'e d'Orleans, 45067
Orl\'eans Cedex 2, France} \email{{\tt
Sandrine.Grellier@univ-orleans.fr}}


\begin{abstract}
The goal of this paper is to  construct a nonlinear Fourier transformation on the space of symbols of compact Hankel operators on the circle. This transformation allows to solve a general inverse spectral problem involving singular values  of a compact Hankel operator, with arbitrary multiplicities. The formulation of this result requires the introduction of the pair made with a Hankel  operator and its shifted Hankel operator.  As an application, we prove that the space of symbols of compact Hankel operators on the circle admits a singular foliation made of tori of finite or infinite dimensions, on which the flow of the cubic Szeg\H{o} equation acts. In particular, we infer that arbitrary solutions of the cubic Szeg\H{o} equation on the circle with finite momentum are almost periodic with values in $H^{1/2}(\S ^1)$.

\end{abstract}

\thanks { The authors are grateful to T. Kappeler, V. Peller and A. Pushnitski for valuable discussions and comments about this paper, to L. Baratchart for pointing them references \cite{AAK} and \cite{Bar} in connection with Proposition \ref{action}, and to A. Nazarov for references concerning Bateman--type formulae. The first author is supported by ANR ANAE 13-BS01-0010-03.}

\maketitle
\tableofcontents

\section{Introduction} \label{intro}

The theory of Hankel operators has many applications in various areas of mathematics, such as  operator theory, approximation theory, control theory. We refer to the books \cite{Pe2} and \cite{N} for  a systematic presentation of this theory. More recently, spectral theory of Hankel operators arose as a key tool in the study  of some completely integrable Hamiltonian system, called the cubic Szeg\H{o} equation, see \cite{GG1}, \cite{GG2}, \cite{GG4}. The goal of this paper is two-fold. On the one hand, we present the complete solution of some double inverse spectral problem for compact Hankel operators. On the other hand, we apply this theory in order to obtain qualitative results on the dynamics of the cubic Szeg\H{o} equation.

We first recall the definition of a Hankel operator on the space $\ell ^2(\Z _+)$. Given a sequence $c=(c_n)_{n\ge 0}\in \ell ^2(\Z _+)$, the associated Hankel operator $\Gamma _c$ is formally defined by
$$\forall x=(x_n)_{n\ge 0}\in\ell^2(\Z_+)\ ,\ \Gamma _c(x)_n=\sum _{k=0}^\infty c_{n+k}x_k\ .$$
Hankel operators are strongly related to the shift operator
$$ \begin{array}{lcl} \Sigma : &\ell ^2(\Z _+)&\longrightarrow \ell ^2(\Z _+)\\
&(x_0,x_1,x_2,\cdots )&\longmapsto (0,x_0,x_1,x_2,\cdots )\ ,\end{array}$$
and to its adjoint
$$ \begin{array}{lcl} \Sigma ^*: &\ell ^2(\Z _+)&\longrightarrow \ell ^2(\Z _+)\\
&(x_0,x_1,x_2,\cdots )&\longmapsto (x_1,x_2,\cdots )\ .\end{array}$$
Indeed, Hankel operators are those operators $\Gamma $ on $\ell ^2(\Z _+)$ such that
\begin{equation}\label{Hankelshift}
\Sigma ^*\Gamma =\Gamma \Sigma \ .
\end{equation}
A famous result due to Nehari \cite{Ne} characterizes the boundedness of  $\Gamma _c$  on $\ell ^2(\Z _+)$ by
$$\exists f\in L^\infty (\T ): \forall n\ge 0,\ c_n=\hat f(n)\ ,$$
where $\hat f$ denotes the sequence of Fourier coefficients of any distribution $f$ on $\T :=\R /2\pi \Z $. Using Fefferman's theorem \cite{F}, this is equivalent to $u_c \in BMO(\T )$, where we define
\begin{equation}\label{uc}
 u_c(\expo _{ix}):=\sum _{n=0}^\infty  c_n{\rm e}^{inx}\ ,\ x\in \T .
 \end{equation}
Throughout this paper, we shall focus on the special case where $\Gamma _c$ is compact, which corresponds to $u_c\in VMO(\T )$ by 
a theorem due to Hartman \cite{Ha}. 

\subsection{Inverse spectral theory of self-adjoint compact Hankel operators} We first discuss the case of self-adjoint operators $\Gamma _c$,
which corresponds to a real valued sequence $c$. Assume moreover that $\Gamma _c$ is compact. The spectrum of $\Gamma _c$ consists of $0$ and  of a finite or infinite sequence of real nonzero  eigenvalues $(\lambda _j)_{j\ge 1}$, repeated according to their  finite multiplicities. A natural question is the following : 
{\sl given any finite or infinite sequence of nonzero real numbers $(\lambda_j)_{j\ge 1}$, does there exist a compact selfadjoint Hankel operator $\Gamma_c$ having this sequence as non zero eigenvalues, repeated according to their multiplicity?} 

Of course, if the sequence $(\lambda _j)$ is infinite, it is necessary that $\lambda _j$ tends to $0$. A much more subtle constraint was found by Megretskii-Peller--Treil in \cite{MPT}, who proved the following theorem, which we state only in the compact case.

\begin{theo}[Megretskii, Peller, Treil]
A finite or infinite sequence  $(\lambda _j)$ of nonzero real numbers is the sequence of nonzero eigenvalues of  a compact selfadjoint Hankel operator if and only if 
\begin{enumerate}
\item If $(\lambda _j)$ is infinite, then $\lambda _j \td_j,\infty 0 $ ; 
\item For any $\lambda \in \R^*$, $\vert  \# \{ j: \lambda _j=\lambda \} - \# \{ j: \lambda _j=-\lambda \} \vert \le 1.$
\end{enumerate}
\end{theo}
 
 Our first objective is to describe the set of solutions of this inverse problem, namely the isospectral sets for any given sequence $(\lambda _j)$. 
 Observe that, even in the rank one case, there is no uniqueness to be expected. Indeed, $\Gamma _c$ is a selfadjoint rank one operator  if and only if
 $$c_n= \alpha p^n\ ,\ \alpha \in \R ^*\ ,\ p\in (-1,1)\ .$$
 In this case, the only nonzero eigenvalue is
 $$\lambda _1=\frac{\alpha }{1-p^2}\ .$$
 Isospectral sets are therefore manifolds diffeomorphic to $\R $. Hence we need to introduce additional parameters. The study of the cubic Szeg\H{o} equation led us to introduce a second Hankel operator $\Gamma _{\tilde c}$  in \cite{GG2}, where 
 $$\tilde c_n:=c_{n+1}\ ,\ n\in \Z _+\ .$$
 Notice that $\Gamma _{\tilde c}$ is quite a natural operator since it is precisely the operator arising in the identity (\ref{Hankelshift})
 where $\Gamma =\Gamma _c$. Coming back to the rank one case, we observe that, if $p\ne 0$, the only nonzero eigenvalue of 
 $\Gamma _{\tilde c}$  is 
 $$\mu _1=\frac{\alpha p}{1-p^2}\ .$$
 Notice that the knowledge of $\lambda _1$ and $\mu _1$ characterizes $\alpha $ and $p$, hence $c$. More generally, 
 it is easy to check from (\ref{Hankelshift}) that
 \begin{equation}\label{Gamma2} 
 \Gamma _{\tilde c}^2=\Gamma _c^2-(\cdot \vert c)c\ .
 \end{equation}
Denoting by $(\lambda _j)$, $(\mu _k)$ the sequences of nonzero eigenvalues of $\Gamma _c$, $\Gamma _{\tilde c}$ respectively, 
labelled in decreasing order of their absolute values,  this identity implies, from the min-max formula, the following interlacement inequalities,
\begin{equation}\label{interlace}
\vert \lambda _1\vert \ge \vert \mu _1\vert \ge \vert \lambda _2\vert \ge \dots 
\end{equation}
Let us now investigate the inverse spectral problem for both operators $\Gamma _c, \Gamma _{\tilde c}$. In the special case where inequalities in (\ref{interlace}) are strict, we proved in \cite{GG2} and \cite{GG3} that this problem admits a unique solution $c$. In the general case, 
one can prove that eigenspaces of $\Gamma _c^2$ and $\Gamma _{\tilde c}^2$ corresponding to the same positive eigenvalue, are such that one of them is of codimension $1$ into the other one. As a consequence, in the sequence of inequalities (\ref{interlace}), the length of every maximal string with consecutive equal terms is odd. Our first result is that this condition is optimal.
\begin{theorem}\label{realmain}
Let $(\lambda _j), (\mu _k)$ be two finite or infinite tending to zero sequences of nonzero real numbers satisfying
\begin{enumerate}
\item $\vert \lambda _1\vert \ge \vert \mu _1\vert \ge \vert \lambda _2\vert \ge \dots $
\item In the above sequence of inequalities, the lengths of maximal strings with consecutive equal terms are odd. Denote them by $2d_r+1$.
\item For any $\lambda \in \R^*$, $\vert  \# \{ j: \lambda _j=\lambda \} - \# \{ j: \lambda _j=-\lambda \} \vert \le 1.$
\item For any $\mu \in \R^*$, $\vert  \# \{ k: \mu _k=\mu \} - \# \{ k: \mu _k=-\mu \} \vert \le 1.$
\end{enumerate}
Then there exists a sequence $c$ of real numbers such that $\Gamma _c$ is compact and the nonzero eigenvalues of $\Gamma _c $ and $\Gamma _{\tilde c}$ are respectively the $\lambda _j$'s and the $\mu _k$'s. Moreover, introduce
$$M:=\sum _r d_r\ \in \Z _+\cup \{\infty \}\ .$$
The isospectral set is a manifold diffeomorphic to $\R ^M$ if $M<\infty $, and it is homeomorphic to $\R ^\infty  $ if $M=\infty $.
\end{theorem}
Moreover, in the case of a finite sequence of nonzero eigenvalues, one can produce explicit formulae for $u_c$. For instance, given 
four real numbers $\lambda _1,\mu _1, \lambda _2, \mu _2$ such that
$$\vert \lambda _1 \vert > \vert \mu _1\vert >\vert \lambda _2 \vert > \vert \mu _2\vert >0 \ ,$$
we get
\beno
u_c(\expo_{ix})=\frac{\displaystyle{\frac{\lambda _1-\mu _1\expo_{ix}}{\lambda _1^2-\mu _1^2}+\frac{\lambda _2-\mu _2\expo_{ix}}{\lambda _2^2-\mu _2^2}-\frac{\lambda _1-\mu _2\expo_{ix}}{\lambda _1^2-\mu _2^2}-\frac{\lambda _2-\mu _1\expo_{ix}}{\lambda _2^2-\mu _1^2}}}
{ \displaystyle {\left \vert \begin{array}{ll} \frac{\lambda _1-\mu _1\expo_{ix}}{\lambda _1^2-\mu _1^2}&\frac{\lambda _2-\mu _1\expo_{ix}}{\lambda _2^2-\mu _1^2}
\\
\frac{\lambda _1-\mu _2\expo_{ix}}{\lambda _1^2-\mu _2^2}&\frac{\lambda _2-\mu _2\expo_{ix}}{\lambda _2^2-\mu _2^2}\end{array}
\right \vert }} \eeno
If $\vert \lambda _1 \vert  >\vert \lambda _2 \vert >0$ and $\mu _1=\lambda _2, \mu _2=-\lambda _2$, then, there exists $p\in (-1,1)$ such that
$$u_c(\expo_{ix})=(\lambda_1^2-\lambda_2^2)\frac{1-p\, {\rm e}^{ix}}{\lambda_1-p\, {\rm e}^{ix}(\lambda_1-\lambda_2)-\lambda_2\, {\rm e}^{2ix}}\ .$$
Finally, notice that, if $\lambda _1, \lambda _2$ are given such that $\vert \lambda _1 \vert  >\vert \lambda _2 \vert >0$, the corresponding isospectral set consists of sequences $c$ given by the above  two formulae. Also notice that the second expression is obtained from the first one
by making $\mu _1\rightarrow \lambda _2\ ,\ \mu _2\rightarrow -\lambda _2\  $, and
$$\frac{2\lambda _2+\mu _2-\mu _1}{\mu _1+\mu _2}\rightarrow p\ .$$

\subsection{Complexification and the Hardy space representation}
In the general case where $c$ is complex-valued, $\Gamma _c$
is no more selfadjoint, and the natural inverse spectral problem rather concerns singular values of $\Gamma _c$, namely 
square roots of nonzero eigenvalues of $\Gamma _c\Gamma _c^*$. In order to have a better understanding of the multiplicity phenomena, we are going to change the representation of these operators. A natural motivation for this new representation comes back to a celebrated paper by Beurling \cite{B}
characterizing the closed subspaces of $\ell ^2(\Z _+)$ invariant by $\Sigma $. The connection with Hankel operators is made by the observation that, because of identity (\ref{Hankelshift}), the kernel of a Hankel operator is always such a space. According to Beurling's theorem, these spaces can be easily described by using the isometric Fourier isomorphism
$$\begin{array}{lcl}&\ell ^2(\Z _+)&\longrightarrow L^2_+(\T )\\
&c &\longmapsto u_c
\end{array}$$
where $L^2_+(\T )$ denotes the closed subspace of $L^2(\T )$ made of functions $u$ with 
$$\forall n <0\ ,\ \hat u(n)=0\ ,$$
endowed with the inner product
$$(f\vert g)=\frac{1}{2\pi }\int _{\T} f(\expo_{ix})\overline{g(\expo_{ix})}\, dx\ .$$
Notice that $L^2_+(\T )$ is isomorphic to the Hardy space ${\mathbb H}^2({\mathbb D})$ of holomorphic functions on the unit disc
with $L^2$ traces on the unit circle. Under this representation, the shift operator $\Sigma $ becomes 
$$\begin{array}{lcl}S:&{\mathbb H}^2({\mathbb D})&\longrightarrow {\mathbb H}^2({\mathbb D})\\
&f &\longmapsto zf\ ,
\end{array}$$
and its adjoint reads
$$\begin{array}{lcl}S^*:&{\mathbb H}^2({\mathbb D})&\longrightarrow {\mathbb H}^2({\mathbb D})\\
&f &\longmapsto \Pi (\overline zf)\ ,
\end{array}$$
where $\Pi $ is the orthogonal projector from $L^2(\T )$ onto $L^2_+(\T )\simeq {\mathbb H}^2({\mathbb D})$, usually referred as the Szeg\H{o} projector.
Using this representation, the Beurling theorem claims that non trivial closed subspaces of ${\mathbb H}^2({\mathbb D})$ invariant by $S$ are exactly 
the spaces
$$\Psi {\mathbb H}^2({\mathbb D})\ ,$$
where $\Psi $ is an inner function, namely a bounded holomorphic function on $\mathbb D$ with modulus $1$ on the unit circle.
\s
Let us come back to Hankel operators. Using the above representation, $\Gamma _c$ corresponds to an operator $H_u$, $u=u_c$,
defined by
$$H_u(h)=\Pi (u\overline h)\ ,\ h\in L^2_+(\T )\ .$$
Precisely, we have the following identity, for every $u\in BMO_+(\T ):= BMO(\T )\cap L^2_+(\T )\ ,$
$$\widehat {H_u(h)}=\Gamma _{\hat u}\left (\, \overline{\hat h}\, \right )\ ,\ h\in L^2_+(\T )\ .$$
Notice that $H_u$ is an antilinear operator, satisfying the following selfadjointness property,
\ben \label{selfadjoint}
(H_u(h_1)\vert h_2)=(H_u(h_2)\vert h_1)\ ,\ h_1,h_2\in L^2_+(\T )\ .
\een
Consequently, $H_u^2$ is a linear positive selfadjoint operator on $L^2_+(\T )$, which is conjugated to $\Gamma _c\Gamma _c^*$ through the Fourier representation, hence the square roots of its positive eigenvalues are the singular values of $\Gamma _c$ or of $H_u$. In the same way, 
$\Gamma _{\tilde c}$ corresponds to 
\ben \label{Ku}
K_u:=S^*H_u=H_uS=H_{S^*u}\ ,
\een
and identity (\ref{Gamma2}) reads
\ben \label{Ku2}
K_u^2=H_u^2-(\, \cdot \, \vert u)u\ .
\een
For every $s\ge 0$ and $u\in VMO_+(\T ):=VMO(\T )\cap L^2_+(\T )$, we set
\ben \label{propres}
E_u(s):=\ker (H_u^2-s^2I)\ ,\ F_u(s):=\ker (K_u^2-s^2I)\ .
\een
Notice that $E_u(0)=\ker H_u\ ,\  F_u(0)=\ker K_u$. Moreover, from the compactness of $H_u$, if $s>0$, $E_u(s)$ and $F_u(s)$ are finite dimensional.
Using (\ref{Ku}) and (\ref{Ku2}), one can show the following result.
\begin{lemma}\label{EF}
Let $s>0$ such that $E_u(s)\ne \{ 0\} $ or $F_u(s)\ne \{ 0\}$.Then one of the following properties holds. 
\begin{enumerate}
\item $\dim E_u(s)=\dim F_u(s)+1$,  $u \not \perp E_u(s)$, and $F_u(s)=E_u(s)\cap u^\perp $.
\item $\dim F_u(s)=\dim E_u(s)+1$,  $u \not \perp F_u(s)$, and $E_u(s)=F_u(s)\cap u^\perp $.
\end{enumerate}
\end{lemma}
We define
\begin{equation}\label{SigmaH}
\Sigma_H(u):=\left\{s\ge 0;\; u\not  \perp  E_u(s)\right\}, 
\end{equation}
\begin{equation}\label{SigmaK}
\Sigma_K(u):=\left\{s\ge 0;\; u\not \perp F_u(s)\right\}, 
\end{equation}
Remark that $0\notin \Sigma_H(u)$, since $u=H_u(1)$ belongs to the range of $H_u$ hence, is orthogonal to its kernel. As a consequence of Lemma \ref{EF}, $\Sigma_H(u)$ coincides with the set of $s>0$ with $\dim E_u(s)=\dim F_u(s)+1$.
In contrast, it may happen that $0$ belongs to $\Sigma_K(u)$.
\subsection{Multiplicity and Blaschke products}
Assume $u\in VMO_+(\T )$ and $s\in \Sigma_H(u)$. Then $H_u$ acts on the finite dimensional vector space $E_u(s)$.
It turns out that this action can be completely described by an inner function. A similar fact holds for the action of $K_u$ onto $F_u(s)$ when $s\in \Sigma_K(u)$, $s\neq 0$.
In order to state this result, recall that  a finite Blaschke product is an inner function of the form
$$\Psi (z)=\expo_{-i\psi}\prod _{j=1}^k \chi _{p_j}(z)\ ,\  \psi \in \T \ ,\ p_j\in \D\ ,\ \chi _p(z):=\frac{z-p}{1-\overline pz}\ ,\ p\in \D \ .$$
The integer $k$ is called the degree of $\Psi $. Alternatively, $\Psi$ can be written as 
$$\Psi (z)=\expo_{-i\psi}\frac{P(z)}{z^k\overline P\left (\frac 1z\right )}\ ,$$
where $\psi \in \T $ is called the angle of $\Psi $ and $P$ is a monic polynomial of degree $k$ with all its roots in $\D $.  Such polynomials are called Schur polynomials. We denote by $\mathcal B_k$ the set of Blaschke products of degree $k$. It is a classical result --- see {\it e.g.} \cite{He} or Appendix B --- that $\mathcal B_k$ is diffeomorphic to $\T \times \R ^{2k}$. Finally, we shall denote by
$$D(z)=z^k\overline P\left (\frac 1z\right )$$
the normalized denominator of $\Psi $.
\begin{proposition}\label{action}
Let $s>0$ and $u\in VMO_+(\T )$.
\begin{enumerate}
\item Assume $s\in \Sigma_H(u)$ and $m:=\dim E_u(s)=\dim F_u(s)+1$. Denote by $u_s$ the orthogonal projection of $u$ onto $E_u(s)$. There exists an inner function $\Psi _s\in \mathcal B_{m-1}$ such that
$$su_s=\Psi _sH_u(u_s)\ ,$$
and if $D$ denotes the normalized denominator of $\Psi _s$, 
\begin{eqnarray} \label{descEF1}
E_u(s)&=&\left \{ \frac fDH_u(u_s)\ ,\ f\in \C _{m-1}[z]\right \} \ ,\\
F_u(s)&=&\left\{ \frac gDH_u(u_s)\ ,\ g\in \C _{m-2}[z]\right \}, 
\end{eqnarray}
and, for $a=0,\dots,m-1\ ,\ b=0,\dots ,m-2$, 
\begin{eqnarray} \label{actionHu}
H_u\left (\frac {z^a}DH_u(u_s)\right )&=&s\expo_{-i\psi _s}\frac {z^{m-a-1}}DH_u(u_s)\ , \\
K_u\left (\frac {z^b}DH_u(u_s)\right )&=&s\expo_{-i\psi _s}\frac {z^{m-b-2}}DH_u(u_s)\ ,
\end{eqnarray}
where $\psi _s$ denotes the angle of $\Psi _s$.
\item Assume $s\in \Sigma_K(u)$ and $\ell :=\dim F_u(s)=\dim E_u(s)+1$. Denote by $u_s'$ the orthogonal projection of $u$ onto $F_u(s)$. There exists an inner function $\Psi _s\in \mathcal B_{\ell -1}$ such that
$$K_u(u_s')=s\Psi _su_s'\ ,$$ and if $D$ denotes the normalized denominator of $\Psi _s$, 
\begin{eqnarray} \label{descFE2}
F_u(s)&=&\left \{ \frac fDu_s'\ ,\ f\in \C _{\ell -1}[z]\right \} \ ,\\
E_u(s)&=&\left \{ \frac {zg}Du_s'\ ,\ g\in \C _{\ell -2}[z]\right \}, 
\end{eqnarray}
and, for $a=0,\dots,\ell -1\ ,\ b=0,\dots ,\ell-2$, 
\begin{eqnarray} \label{actionKu}
K_u\left (\frac {z^a}Du_s'\right )&=&s\expo_{-i\psi _s}\frac {z^{\ell -a-1}}Du_s'\ , \\
H_u\left (\frac {z^{b+1}}Du_s'\right )&=&s\expo_{-i\psi _s}\frac {z^{\ell-b-1}}Du_s'\ ,
\end{eqnarray}
where $\psi _s$ denotes the angle of $\Psi _s$.
\end{enumerate}
\end{proposition}
Notice that, if we come back to the case of selfadjoint Hankel operators, which corresponds to symbols $u$ with real Fourier coefficients,
the angles $\psi _s$ belong to $\{ 0,\pi \} $, and condition (2) in the Megretskii--Peller--Treil Theorem is an elementary consequence
of Proposition \ref{action}. Indeed, the identities in Proposition \ref{action} provide very simple matrices for the action of $H_u$ and $K_u$ on $E_u(s)$ and $F_u(s)$, and one can easily check that the dimensions of the eigenspaces of these matrices associated to the eigenvalues $\pm s$ differ
of at most $1$.
\s
Part of the content of Proposition \ref{action} was in fact already proved by Adamyan-Arov-Krein in their famous paper \cite{AAK}.
Indeed, translating Theorem 1.2 of this paper, in the special case of finite multiplicity, into our normalization, one gets that, for every pair $(h,f)\in E_u(s)\times E_u(s)$ satisfying
$$H_u(h)=sf\ ,\ H_u(f)=sh \ ,$$
there  exists a polynomial $Q$ of degree at most $m-1$, where $m=\dim E_u(s)$, and a function $g\in L^2_+$ such that
$$h(z)=Q(z)g(z)\ ,\ f(z)= z^{m-1} \overline Q\left (\frac 1z\right )g(z)\ .$$
We refer the reader to Appendix C for a self-contained proof of this property. It is easy to check that this property is a consequence of Proposition \ref{action}, which in fact says more about the structure
of the space $E_u(s)$. 

\subsection{Main results}
We  now come to  the main results of  this paper. First let us introduce some additional notation. Given a positive integer $n$,
we set $$\Omega _n:=\{s_1>s_2>\dots>s_n>0\}\subset \R^n\ . $$
Similarly $\Omega_\infty$ is the set of sequences $(s_r)_{r\ge 1}$ such that $$s_1>s_2>\dots >s_n\to 0 \ .$$ We set 
$$\Omega:=\bigcup_{n=1}^\infty \Omega_n\cup \Omega_\infty\ ,\  \mathcal B=\bigcup_{k=0}^\infty \mathcal B_k\ ,$$
and
$$\mathcal S_n=\Omega_n\times \mathcal B^n\ ,\  \mathcal S_\infty=\Omega_\infty\times \mathcal B^\infty\ ,\  \ \mathcal S:=\bigcup_{n=1}^\infty \mathcal S_n\cup \mathcal S_\infty \ .$$
Given $u\in VMO_+(\T)\setminus \{ 0\} $, one can define, according to proposition \ref{action} and (\ref{Ku2}) combined with the min-max formula, a finite or infinite sequence $s=(s_1>s_2>\dots )\in\mathcal S$ such that
\begin{enumerate}
\item The $s_{2j-1}$'s   are the singular values of $H_u$ in $\Sigma_H(u)$.
\item The $s_{2k}$'s are the singular values of $K_u$ in $\Sigma_K(u)\setminus\{0\}$.
\end{enumerate}
For every $r\ge 1$, associate to each $s_r$ an inner function $\Psi _r$ by means of Proposition \ref{action}. This defines a mapping
$$\Phi : VMO _+(\T )\setminus \{ 0\} \longrightarrow \mathcal S\ .$$
\begin{theorem}\label{Phibij}
The map $\Phi $ is bijective. 

\noindent Moreover, we have the following explicit formula for $\Phi ^{-1}$ on $\mathcal S_n$. If $n=2q$ is even,
$$\rho _j:=s_{2j-1}\ ,\ \sigma _k:=s_{2k}\ ,\ j,k=1,\dots ,q\ ,$$
introduce the $q\times q$ matrix $\mathcal C(z)$ with coefficients
$$c_{kj}(z):=\frac{\rho _j-\sigma _kz\Psi _{2k}(z)\Psi _{2j-1}(z)}{\rho _j^2-\sigma _k^2}\ ,\ j,k=1,\dots ,q\ .$$
Denote by $\Delta _{kj}(z) $ the minor determinant of this matrix corresponding to line $k$ and column $j$. We have
\ben \label{luminy}
u(z)=\sum _{1\le j,k\le q}(-1)^{j+k}\Psi _{2j-1}(z)\frac{\Delta _{kj}(z)}{\det (\mathcal C(z))}\ .
\een
If $n=2q-1$ is odd, the same formula holds by setting $\sigma _q:=0$.
\end{theorem}
Let us now come to the topological features of the mapping $\Phi $. We shall not describe the topology on $\mathcal S$ transported by $\Phi $ from  $VMO _+(\T )$, because it is a complicated matter. In the finite rank case, it is simpler to deal with the restrictions of $\Phi $ to the preimages of $\mathcal S_n$, which we denote by $\mathcal U_n$. We endow $\mathcal U_n$ with the topology induced by $VMO_+$, each $\Omega _n$ with the topology induced by $\R ^n$, $\mathcal B$ with the disjoint sum of topologies of $\mathcal B_k$,  and $\mathcal S _n$ with the product topology. In the infinite rank case, $\Omega _\infty $ is endowed with the topology induced by $c_0$, the Banach space of sequences tending to $0$.
\begin{theorem}\label{Phitopodif}
The following restriction maps of $\Phi $,
$$\Phi _n:\mathcal U_n\rightarrow \mathcal S_n $$
are homeomorphisms. Moreover, given a positive integer $n$, and a sequence $(d_1,\dots ,d_n)$ of nonnegative integers, the  map
$$\Phi ^{-1}:\Omega _n\times \prod _{r=1}^n\mathcal B_{d_r}\longrightarrow VMO_+(\T )$$
is a smooth embedding. Given a sequence $(d_r)_{r\ge 1}$ of nonnegative integers, the  map
$$\Phi ^{-1}:\Omega _\infty \times \prod _{r=1}^\infty \mathcal B_{d_r}\longrightarrow VMO_+(\T )$$
is a continuous embedding.
\end{theorem}
As a consequence of the second statement of Theorem \ref{Phitopodif}, the set
$$\mathcal V_{(d_1,\dots ,d_n)}:=\Phi ^{-1}\left (\Omega _n\times \prod _{r=1}^n\mathcal B_{d_r}\right )$$
is a submanifold of $VMO_+(\T )$ of dimension $$\dim \mathcal V_{(d_1,\dots ,d_n)}=2n+2\sum _{r=1}^nd_r\ .$$ Notice that $\mathcal V_{(d_1,\dots ,d_n)}$ is the set of symbols
$u$ such that
\begin{enumerate}
\item The singular values $s$ of $H_u$ in $\Sigma_H(u)$, ordered decreasingly, have respective multiplicities 
$$d_1+1,d_3+1,\dots \ .$$
\item The singular values $s$ of $K_u$ in $\Sigma_K(u)$, ordered decreasingly, have respective multiplicities 
$$d_2+1,d_4+1,\dots \ .$$
\end{enumerate}
In the last section of this paper, we shall investigate the properties of the manifold $\mathcal V_{(d_1,\dots ,d_n)}$ with respect to the symplectic form.
\subsection{Applications to the cubic Szeg\H{o} equation}
The cubic Szeg\H{o} equation has been introduced in \cite{GG1} as a toy model of Hamiltonian evolution PDEs with lack of dispersion.
It can be formally described as the Hamiltonian equation on $L^2_+(\T )$ associated to the energy
$$E(u):=\frac 14 \int _\T \vert u\vert ^4\, dx\ ,$$
and to the symplectic form
$$\omega (h_1,h_2):={\rm Im} (h_1\vert h_2)\ .$$
It reads
\ben \label{szego}
i\dot u=\Pi (\vert u\vert ^2u)\ .
\een
 For every $s\ge 0$, we denote by $H^s(\T )$ the Sobolev space of regularity $s$  
on $\T $, and 
$$H^s_+(\T ):=H^s(\T )\cap L^2_+(\T )\ .$$
We first recall the wellposedness results from \cite{GG1}. For every $u_0\in H^s_+(\T )$, $s\ge \frac 12$, there exists $u\in C(\R ,H^s_+(\T ))$ unique solution of equation (\ref{szego}) with
$u(0)=u_0$. Moreover, we proved in \cite{GG1} that equation (\ref{szego}) enjoys a Lax pair structure implying that $H_{u(t)}$ remains unitarily equivalent to $H_{u_0}$, and $K_{u(t)}$ remains unitarily equivalent to $K_{u_0}$. In particular, their singular values are preserved by the evolution.
It is therefore natural to understand this evolution through the mapping $\Phi $ introduced in the previous subsection. This question was solved in \cite{GG2} in the special case of generic states $u$ corresponding to simple singular values of $H_u$ and $K_u$. These states correspond through the map $\Phi $ to Blaschke products $\Psi _r$ of degree $0$. In this case, one can write
$$\Psi _r=\expo_{-i\psi _r}\ ,$$
and the evolution of the angle $\psi _r$ is given by
$$\frac{d\psi _r}{dt}=(-1)^{r-1}s_r^2\ .$$
More precisely, using the notation introduced in \cite{GG3}, consider the set $\mathcal V(d)$ defined by
\begin{equation}\label{V(d)}
\mathcal V(d)= \left \{u;\; {\rm rk}H_u=\left [\frac{d+1}2\right ]\ ,\ {\rm rk}K_u=\left [\frac d2\right ]\right\}\ .
\end{equation}
One can prove that $\mathcal V(d)$ is a K\"ahler submanifold of $L^2_+(\T )$, and that its open subset $\mathcal V(d)_{\rm gen}$ 
made of generic states  is diffeomorphic through $\Phi $
to 
$$\Omega _d\times \mathcal B _0^d\ .$$
The corresponding coordinates $(s_1,\dots ,s_d)\in \Omega _d$ and $(\psi _1,\dots, \psi _d)\in \T ^d$ are action angle variables on $\mathcal V(d)_{\rm gen}$, in the following sense, see \cite{GG2},
$$\omega _{\vert \mathcal V(d)_{\rm gen}}=\sum _{r=1}^d d\left (\frac{s_r^2}{2}\right )\wedge d\psi _r\ ,\ E=\frac 14 \sum _{r=1}^d (-1)^{r-1}s_r^4\ .$$
Our next result generalizes this fact to non generic states. Given $u\in H^{1/2}_+(\T )$, we decompose the Blaschke products associated to $u$ by $\Phi $ as
$$\Psi _r:=\expo _{-i\psi _r}\chi _r\ ,$$
where $\chi _r$ is a Blaschke product built with a monic Schur polynomial.
\begin{theorem}\label{evolszego}
The evolution of equation (\ref{szego}) on $H^{1/2}_+$ reads
$$\frac{ds_r}{dt}=0\ ,\ \frac{d\psi _r}{dt}=(-1)^{r-1}s_r^2\ ,\ \frac{d\chi _r}{dt}=0\ .$$
Moreover, on the manifold $\mathcal V_{(d_1,\dots ,d_n)}$ introduced in subsection 2, the restriction of the symplectic form $\omega $ and of the energy $E$ are given by
$$\omega =\sum _{r=1}^n d\left (\frac{s_r^2}{2}\right )\wedge d\psi _r\ ,\ E=\frac 14 \sum _{r=1}^n (-1)^{r-1}s_r^4\ .$$
In particular, $\mathcal V_{(d_1,\dots ,d_n)}$ is  an involutive submanifold of the K\"ahler manifold $\mathcal V(d)$ with $d=n+2\sum _{r=1}^n d_r $.
\end{theorem}
This theorem shows that the cubic Szeg\H{o} equation can be solved by using the  inverse spectral transform provided by the
mapping $\Phi ^{-1}$. We refer for instance to the first part of the book \cite{KP} for a similar situation in the case of the Korteweg--de Vries equation.

As a corollary of Theorem \ref{evolszego}, one gets the following qualitative information about all the trajectories of (\ref{szego}).
\begin{corollary}\label{almostperiodic}
Every solution of equation (\ref{szego}) with initial data in $H^{1/2}_+(\T )$ is an almost periodic function from $\R $ to $H^{1/2}_+(\T )$.
\end{corollary}
More precisely, we will show that the tori obtained as inverse images by the map $\Phi $ of the sets
$$\{ (s_r)\} \times \prod _r \S ^1\Psi _r\ ,$$
where $((s_r);(\Psi _r))\in \mathcal S$ is given, induce a singular foliation of the phase space $VMO_+(\T )\setminus \{ 0\} $. The cubic Szeg\H{o} flow acts on those tori which are included in $H^{1/2}_+(\T )$. In the generic case where all the $\Psi _r$ have degree $0$, it is easy to check that these tori are classes of unitary equivalence for the pair of operators $(H_u,K_u)$. In the general case, we introduce a more stringent unitary equivalence of which these tori are the classes. This answers a question asked to us by T. Kappeler. 

Let us end this introduction by mentioning that a natural generalization of the results of this paper would concern  bounded --- not necessarily compact --- Hankel operators. A first step in this direction was recently made in the paper \cite{GP}.

\subsection{Organization of the paper}
In section \ref{SpectralAnalysis}, we prove Lemma \ref{EF} about the eigenspaces of $H_u^2$ and $K_u^2$. Section \ref{Blaschke} is devoted to Proposition \ref{action} 
which introduces Blaschke products encoding the action of $H_u$ and $K_u$ on these eigenspaces. Section \ref{finiterank} gives the proof of the main theorem \ref{Phibij}, as well as the proof of Theorems \ref{Phitopodif} and \ref{realmain}, in the special case $n<\infty $ of finite rank Hankel operators. Section \ref{compact} deals with the case $n=\infty $ of infinite rank compact Hankel operators.  Sections \ref{szegoflow}, \ref{szegohier}, \ref{symplectic}  contain the proof of Theorem \ref{evolszego}, as well as applications to almost periodicity --- Corollary \ref{almostperiodic} --- , and to a new proof of the classification of traveling waves for the cubic Szeg\H{o} equation. Finally, section \ref{invtori} provides the  description of the singular foliation of
$VMO_+(\T )\setminus \{ 0\} $ in terms of equivalent classes for some special  unitary equivalence for the pair of operators $(H_u,K_u)$.  The paper ends with three appendices. The first one is devoted to a classical set of formulae connected to the relative determinant of two selfadjoint compact operators with a rank one difference. The second one specifies the structure of  the set of Blaschke products of a given degree. The third one gives a self-contained proof of two important results from the paper \cite{AAK} by Adamyan-Arov-Krein, which are used throughout the paper.
\section{Spectral decomposition of the Hankel operators $H_u$ and $K_u$}\label{SpectralAnalysis}

We begin with a  precise spectral analysis of operators $H_u^2$ and $K_u^2$ on the  closed range of $H_u$.

We prove a more precise version of Lemma \ref{EF}, namely
\begin{proposition}\label{rigidity}
Let $u\in VMO_+(\T )\setminus \{ 0\} $ and $s>0$ such that $$E_u(s)\ne \{ 0\} \quad or\quad F_u(s)\ne \{ 0\}\ .$$
Then one of the following properties holds. 
\begin{enumerate}
\item $\dim E_u(s)=\dim F_u(s)+1$,  $u \not \perp E_u(s)$, and $F_u(s)=E_u(s)\cap u^\perp $.
\item $\dim F_u(s)=\dim E_u(s)+1$,  $u \not \perp F_u(s)$, and $E_u(s)=F_u(s)\cap u^\perp $.
\end{enumerate}
Moreover, if $u_\rho$ and $u'_\sigma$ denote respectively the orthogonal projections of $u$ onto $E_u(\rho)$, $\rho\in \Sigma_H(u)$, and onto $F_u(\sigma)$, $\sigma\in \Sigma_K(u)$, then
\begin{enumerate}
\item $\Sigma_H(u)$ and $\Sigma_K(u)$ are disjoint, with the same cardinality;
\item if $\rho\in \Sigma_H(u)$, \begin{equation}\label{urho}
u_\rho=\Vert u_\rho\Vert^2\sum_{\sigma\in \Sigma_K(u)}\frac{u'_\sigma}{\rho^2-\sigma^2}\ ,\end{equation}
 \item if $\sigma\in \Sigma_K(u)$,
\begin{equation}\label{usigma}u'_\sigma=\Vert u'_\sigma\Vert^2 \sum_{\rho\in \Sigma_H(u)}\frac{u_\rho}{\rho^2-\sigma^2}\ .\end{equation}
\item A nonnegative number $\sigma $ belongs to $\Sigma _K(u)$ if and only if it does not belong to $\Sigma _H(u)$ and
\begin{equation}\label{eqsigma}\sum_{\rho \in \Sigma _H(u)}\frac{\Vert u_\rho \Vert ^2}{\rho ^2-\sigma ^2}=1\ .\end{equation}
\end{enumerate}
\end{proposition}
\begin{proof} Let $s>0$ be such that $E_u(s)+F_u(s)\neq \{0\}$. We first claim that either $u\perp E_u(s)$ or $u\perp F_u(s)$. Assume first $E_u(s)\neq \{0\}$ and $u\not\perp E_u(s)$,  then there exists $h\in E_u(s)$ such that $(h\vert u)\neq 0$. From equation (\ref{Ku2}), $$-(h\vert u)u=(K_u^2-s^2I)h\in (F_u(s))^\perp\ ,$$ hence $u\perp F_u(s)$. Similarly, if $F_u(s)\neq \{0\}$ and $u\not\perp F_u(s)$, then $u\perp E_u(s)$.\\
Let $s$ be such that $F_u(s)\neq \{0\}$. Assume $u\perp F_u(s)$. Then, for any $h\in  F_u(s)$, as $K_u^2=H_u^2-(\cdot\vert u)u$, $H_{u}^2(h)=K_u^2(h)=s^2 h$, hence  $  F_u(s)\subset   E_u(s)$. We claim that this inclusion is strict. Indeed, suppose it is an equality. Then $H_{u}$ and $K_{u}$ are both 
automorphisms of the vector space $$N:= F_u(s)=   E_u(s)\ .$$ 
Consequently, since $K_u=S^*H_u$, $S^*(N)\subset N$. 
On the other hand, 
 since every $h\in N$ is orthogonal to $u$, we have
$$0=(H_{u}(h)\vert u)=(1\vert H_{u}^2 h)=\sigma ^2(1\vert h)\ ,$$
hence $N\perp 1$.  Therefore, for every $h\in N$, for every integer $k$,
$(S^*)^k(h)\perp 1$. Since $S^k(1)=z^k$, we conclude that all the Fourier coefficients of $h$ are $0$, hence $N=\{ 0\}$,
a contradiction. Hence, the inclusion of $F_u(s)$ in $E_u(s)$ is strict and, necessarily $u\not\perp E_u(s)$ and $F_u(s)=E_u(s)\cap u^\perp$. One also has $\dim E_u(s)=\dim F_u(s)+1$.\\
One proves as well that if $E_u(s)\neq\{0\}$ and $u\perp E_u(s)$ then $u\not\perp F_u(s)$, $E_u(s)=F_u(s)\cap u^\perp$ and $\dim F_u(s)=\dim E_u(s)+1$.
This gives the first part of Proposition (\ref{rigidity}).\\

For the second part, we first observe that $u=H_u(1)\in E_u(0)^\perp $, hence $0\not \in \Sigma _H(u)$. From what we just proved, we conclude that $\Sigma _H(u)$ and $\Sigma _K(u)$ are disjoint. Furthermore, by the spectral theory of $H_{u}^2$ and of $K_u^2$, we have the orthogonal decompositions,
$$L^2_+=\overline{\oplus_{s\ge 0}E_u(s)}=\overline{\oplus_{s\ge 0}F_u(s)}\ .$$
Writing $u$ according to these two orthogonal decompositions yields
$$u=\sum _{\rho \in \Sigma _H(u)}u_\rho =\sum _{\sigma  \in \Sigma _K(u)}u'_\sigma \ .$$
Consequently, the cyclic spaces generated by $u$ under the action of $H_u^2$ and $K_u^2$ are given by
$$\langle u\rangle _{H_{u}^2} =\overline {\oplus _{\rho \in \Sigma _H(u)}\C u_\rho }\ ,\ \langle u\rangle _{K_{u}^2} =\overline {\oplus _{\sigma \in \Sigma _K(u)}\C u'_\sigma }\ .$$
Since $K_u^2=H_u^2-(\, .\, \vert u)u$, these cyclic spaces  are  equal. This proves that $\Sigma _H(u)$ and $\Sigma _K(u)$ have the same --- possibly infinite --- number of elements. 
\s
Let us prove (\ref{urho}) and (\ref{usigma}). Observe that, by the Fredholm alternative, for $\sigma >0$, 
$H_{u}^2-\sigma ^2I$ is an automorphism of $ E_u(\sigma)^\perp $. Consequently, if moreover $\sigma \in \Sigma _K(u)$,
$u\in E_u(\sigma )^\perp $ and there exists $v\in E_u(\sigma )^\perp $ unique such that 
$$(H_u^2-\sigma ^2I)v=u\ .$$
We set $v:=(H_{u}^2-\sigma^2I)^{-1}(u)$. If $\sigma =0\in \sigma _K(u)$, of course $H_u^2$ is no more a Fredholm operator, however
there still exists $w\in E_u(0)^\perp $ such that
$$H_u^2(w)=u\ .$$
Indeed, since $K_u=S^*H_u$, $E_u(0)\subset F_u(0)$, and the hypothesis $u\not \perp F_u(0)$ implies that the latter inclusion is strict.
This means that there exists $w\in E_u(0)^\perp $ such that $H_u(w)=1$, whence $H_u^2(w)=u$. Again, we set $w:=(H_u^2)^{-1}(u)\ .$

For every $ \sigma \in \Sigma _K(u)$, the equation
$$K_{u}^2h=\sigma ^2h $$
is equivalent to
$$(H_{u}^2-\sigma^2I)h=(h\vert u)u\ ,$$
or
$h\in \C (H_{u}^2-\sigma^2I)^{-1}(u)\oplus  E_u(\sigma)\ ,$
with \begin{equation}\label{condsigma}  ((H_{u}^2-\sigma^2I)^{-1}(u)\vert u)=1\ . \end{equation}
Since $u'_\sigma \in E_u(\sigma )^\perp $, this  leads to
$$\frac { u'_{\sigma }}{\Vert u'_{\sigma }\Vert ^2}  =(H_{u}^2-\sigma^2I)^{-1}(u)\ .$$
In particular, if $\rho\in\Sigma_H(u),$ $\sigma\in \Sigma_K(u)$,
$$\left(\frac{u'_\sigma}{\Vert u'_\sigma\Vert^2}\Big \vert \frac{u_\rho}{\Vert u_\rho\Vert^2}\right)=\frac 1{\rho^2-\sigma^2}\ .$$
This leads to  equations (\ref{urho}) and (\ref{usigma}). Finally, equation (\ref{eqsigma}) is nothing but the expression of (\ref{condsigma})  in view of  equation (\ref{urho}).
\s
This completes the proof of Proposition \ref{rigidity}.
\end{proof}
 
\section{Multiplicity and Blaschke products. Proof of Proposition \ref{action}} \label{Blaschke}

In this section, we prove Proposition \ref{action}. 

\subsection{Case of $\rho \in \Sigma_H(u)$} 
Let $u\in VMO_+(\T)$.
Assume  that $\rho\in \Sigma_H(u)$  and $m:=\dim E_u(\rho)$. We may assume $\rho =1$ and write $u_\rho=u_1$, $$E=E(1)=\ker (H_u^2-I),\; F=F(1)=\ker(K_u^2-I)$$ for simplicity. 
By Proposition \ref{rigidity}, $$F=E\cap u_1^\perp \ .$$ 
\subsubsection{Definition of $\Psi $.}
We claim that, at every point of $\T $,
$$\vert u_1\vert ^2=\vert H_u(u_1)\vert ^2\ .$$
Indeed, denoting by $S$ the shift operator, for every integer $n\ge 0$,
\begin{eqnarray*}
(\vert u_1\vert ^2\vert z^n)&=&(u_1\vert S^n u_1)=(H_u^2(u_1)\vert S^n u_1)=(H_u(S^nu_1)\vert H_u(u_1))\\
&=&((S^*)^nH_u(u_1)\vert H_u(u_1))=(H_u(u_1)\vert S^nH_u(u_1))=(\vert H_u(u_1)\vert ^2\vert z^n)\ .
\end{eqnarray*}
Since $\vert u_1\vert ^2$ and $\vert H_u(u_1)\vert ^2$ are real valued, this proves the claim.

We  thus define
$$\Psi := \frac{u_1}{H_u(u_1)}\ .$$
\subsubsection{The function $\Psi$ is an inner function.}
We know that $\Psi $ is of modulus $1$ at every point of $\T $. Let  us show that $\Psi $ is in fact an inner function. By part 
(1) of the Adamyan--Arov--Krein theorem in Appendix C, we already know that $\Psi $ is a rational function with no poles on the unit circle.
Therefore, it is enough to prove that $\Psi $ has no pole in the open unit disc. 
Assume that $q\in \overline \D$ is a zero of $H_u(u_1)$, 
and let us show that $q$ is a zero of $u_1$ with at least the same multiplicity. 

Denote by $(e_1,\dots ,e_m)$  an orthonormal basis of $E$, such that 
$$H_u(e_j)=e_j\ ,\ j=1,\dots ,m\ .$$
Such a basis always exists, in view of the antilinearity of $H_u$.
Since, for every $f\in F$,  $(u\vert f)=0$, $(H_u(f)\vert 1)=0$, hence 
\begin{equation}\label{HSK}
H_u(f)=SK_u(f)\ .
\end{equation}
Assume $H_u(u_1)(q)=0$ and consider $$f:=\sum_{j=1}^m e_j(q)e_j \ .$$ Then 
$$(f\vert u_1)=\sum_{j=1}^m e_j(q)(e_j\vert u_1)=\sum_{j=1}^m H_u(e_j)(q)(e_j\vert u_1)=H_u(u_1)(q)=0\ ,$$
therefore $f$ belongs to $E\cap u_1^\perp=F$ as well as $K_u(f)$. Hence, by (\ref{HSK}), $$H_u K_u(f)=SK_u^2(f)=S(f)$$ from which we get
$$K_u(f)=\sum_{j=1}^m(K_u(f)\vert e_j) e_j=\sum_{j=1}^m(K_u(f)\vert H_u(e_j)) e_j=\sum_{j=1}^m(e_j\vert Sf) e_j\ , $$
hence $$K_u(f)(q)=(f\vert Sf) \ .$$
Therefore, using again  (\ref{HSK}),
$$\Vert f\Vert ^2=H_u(f)(q)=qK_u(f)(q)=q(f\vert Sf)\ .$$
Since $\Vert Sf\Vert =\Vert f\Vert $ and $\vert q \vert < 1$, we infer  $f=0$, hence 
$$e_j(q)=0 \ ,\ j=1,\dots ,m\ ,$$
in particular $u_1(q)=0$. 
\s
Assume now that $q$ is a zero of order $r$ of $H_u(u_1)$, so that, for every $a\leq r-1$,
$$(H_u(u_1))^{(a)} (q)=\sum_{j=1}^m (u_1\vert e_j)e_j^{(a)}(q)=0\ .$$
Let us prove by induction that
$$e_j^{(a)} (q)=0\ ,\ a\le r-1\ ,\ j=1,\dots ,m\ .$$
Assuming we have this property for $a<r-1$, we consider 
$$f:=\sum_{j=1}^m e_j^{(r-1)}(q) e_j\ .$$
As before, $f$ belongs to $F$ as well as $K_u(f)$ so that, as above
$$ (H_u(f))^{(r-1)} (q)= \Vert f\Vert ^2\  ,\ (K_u(f))^{(r-1)} (q)= (f\vert Sf)\ .$$
We then derive $r-1$ times identity (\ref{HSK}) at $z=q$ and use the induction hypothesis. We obtain
$$(H_u(f))^{(r-1)} (q)=q(K_u(f))^{(r-1)} (q)\ ,$$
hence $\Vert f\Vert ^2=q(f\vert Sf)$, and we conclude as before.\\

\subsubsection{The function $\Psi $ is a Blaschke product of degree $m-1$, $m=\dim E$}
In fact this is a consequence of what we have just done, and of part (1) of the Adamyan--Arov--Krein theorem in Appendix C. However it is useful to give another proof. We start with proving the following lemma.
\begin{lemma}\label{crucialHuGeneral}
Let $f\in \mathbb H^\infty (\D )$ such that $\Pi (\Psi \overline f)=\Psi \overline f\ .$
Then
$$H_u(fH_u(u_1))=\Psi \overline fH_u(u_1)\ .$$
\end{lemma}
The proof of the lemma is straightforward,
\begin{eqnarray*}
H_u(fH_u(u_1))&=&\Pi (u\overline f \overline {H_u(u_1)})=\Pi (\overline fH_u^2(u_1))=\Pi (\overline fu_1)=\Pi (\overline f \Psi H_u(u_1))\\
&=&\Psi \overline fH_u(u_1)\ .
\end{eqnarray*}
As a first consequence of the lemma, we observe that, if $\Psi =\Psi _a\Psi _b$, where $\Psi _a,\Psi _b$ are inner functions, then 
$$H_u(\Psi _aH_u(u_1))=\Psi _bH_u(u_1)\ .$$
In particular, $\Psi _a H_u(u_1)$ belongs to $E$, and the number of inner divisors of $\Psi $ is at most equal to the dimension of $E$.
Thus $\Psi $ is a Blaschke product of degree at most $m-1$. 

We now show that $\Psi\in \mathcal B_{m-1}$.  Write 
$$\Psi(z) ={\rm e}^{-i\psi}\frac{z^k\overline D\left(\frac 1z\right)}{D(z)}$$ where $D$ is a normalized polynomial of degree $k$. 
Using again the lemma, we have, for any $0\le a\le k$,
 $$H_u\left(\frac{z^a }{D}H_u(u_1)\right)={\rm e}^{-i\psi}\frac{z^{k-a}}{D}H_u(u_1)\ .$$
Let us set
$$ V:={\rm span}\left(\frac{z^a}{D}H_u(u_1), \ 0\le a\le k\right)\ .$$
Notice that 
$$V\subset E\ ,\ H_u(V)=V\ .$$
Since $\dim V=k+1$, this imposes $k\le m-1$. In order to prove $k=m-1$, we introduce
$$G:=V^\perp \cap E\ .$$
The proof will be complete if we establish that $G=\{ 0\} $. It is enough to  prove that $G\subset 1^\perp $ and that $S^*(G)\subset G$ (see the argument in the proof of Proposition \ref{rigidity}).
\s
Since $H_u(V)=V$, then $H_u(G)=G$. On the other hand, as $u_1=\Psi H_u(u_1)\in V$, $G\subset u_1^\perp \cap E\subset u^\perp $, hence $H_u(G)\subset 1^\perp $. This proves the first fact.  Remark also that, since $K_u^2=H^2_u-(\cdot\vert u)u$, one gets that $K_u^2=H_u^2$ on $G$ and $G\subset F$.

As for the second fact,  it is enough to prove that $H_u(G)\subset S(G)$. Let $g\in G$. By (\ref{HSK}), since $g\in F$, $SK_u(g)=H_u(g)$ so it suffices to prove that $K_u(g)$ belongs to $G$. We set
$$v_a:=\frac{z^a}{D}H_u(u_1), \ 0\le a\le k\ ,$$ and we prove that $(K_u(g)\vert v_a)=0$ for $0\le a\le k$.\\
For $1\le l\le k$, we write 
$$0=(H_u(g)\vert v_l)=(SK_u(g)\vert v_l)=(K_u(g)\vert S^*v_l)=(K_u(g)\vert v_{l-1}) \ .$$
For the scalar product with $v_k$ we remark that $v_k$ is a linear combination of the $v_j$'s, $0\le j\le k-1$, and of $u_1=\Psi H_u(u_1)$. As $K_u(g)\in F$, $(K_u(g)\vert u_1)=0$ hence finally
$(K_u(g)\vert v_k)=0\ .$

This proves that $\Psi$ is of degree $m-1$, that $\displaystyle E=\frac{\C_{m-1}[z]}DH_u(u_1)$ and that the action of $H_u$ on $E$ is as expected in Equation (\ref{actionHu}). It remains to prove that $$F=\frac{\C_{m-2}[z]}DH_u(u_1) $$ and that the action of $K_u$ is described as in (\ref{actionHu}).
We have, for $0\le b\le m-2$,
\begin{eqnarray*}
K_u\left( \frac{z^b}D H_u(u_1)\right)&=&H_uS\left( \frac{z^b}D H_u(u_1)\right)= H_u\left( \frac{z^{b+1}}D H_u(u_1)\right)\\
&=&{\rm e}^{-i\psi}\frac{z^{m-2-b}}{D} H_u(u_1)
\end{eqnarray*}
In particular, it proves that $\frac{\C_{m-2}[z]}DH_u(u_1)\subset F$. As the dimension of $F$ is $m-1$ by assumption, we get the equality.

\subsection{Case of $\sigma\in \Sigma_K(u)$}
The second part of the proposition, concerning the case of $\sigma\in \Sigma_K(u)$, can be proved similarly. We just give the main lines of the argument. As before, we assume that $\sigma=1$ for simplicity and denote by $u'_1$ the function $u'_\sigma$. The first step is to prove that $$\frac{K_u(u'_1)}{u'_1}$$ is an inner function. The same argument as the one used above proved that it has modulus one. To prove that it is an inner function, we argue as before. Namely, using again part (1) of the Adamyan-)Arov--Krein theorem in Appendix C, for $S^*u$ in place of $u$, we prove that if $u'_1$ vanishes at some $q\in \D$, $K_u(u'_1)$ also vanishes at $q$ at the same order. 
 We introduce an orthonormal basis  $\{f_1,\dots, f_{\ell }\}$ of $F:=F_u(1)$ such that $$K_u(f_j)=f_j, \;j=1,\dots,m.$$
Assume $u'_1(q)=0$ and consider $$e:=\sum_{k=1}^\ell  \overline{f_k(q)}f_k\ .$$ Let us prove that   $e=0$.
Observe first that $e$ belongs to $E:=E_u(1)$ since
$$(u'_1\vert e)=\sum_{k=1}^\ell f_k(q)(u'_1\vert f_k)=\sum_{k=1}^\ell  K_u^2(f_k)(q)(u'_1\vert f_k)=K_u^2(u'_1)(q)=u'_1(q)=0 \ .$$
We infer that $H_u(e)\in E$ as well, and therefore
$$(e\vert 1)=(H_u(H_u(e))\vert 1)=(u\vert H_u(e))=0\ ,$$
which implies
$$e=SS^*e=SS^*H_u^2(e)=SK_uH_u(e)\ .$$
Consequently,
\begin{eqnarray*}
\Vert e\Vert ^2&=&e(q)=qK_uH_u(e)(q)=q\sum _{k=1}^\ell (K_uH_u(e)\vert f_k)f_k(q)\\
&=&q\sum _{k=1}^\ell (f_k\vert H_u(e))f_k(q)=q(K_u(e)\vert H_u(e))=
q(H_u(e)\vert SH_u(e))\ .
\end{eqnarray*}
Since $\Vert H_u(e)\Vert =\Vert e\Vert $, we conclude as before that $SH_u(e)=qH_u(e)$ and finally $H_u(e)=0=e$.
One proves as well that if $q$ is a zero of order $r$ of $(u'_1)$, each $f_k$, $1\le k\le \ell$, vanishes at $q$ with the same order.
\s

\s
We now come to the third part of the proof to get $$\Psi:=\frac{K_u(u'_1)}{u'_1}\in\mathcal B_{\ell-1}.$$ 
We start with a lemma analogous to Lemma \ref{crucialHuGeneral}.
\begin{lemma}\label{crucialKuGeneral}
Let $f\in \mathbb H^\infty (\D )$ such that $\Pi (\Psi \overline f)=\Psi \overline f$.Then
$$K_u(fu'_1)=\Psi \overline fu'_1\ .$$
\end{lemma}
The proof of the lemma is similar to the one of Lemma \ref{crucialHuGeneral}.
In particular, for every inner divisor $\Psi _a$ of $\Psi $, $\Psi _au'_1$ belongs to $F$, and therefore the number of inner divisors of $\Psi $ is at most the dimension $\ell $ of $F$. In order to prove the equality, write
$$\Psi ={\rm e}^{-i\psi} \frac{z^k\overline D\left(\frac 1z\right)}{D(z)},$$ where $D$ is some normalized polynomial of degree $k$.
From the above lemma, for $0\le a\le k$,
\begin{eqnarray*}
K_u\left(\frac{z^a}{D}u'_1\right)={\rm e}^{-i\psi}\frac{z^{k-a}}Du'_1\ .
\end{eqnarray*}
Let us set
$$W:={\rm span}\left (\frac{z^a}Du'_1\ , 0\le a\le k\right)\ ,$$
so that  $$W\subset F\ ,\ K_u(W)=W\ .$$
To prove $k=\ell-1$, we introduce as before
$$H:=W^\perp \cap F\ $$ and we prove that $H=\{ 0\} $ by proving that $H_u(H)\subset 1^\perp $ and that $S^*(H_u(H))\subset H_u(H)$. It would imply $H_u(H)=\{0\}$ hence $H=\{0\}$ since $H$ is a subset of the range of $H_u$ by assumption.
\s 
First, remark that  $H\subset u^\perp $  since $H\subset {u'_1}^\perp$ as $u'_1\in W$, hence $H_u(H)\subset 1^\perp$.

 For the second fact,  take $h\in H$ and write $S^*H_u(h)=K_u(h)=H_u(S(h))$ so, it suffices to prove that $S(h)$ belongs to $H$. 
 Let us first prove that $S(h)$ belongs to $E$. By (\ref{HSK}), since $K_u(h)$ belongs to $H$, one has
 $$H_u^2(Sh)=H_u(K_u(h))=SK_u^2(h)=Sh\ .$$
 It remains to prove that $Sh\in W^\perp$. 

Let $\displaystyle w_j:=\frac{z^j}D u'_1$, $0\le j\le k$. For $1\le j\le k$, we have $$(Sh\vert w_j)=(h\vert S^*w_j)=(h\vert w_{j-1})=0.$$ It remains to prove that $(Sh\vert w_0)=0$. It is easy to check that $w_0$ is a linear combination of  the $w_j$'s, $1\le j\le k$ and of $u'_1$. As $S(h)$ belongs to $H$, $(S(h)\vert u'_1)=0$ hence $(h\vert w_0)=0$.

In order to complete the proof, we just need to describe $E$ as the subspace of $F$ made with functions which vanish at $z=0$, or equivalently are orthogonal to $1$. We already know that vectors of $E$ are orthogonal to $u$, and that $H_u$ is a bijection from $E$ onto $E$. We infer that vectors of $E$ are orthogonal to $1$. A dimension argument allows to conclude.

\section{The inverse spectral theorem in the finite rank case} \label{finiterank}

In this section, we prove Theorem \ref{Phibij}, Theorem \ref{Phitopodif}, and Theorem \ref{realmain} in the case of finite rank Hankel operators. 
Let $u$ be such that $H_u$ has finite rank. Then the sets $\Sigma _H(u)$ and $\Sigma _K(u)$ are finite. We set
$$q:=\vert \Sigma _H(u)\vert =\vert \Sigma _K(u)\vert \ .$$
If 
\begin{eqnarray*}
\Sigma _H(u):=\{ \rho _j, j=1,\dots ,q\},\  \rho _1>\dots >\rho _q>0\ ,\\
\ \Sigma _K(u):=\{ \sigma _j, j=1,\dots ,q\},\  \sigma _1>\dots >\sigma  _q\ge 0\ ,
\end{eqnarray*}
we know from (\ref{eqsigma}) that
\begin{equation}\label{Sigmaordered}
\rho _1>\sigma _1>\rho _2>\sigma _2>\dots >\rho _q>\sigma _q\ge 0\ .
\end{equation}
We set $n:=2q$ if $\sigma _q>0$ and $n:=2q-1$ if $\sigma _q=0$. For $2j\le n$, we set
$$s_{2j-1}:=\rho _j\ ,\ s_{2j}:=\sigma _j\ ,$$
so that the positive elements in the list (\ref{Sigmaordered}) read
\begin{equation}\label{ineqs} s_1>s_2>\dots >s_n>0\ .\end{equation}
Recall that we denote by $\mathcal U_n$ the set of symbols $u$ such the number of non zero elements of $\Sigma _H(u)\cup \Sigma _K(u)$ is exactly $n$, and that $\Omega _n$ denotes the open subset of $\R ^n$ defined by inequalities (\ref{ineqs}).
Using  Proposition \ref{action}, we define $n$ finite Blaschke products $\Psi _1,\dots ,\Psi _n$ by
$$\rho _ju_j=\Psi _{2j-1}H_u(u_j)\ ,\ K_u(u'_j)=\sigma _j\Psi _{2j}u'_j\ ,\ 2j\le n\ ,$$
where $u_j$ denotes the orthogonal projection of $u$ onto $E_u(\rho _j)$, and $u'_j$ denotes the orthogonal projection of $u$ onto $F_u(\sigma _j)$.
Our goal in this section is to prove the following statement.
\begin{theorem}\label{mainfiniterank} The mapping
$$\begin{array}{rcl}\Phi_n:\mathcal U_n&\longrightarrow & \mathcal S_n=\Omega _n\times \mathcal B^n\\
u&\longmapsto& \left((s_r)_{1\le r\le n},(\Psi_r)_{1\le r\le n}\right)\end{array}$$ is a homeomorphism.
\end{theorem}
\begin{proof}
The proof of Theorem \ref{mainfiniterank} involves several steps. Firstly,  we prove the continuity of $\Phi _n$, and we prove that, for  $r=1,\dots ,n$, the degree of $\Psi _r$ is locally constant. We then consider, for each $n$-uple $(d_1,\dots ,d_n)$ of nonnegative integers, the open set of $\mathcal U_n$
$$\mathcal V_{(d_1,\dots ,d_n)}:=\Phi _n^{-1}(\Omega _n\times \mathcal B_{d_1}\times \dots \times \mathcal B_{d_n})\ ,$$
and we just have to prove that $\Phi _n$ is a homeomorphism from $\mathcal V_{(d_1,\dots ,d_n)}$ onto $\Omega _n\times \mathcal B_{d_1}\times \dots \times \mathcal B_{d_n}$.

 We first prove this fact in the case $n$ even, along the following lines :
\begin{itemize}
\item $\Phi _n$  is injective, with an explicit formula for its left inverse.
\item $\Phi _n$ is an open mapping.
\item $\Phi _n$ is a proper mapping.
\item $\mathcal V_{(d_1,\dots ,d_n)}$ is not empty.
\end{itemize}
Since the target space $\Omega _n\times \mathcal B_{d_1}\times \dots \times \mathcal B_{d_n}$ is connected, these four items trivially lead to 
the result. The fourth item is proved by an induction argument on $\sum _rd_r$.  

Finally, the case $n$ odd is deduced from a simple limiting argument. 

As a complementary information, we prove that $\Phi _n^{-1}$ is a smooth embedding of the manifold $\Omega _n\times \mathcal B_{d_1}\times \dots \times \mathcal B_{d_n}$, which implies that $\mathcal V_{(d_1,\dots ,d_n)}$ is a manifold. 

\subsection{Continuity of $\Phi_n$}\label{continuity}

In this part, we prove that $\Phi_n$ is continuous from $\mathcal U_n$ into $\mathcal S_n$. 
Fix $u_0\in\mathcal U_n$. We prove that, in a neighborhood $V_0$ of  $u_0$ in $\mathcal U_n$, the degrees of the $\Psi _r$'s are constant.

Let $\rho\in \Sigma_H(u_0)$. The orthogonal projector $P_\rho$ on the eigenspace $E_{u_0}(\rho )$ is given by 
$$P_\rho=\int_{\mathcal C_\rho}(zI-H_{u_0}^2)^{-1} \frac{dz}{2i\pi}$$ where $\mathcal C_\rho$ is a circle, centered at $\rho^2$ whose radius is small enough so that the closed disc $\overline D_\rho $ delimited by  $\mathcal C_\rho$   is at positive distance to the rest of the spectrum  of $H_{u_0}^2$. For $u$ in a neighborhood $V_0$  of $u_0$ in $VMO_+$, $C_\rho $ does not meet the spectrum of $H_u^2$, and one may consider
$$P_\rho^{(u)}:=\int_{\mathcal C_\rho}(zI-H_{u}^2)^{-1} \frac{dz}{2i\pi}$$ which is  a finite rank orthogonal projector smoothly dependent on  $u$. Hence, $P_\rho^{(u)}(u) $ is well defined and smooth. Since this vector is not zero for $u=u_0$, it is still not zero for every $u$ in $V_0$. 
This implies in particular that $\Sigma _H(u)$ meets the open disc $D_\rho $.

We can do the same construction with any  $\sigma\in \Sigma_K(u_0)\setminus \{ 0\} $. We have therefore constructed $n$ smooth functions $u\in V_0\mapsto P_r^{(u)}\ ,\ r=1,\dots ,n,$ valued in the finite orthogonal projectors, and satisfying 
$$P_r^{(u)}(u)\ne 0\ ,\ r=1,\dots ,n\ .$$
Moreover, by continuity,
$${\rm rk} P_r^{(u)}={\rm rk}P_r^{(u_0)}:=d_r+1\ .$$
If we assume moreover that $u\in \mathcal U_n$, we conclude that $\Sigma _H(u)$ has exactly one element in each $D_{s _{2j-1}}$, and that
$\Sigma _K(u)$ has exactly one element in each $D_{s_{2k}}$, and that the dimensions of the corresponding eigenspaces are independent of $u$,
hence equal to $d_r+1$. In other words, the degrees of the corresponding Blaschke products are $d_r$. In other words, $V_0\cap \mathcal U_n$ is contained into  $\mathcal V_{(d_1,\dots ,d_n)}$.

Since, for every $u\in \mathcal V_{(d_1,\dots ,d_n)}$, we have
\begin{eqnarray*}
\vert \Sigma _H(u)\vert &=&\sum _{2j-1\le n} (d_{2j-1} +1)+\sum _{2k\le n} d_{2k}\ ,\\
\vert \Sigma _K(u)\setminus \{ 0\} \vert &=&\sum _{2j-1\le n} 
d_{2j-1} +\sum _{2k\le n} (d_{2k}+1)\ ,
\end{eqnarray*}
we conclude that
$${\rm rk} H_u=\left [\frac{d+1}2\right ]\ ,\ {\rm rk} K_u=\left [\frac d2\right ]\ ,\ d:=2\sum _{r=1}^n d_r + n\ ,$$
namely that $u\in \mathcal V(d)$, with the notation of  \cite{GG3}, \cite{GG4}. Recall   that $\mathcal V(d)$ is a complex manifold of dimension $d$. We then define a  map $\tilde \Phi _n$ on $ V_0\cap \mathcal V(d)$ by setting
$$\tilde \Phi _n(u)= ((s_r(u))_{1\le r\le n} ; (\Psi _r(u))_{1\le r\le n})\ ,$$
with
\begin{eqnarray*}
s_{2j-1}(u):=\frac{\Vert H_u(P^{(u)}_{2j-1}(u))\Vert }{\Vert P^{(u)}_{2j-1}(u)\Vert }\  &,& \   s_{2k}(u):=\frac{\Vert K_u(P^{(u)}_{2k}(u))\Vert }{\Vert P^{(u)}_{2k}(u)\Vert }\ ,\\
\Psi _{2j-1}(u):=\frac{s_{2j-1}(u)P^{(u)}_{2j-1}(u)}{H_u(P^{(u)}_{2j-1}(u))}\ &,&\ \Psi _{2k}(u)=\frac{K_u(P^{(u))}_{2k}(u)}{s_{2k}(u)P^{(u)}_{2k}(u)}\ .
\end{eqnarray*}
The mapping $\tilde \Phi _n$ is smooth from $\mathcal V (d)$ into $\Omega _n\times \mathcal R_d^n $, where $\mathcal R_d$ denotes the manifold of rational functions with  numerators and denominators of degree at most $\left [\frac{d+1}2\right ]$. Moreover, the restriction of $\tilde \Phi _n$ to $V_0\cap \mathcal V_{(d_1,\dots, d_n)} $ coincides with $\Phi _n$. This proves in particular that $\Phi _n$ is continuous. For future reference, let us state more precisely what we have proved.
\begin{lemma}\label{Phismooth}
For every $u_0\in \mathcal V_{(d_1,\dots ,d_n)}$, there exists a neighborhood $V$ of $u_0$ in $\mathcal V(d)$, $d=n+2\sum _{r=1}^n d_r\ $, and a smooth mapping $\tilde \Phi _n$ from this neighborhood into some manifold, such that the restriction of $\tilde \Phi _n$  to $V\cap \mathcal V_{(d_1,\dots,d_n)}$ coincides with $\Phi _n$.
\end{lemma}
\subsection{The explicit formula, case $n$ even.}\label{section explicit}
Assume that $n=2q$ is an even integer.\\
The fact that the mapping $\Phi_n$ is one-to-one follows from an explicit formula giving $u$ in terms of $\Phi_n(u)$, which we  establish
in this subsection.

We use the expected description of elements of $\Phi^{-1}(\mathcal S_n)$ suggested by the action of $H_u, K_u$ onto the orthogonal projections $u_j,u'_k$  of $u$ onto the 
corresponding eigenspaces of $H_u^2, K_u^2$ respectively, namely
\begin{equation}\label{crucial}
\rho _ju_j=\Psi_{2j-1}H_u(u_j)\ ,\ K_u(u'_k)=\sigma _k\Psi_{2k} u'_k\ ,\ j,k=1,\dots, q\ ,
\end{equation}
where the $\Psi_r$'s  are  Blaschke products.
\s
We then define $\tau _j, \kappa _k >0$ by 
\begin{equation}\label{J(x)}
\prod_{j=1}^q \frac{1-x\sigma_j^2}{1-x\rho_j^2}=1+x\sum_{j=1}^q \frac{\tau_j^2}{1-x\rho_j^2}
\end{equation}
\begin{equation}\label{1/J(x)}
\prod_{j=1}^q \frac{1-x\rho_j^2}{1-x\sigma_j^2}=1-x\left (\sum_{j=1}^q \frac{\kappa_j^2}{1-x\sigma_j^2}\right )
\end{equation}
From Appendix A, we have
$$\Vert u_j\Vert ^2=\tau_j^2\ ,\ \Vert u_k'\Vert ^2=\kappa _k^2\ ,\ j,k=1,\dots ,q\ .$$
\s
 Applying the operator $S$ of multiplication by $z$ to the second set of equations in (\ref{crucial}), and using $SS^*=I-(\, .\, \vert 1)\ ,$ we get
$$H_u(u'_k)(z)=\sigma _kz\Psi_{2k}(z)u'_k(z)+\kappa _k^2\ .$$
We use the identities (\ref{urho}), (\ref{usigma}) in this setting
\begin{equation}\label{uj}
u_j=\tau _j^2\sum _{k=1}^q \frac 1{\rho _j^2-\sigma _k^2}u'_k\ ,
\end{equation}
\begin{equation}\label{uprimek}
u_k'=\kappa _k^2\sum _{j=1}^q \frac 1{\rho _j^2-\sigma _k^2}u_j\ ,
\end{equation}
and we introduce the new unknowns $h_1,\dots ,h_q$ defined by 
$$u_j=\Psi _{2j-1}h_j\ ,\ {\rm or}\ h_j:=\frac{1}{\rho _j}H_u(u_j)\ .$$
For the vector valued function
$$\mathcal H(z):=\left(h_j(z)\right)_{1\le j\le q}\ ,$$
we finally obtain the following linear system,
\begin{equation}\label{eqH}
\mathcal H(z)=\mathcal F(z)+\mathcal A(z)\mathcal H(z)\ ,
\end{equation}
where, thanks to equation (\ref{sommesimplekappa})
\begin{eqnarray*}
\mathcal F(z)&:=&\left (\frac{\tau _j^2}{\rho_j}\sum _{k=1}^q \frac {\kappa _k^2}{\rho _j^2-\sigma _k^2}\right )_{1\le j\le q}=\left (\frac{\tau _j^2}{\rho_j}\right )_{1\le j\le q}\ ,\\
\mathcal A(z)&:=&\left (\frac{\tau _j^2}{\rho_j}\sum _{k=1}^q \frac {\kappa _k^2\sigma _kz\Psi _{2k}(z)\Psi _{2\ell -1}(z)}{(\rho _j^2-\sigma _k^2)(\rho _\ell ^2-\sigma _k^2)}\right )_{1\le j,\ell  \le q}\ .
\end{eqnarray*}
Notice that the matrix $\mathcal A (z)$ depends holomorphically on $z\in \D$   and satisfies $\mathcal A(0)=0$. Hence $I-\mathcal A(z)$ is invertible at least for $z$ in a neighborhood of $0$, which characterizes $\mathcal H(z)$, hence characterizes
$$u(z)=\sum _{j=1}^q \Psi _{2j-1}(z)h_j(z)\ .$$
This is enough for proving the injectivity of $\Phi_n$. However, we are going to transform the expression of $\mathcal H(z)$ into a simpler one, which will be very useful in the sequel.

\s
Introduce the matrix $\mathcal B=(b_{jk})_{1\le j,k\le q}$ defined by
$$b_{jk}:=\frac{\kappa _k^2}{\rho _j^2-\sigma _k^2}\ .$$
From the identities (\ref{sommedoublekappa}) and (\ref{sommedoubletau}) in Appendix A, we know that $\mathcal B$ is invertible, with
$$\mathcal B^{-1}=\left (\frac{\tau _j^2}{\rho _j^2-\sigma _k^2}\right )_{1\le k,j\le q}\ .$$
In view of these identities, we observe that
$$I-\mathcal A(z)={\rm diag}\left (\frac{\tau _j^2}{\rho _j}\right )\mathcal B\mathcal C(z)\ ,$$
where $\mathcal C(z)=(c_{k\ell }(z))_{1\le k,\ell \le q}$  is defined by
\begin{equation}\label{c}
c_{k\ell }(z):=\frac{\rho _\ell -\sigma _k z\Psi _{2k}(z)\Psi _{2\ell -1}(z)}{\rho _\ell ^2-\sigma _k^2}\ .
\end{equation}
Consequently, Equation (\ref{eqH}) above reads
$${\rm diag}\left (\frac{\tau _j^2}{\rho _j}\right )\mathcal B\mathcal C(z)\mathcal H(z)=\mathcal F(z)={\rm diag}\left (\frac{\tau _j^2}{\rho _j}\right )\mathcal B ({\bf 1})\ ,$$
where 
$${\bf 1}:=\left (\begin{array}{l} 1 \\ .\\ .\\ .\\ 1\end{array}\right )\ .$$
Notice that we again used (\ref{sommesimplekappa}) under the form $\mathcal B({\bf 1})={\bf 1}\ .$ Finally, equation (\ref{eqH})  is equivalent to
\begin{equation}\label{CH}
\mathcal C(z)\mathcal H(z)={\bf 1}\ .
\end{equation}
Using the Cramer formulae, we get
$$h_j(z)=\frac{\sum _{k=1}^q (-1)^{j+k}\Delta _{kj}(z)}{{\det}(\mathcal C (z))}\ ,$$
where $\Delta _{kj}(z)$ is the minor determinant of $\mathcal C(z)$ corresponding to line $k$ and column $j$.
This provides formula (\ref{luminy}) of Theorem \ref{Phibij}.
\s
For future reference, we shall rewrite the above formula in a slightly different manner. Recall that
\begin{equation}\label{PsiPD}
\Psi _r(z)={\rm e}^{-i\psi _r}\frac{P_r(z)}{D_r(z)}\ ,\ D_r(z):=z^{d_r}\overline P_r\left (\frac 1z\right )\ ,
\end{equation}
where $P_r$ is a monic polynomial of degree $d_r$. Introduce the matrix $\mathcal C^\# (z)=(c^\# _{k\ell }(z))_{1\le k,\ell \le q}$ 
as
\begin{equation}\label{cdiese}
c^\# _{k\ell }(z)=\frac{\rho _\ell D_{2k}(z)D_{2\ell -1}(z)-\sigma _kz{\rm e}^{-i(\psi _{2k}+\psi _{2\ell -1})} P_{2k}(z)P_{2\ell -1}(z)}{\rho _\ell ^2-\sigma _k^2}\ ,
\end{equation}
 denote by $Q(z)$ its determinant and by $\Delta _{k\ell }^\# (z)$ the corresponding minor determinant. Then 
\begin{equation}\label{hDR}
 h_j(z)=D_{2j-1}(z)R_{2j-1}(z)\ ,
\end{equation}
with 
\begin{equation}\label{Rimpair}
R_{2j-1}(z):=\frac{\sum _{k=1}^q(-1)^{k+j}D_{2k}(z)\Delta _{kj}^\#(z)}{Q(z)}\ .
\end{equation}
Notice that $Q$ is a polynomial of degree at most 
$$N:=q+\sum _{r=1}^n d_r\ ,$$
and the numerator of $R_{2j-1}$ is a polynomial of degree at most $N-1-d_{2j-1}$.
Consequently,
$$u(z)=\sum _{j=1}^q {\rm e}^{-i\psi _{2j-1}}P_{2j-1}(z)R_{2j-1}(z)\ ,$$
is a rational function with denominator $Q$ and with a numerator of degree at most $N-1$. 
Since the rank of $H_u$ is exactly $N$, we infer that the degree of $Q$ is exactly $N$, and that $Q$ has no zero in the closed unit disc.
Indeed, otherwise the numerator of $u$  would have the same zero in order to preserve the analyticity, and, by simplification, $u$ could be written as a quotient of polynomials  of degrees smaller than $N-1$ and $N$ respectively, so that the rank of $H_u$ would be smaller.
\s
We close this section by giving similar formulae for $u'_k, k=1,\dots ,q$. The main ingredient is the following algebraic lemma.
\begin{lemma}\label{BCPsi}
For every $z\in \overline \D$,
$$^t\mathcal C(z)\ ^t\mathcal B\ {\rm diag}(\Psi _{2\ell -1}(z))_{1\le \ell \le q}={\rm diag}(\Psi _{2j-1}(z))_{1\le j\le q}\ \mathcal B\ \mathcal C(z)\ .$$
\end{lemma}
The proof of this lemma is straightforward, using identity (\ref{sommedoublekappa}). As a consequence  of this lemma and of the identities
(\ref{uprimek}), (\ref{CH}), we infer that $\mathcal U'(z):=(u'_k(z))_{1\le k\le q}$ satisfies
\begin{equation}\label{CUprime}
^t\mathcal C(z)\mathcal U'(z) =(\Psi _{2j-1}(z))_{1\le j\le q}\ .
\end{equation}
Using Cramer's formulae, we infer
\begin{equation}\label{u'DR}
u'_k(z)=\frac{\sum _{j=1}^q (-1)^{j+k}\Delta _{kj}(z)\Psi _{2j-1}(z)}{{\det}(\mathcal C (z))}=D_{2k}(z)R_{2k}(z)\ ,
\end{equation}
where
\begin{equation}\label{Rpair}
R_{2k}(z)=\frac{\sum _{j=1}^q (-1)^{k+j}{\rm e}^{-i\psi _{2j-1}} P_{2j-1}(z)\Delta _{kj}^\# (z)}{Q(z)}\ .
\end{equation}
Notice that the numerator of $R_{2k}$ has degree at most $N-1-d_{2k}$. Moreover, (\ref{uprimek}) now reads
\begin{equation}\label{DRPR}
D_{2k}R_{2k}=\sum _{j=1}^q\frac{\kappa _k^2}{\rho _j^2-\sigma _k^2}{\rm e}^{-i\psi _{2j-1}} P_{2j-1}R_{2j-1} \ .
\end{equation}
\subsection{Surjectivity in the case $n$ even.} 

Our purpose is now to prove that the mapping $\Phi_n$ is onto. Since we got a candidate from the formula giving $u$
in the latter section, it may seem natural to try to check that this formula indeed provides an element $u$ of $\mathcal V_{(d_1,\dots ,d_n)}$
with the required $\Phi_n(u)$. However, in view of the complexity of the formulae (\ref{hDR}), (\ref{u'DR}), it seems difficult to infer from them the spectral properties of $H_u$ and $K_u$. We shall therefore use an indirect method,  by  proving that the mapping $\Phi _n$ on $\mathcal V_{(d_1,\dots ,d_n)}$ is open, closed, and that the source space $\mathcal V_{(d_1,\dots ,d_n)}$ is not empty. Since the target space $\Omega _n \times \prod _{j=1}^n\mathcal B_{d_j}$ is clearly connected, this will imply the surjectivity. A first step in proving that $\Phi _n$ is an open mapping, consists in the construction of an anti-linear operator $H$ satisfying the
required  spectral properties, and which will be finally identified as $H_u$. 
\subsubsection{Construction of the operator $H$}
\noindent Let $$\mathcal P=((s_r)_{1\le r\le n}, (\Psi_r)_{1\le r\le n})$$
be an arbitrary element of  $$\mathcal P\in \Omega _{n}\times \prod_{j=1}^n\mathcal B_{d_j}\text{ for some non negative integers } d_r\ .$$ We look for $u\in \mathcal V_{(d_1,\dots, d_n)}$,
$\Phi_n(u)=\mathcal P$. We set $\rho _j:=s_{2j-1}\ ,\ \sigma _k:=s_{2k}\ ,\ 1\le j,k\le q\ .$
\s
Firstly, we define matrices $\mathcal C(z)$ and $\mathcal C^\# (z)$ using formulae (\ref{c}), (\ref{PsiPD}), (\ref{cdiese}). We assume moreover  the following open properties,
\begin{equation}\label{hypo1}
\det \mathcal C(z)\ne 0\ ,\ z\in \overline \D\ ,\ {\rm deg}(Q)=N:=q+\sum _{r=1}^n d_r\ .
\end{equation}

We then define $R_r(z), r=1,\dots, n$ by formulae (\ref{Rimpair}) and (\ref{Rpair}). Setting $\mathcal H(z):=(D_{2j-1}(z)R_{2j-1}(z))_{1\le j\le q}$
and $\mathcal U'(z):=(D_{2k}(z)R_{2k}(z))_{1\le k\le q}$, this is equivalent to equations (\ref{CH}) and (\ref{CUprime}). Moreover, by Lemma \ref{BCPsi}, one checks that the column
$$\mathcal U''(z):=\left ( \sum _{j=1}^q\frac{\kappa _k^2}{\rho _j^2-\sigma _k^2}{\rm e}^{-i\psi _{2j-1}} P_{2j-1}(z)R_{2j-1}(z) \right )_{1\le k\le q}$$
satisfies 
$$^t\mathcal C(z)\mathcal U''(z)=(\Psi _{2j-1}(z))_{1\le j\le q}\ ,$$
and therefore $\mathcal U''=\mathcal U'$, which is (\ref{DRPR}). 
\s
We are going to define an antilinear operator on $W=\frac{\C_{N-1}[z]}{Q(z)}$. For this, we define the following vectors of $W$,
\begin{eqnarray*}
e_{2j-1, a}(z)&:=&z^aR_{2j-1}(z),\  0\le a\le d_{2j-1}\ ,\\ e_{2k,b}(z)&:=&z^bR_{2k}(z)\ ,\ 1\le b\le d_{2k}\ ,
\end{eqnarray*}
for $1\le j,k\le q$.
We need a second open assumption.
\begin{equation}\label{hypo2}
\mathcal E:=\left ((e_{2j-1,a})_{0\le a\le d_{2j-1}}\ ,\ (e_{2k,b})_{1\le b\le d_{2k}}\right )_{1\le j,k\le q} \text{ is a basis of } W.
\end{equation}
 We define an antilinear operator $H$ on $W$  by
\begin{eqnarray*}
H(e_{2j-1,a})&:=&\rho _j{\rm e}^{-i\psi _{2j-1}}e_{2j-1,d_{2j-1}-a}\ ,\ 0\le a\le d_{2j-1}\ ,\\
H(e_{2k,b})&:=&\sigma _k{\rm e}^{-i\psi _{2k}}
e_{2k,d_{2k}+1- b},\ 1\le b\le d_{2k}\ ,
\end{eqnarray*}
for $1\le j,k\le q$.
\s
From this definition, $H$ satisfies
\begin{eqnarray}
\label{actionH1} H(AR_{2j-1})&=&\rho_j{\rm e}^{-i\psi_{2j-1}}z^{d_{2j-1}}\overline A\left(\frac 1z\right)\\
\label{actionH2} H(zBR_{2k})&=&\sigma_k{\rm e}^{-i\psi_{2k}} z^{d_{2k}}\overline B\left(\frac 1z\right)\ .
\end{eqnarray}
for any $A\in \C_{d_{2j-1}}[z]$ and any $B\in \C_{d_{2k}-1}[z]$.
\subsubsection{Identifying $H$ and $H_u$.}
Notice that $W$ is invariant by $S^*$. The key of the proof of $H=H_u$ is the following lemma.
\begin{lemma}\label{S*H}
\begin{equation}\label{S*HS*}S^*HS^*=H-(1\vert \cdot)u \text{ on }W\ ,\end{equation}
where $$u:=\sum _{j=1}^q {\rm e}^{-i\psi _{2j-1}}P_{2j-1}R_{2j-1}\ .$$
\end{lemma}
\begin{proof}
We check the above identity on all the elements of the basis $\mathcal E$ of $W$. The only non trivial cases correspond to 
$e_{2j-1, 0}$ and $e_{2k,1}$. In other words, we have to prove
\begin{eqnarray}
\label{amontrer1} S^*HS^*(R_{2j-1})&=&H(R_{2j-1})-(1\vert R_{2j-1}) u\ ,\\
\label{amontrer2} S^*H(R_{2k})&=&H(SR_{2k})\  ,
\end{eqnarray}
for $1\le j,k\le q\ .$ We start with (\ref{amontrer2}). We set 
$$D_{2k}(z)=1+zF_{2k}(z)\ ,$$
so that, from (\ref{DRPR}),
$$R_{2k}= \sum _{j=1}^q\frac{\kappa _k^2}{\rho _j^2-\sigma _k^2}{\rm e}^{-i\psi _{2j-1}} P_{2j-1}R_{2j-1}-SF_{2k}R_{2k}\ .$$
Using (\ref{actionH1}), we infer
\begin{equation}\label{HRpair}
H(R_{2k})=\sum _{j=1}^q\frac{\kappa _k^2}{\rho _j^2-\sigma _k^2}\rho _jD_{2j-1}R_{2j-1} -\sigma _k{\rm e}^{-i\psi _{2k}}G_{2k}R_{2k}\ ,
\end{equation}
where
$$G_{2k}(z)=z^{d_{2k}}\overline F_{2k}\left (\frac 1z\right )=z(P_{2k}(z)-z^{d_{2k}})\ .$$
In view of equation (\ref{CH}), 
$$\sum _{j=1}^q\frac{\rho _j}{\rho _j^2-\sigma _k^2}D_{2j-1}(z)R_{2j-1}(z)=1+z\sum _{j=1}^q\frac{\sigma _k\Psi _{2k}(z)}{\rho _j^2-\sigma _k^2}
{\rm e}^{-i\psi _{2j-1}}P_{2j-1}(z)R_{2j-1}(z)\ .$$
Multiplying by $\kappa _k^2$ and applying $S^*$ to both sides, we obtain
$$
 \sum _{j=1}^q\frac{\kappa _k^2\rho _j}{\rho _j^2-\sigma _k^2}S^*(D_{2j-1}R_{2j-1})=\sum _{j=1}^q\frac{\kappa _k^2\sigma _k\Psi _{2k}}{\rho _j^2-\sigma _k^2}
{\rm e}^{-i\psi _{2j-1}}P_{2j-1}R_{2j-1}\ .$$ 
Using (\ref{DRPR}), we conclude 
\begin{equation}\label{S*DR}
 \sum _{j=1}^q\frac{\kappa _k^2\rho _j}{\rho _j^2-\sigma _k^2}S^*(D_{2j-1}R_{2j-1})=\sigma _k{\rm e}^{-i\psi _{2k}} P_{2k}R_{2k}\ .
\end{equation}
Coming back to (\ref{HRpair}), we conclude
$$S^*H(R_{2k})(z)=\sigma _k{\rm e}^{-i\psi _{2k}} z^{d_{2k}}R_{2k}(z)= H(SR_{2k}) (z)\ .$$
This proves (\ref{amontrer2}). 
\s
Let us establish (\ref{amontrer1}). Let us set $D_{2j-1}=1+SF_{2j-1}$, so that
\beno
S^*HS^*(R_{2j-1}) -H(R_{2j-1}) &=& \\S^*HS^*(D_{2j-1}R_{2j-1})&-&
\rho _j{\rm e}^{-i\psi _{2j-1}} (z^{d_{2j-1}}R_{2j-1}+S^*G_{2j-1}R_{2j-1})\ ,
\eeno
with
$$G_{2j-1}(z)=z^{d_{2j-1}}\overline F_{2j-1}\left (\frac 1z\right )=z(P_{2j-1}(z)-z^{d_{2j-1}})\ .$$
This yields
$$S^*HS^*(R_{2j-1}) -H(R_{2j-1})=S^*HS^*(D_{2j-1}R_{2j-1})- \rho _j{\rm e}^{-i\psi _{2j-1}} P_{2j-1}R_{2j-1}\ .$$
Using (\ref{S*DR}), we obtain
\begin{eqnarray*}
 T_k&:=&\sum _{j=1}^q\frac{\kappa _k^2\rho _j}{\rho _j^2-\sigma _k^2}\left(S^*HS^*(R_{2j-1})-H(R_{2j-1})\right)\\&=&\sigma _k{\rm e}^{i\psi _{2k}} S^*H( P_{2k}R_{2k})
- \sum _{j=1}^q\frac{\kappa _k^2\rho _j^2}{\rho _j^2-\sigma _k^2}{\rm e}^{-i\psi _{2j-1}} P_{2j-1}R_{2j-1}\ .
\end{eqnarray*}
At this stage, notice that, in view of (\ref{amontrer2}) and of (\ref{actionH2}), we have, for every $A\in \C _{d_{2k}}[z]$, 
$$S^*H(AR_{2k})=\sigma _k{\rm e}^{-i\psi _{2k}}BR_{2k}\ ,\ B(z):=z^{d_{2k}} \overline A\left (\frac 1z\right )\ .$$
Applying this formula to $A=P_{2k}$ and using (\ref{DRPR}), we finally get
$$T_k=-\kappa _k^2\sum _{j=1}^q {\rm e}^{-i\psi _{2j-1}} P_{2j-1}R_{2j-1}=-\kappa _k^2u\ .$$
On the other hand, using equation (\ref{CH}) at $z=0$, we have
$$\sum _{j=1}^q\frac{\rho _j}{\rho _j^2-\sigma _k^2}R_{2j-1}(0)=1\ .$$
Since the matrix $\mathcal B$ is invertible, we infer that $R_{2j-1}(0)\in \R $, hence equals $(1\vert R_{2j-1})$. In other words,
$$\sum _{j=1}^q\frac{\kappa _k^2\rho _j}{\rho _j^2-\sigma _k^2}(S^*HS^*(R_{2j-1})-H(R_{2j-1})-(1\vert R_{2j-1}) u)=0\ .$$
This completes the proof.
\end{proof}

We now prove that an operator satisfying equality (\ref{S*H}) is actually a Hankel operator. 

\begin{lemma}\label{coroS*HS*}
Let $N$ be a positive integer. Let $$Q(z):= 1-c_1z-c_2z^2-\dots-c_N z^N$$ be a complex valued polynomial with no roots in the closed unit disc. 
Set $$W:=\frac{\C_{N-1}[z]}{Q(z)}\subset L^2_+\ .$$
Let $H$ be an antilinear operator on $W$ satisfying
$$S^*HS^*=H-(1\vert \cdot)u$$ on $W$, for some $u\in W$.
Then $H$ co\"{\i}ncides with the Hankel operator of symbol $u$ on $W$.
\end{lemma}
\begin{proof}
Consider the operator $\tilde H:= H-H_u$, then $S^*\tilde HS^*=\tilde H$ on $W$ and hence, it suffices to show that, if $H$ is an antilinear operator on $W$ such that $S^*HS^*=H$, then $H=0$.\\

The family $(e_j)_{1\le j\le N}$ where $$e_0(z)=\frac 1{Q(z)}\ , \ e_j(z)=S^j e_0(z), \  j=1,\dots,N-1$$ is a basis of $W$. 
Using that $$S^*HS^*=H$$
we get  on the one hand that $He_k=(S^*)^k He_0$. On the other hand, since
$$S^*e_0=S^*\left (\frac 1{Q}\right )=\sum_{j=1}^N c_j  e_{j-1} \ ,$$
this implies  $$He_0=S^*HS^*e_0=\sum_{j=1}^N\overline{c_j} (S^*)^j He_0\ ,$$ hence $\overline Q (S^*)H(e_0)=0$ .
Observe that, by the spectral mapping theorem,  the spectrum of $\overline{Q}(S^*))$ is contained into
$\overline Q(\overline \D)$,  hence $\overline{Q}(S^*)$ is one-to-one. We conclude that $H(e_0)=0$, and finally that $H= 0$.

\end{proof}

Applying Lemma  \ref{coroS*HS*} to our vector space $W$,  we conclude  that $H=H_u$.
\s

It remains to check that $\Phi_n(u)=\mathcal P$. 
\subsubsection{The function $u$ has the required properties.}
Using the definition of $H=H_u$, we observe that the restriction of $H_u^2$ to the space $\C _{d_{2j-1}}[z]R_{2j-1}$ is $\rho _j^2I$.
Similarly, the restriction of $H_u^2$ to the space  $z\C _{d_{2k}}[z]R_{2k}$ is $\sigma _k^2I$. Since the range of $H_u$ is contained into $W$, this provides a complete diagonalization of $H_u^2$. Moreover,
$$u=\sum _{j=1}^q {\rm e}^{-i\psi _{2j-1}} P_{2j-1}R_{2j-1}\ .$$
This implies that
$$\Sigma _H(u)=\{ \rho _1,\dots ,\rho _q\} \ ,\ u_j={\rm e}^{-i\psi _{2j-1}} P_{2j-1}R_{2j-1} ,\ j=1,\dots ,q\ .$$
We argue similarly for $K_u^2$, noticing that
$$\sum _{k=1}^q D_{2k}R_{2k}= \sum _{1\le j,k\le q}\frac{\kappa _k^2}{\rho _j^2-\sigma _k^2}{\rm e}^{-i\psi _{2j-1}} P_{2j-1}R_{2j-1}\ ,$$
from (\ref{DRPR}), we conclude, using again (\ref{sommesimplekappa}), that
$$\sum _{k=1}^q D_{2k}R_{2k}=u\ .$$
This shows that
$$\Sigma _K(u)=\{ \sigma _1,\dots ,\sigma _q\} \ ,\ u_k'=D_{2k}R_{2k}\ ,\ j=1,\dots ,q\ .$$
Finally, from the definition of $H$, we recover exactly identities (\ref{crucial}).

\subsubsection{The mapping $\Phi _n$ is open from $\mathcal V_{(d_1,\dots ,d_n)}$ to $\Omega _n\times  \prod_{r=1}^n\mathcal B_{d_r}$.  }
Notice that we have not yet completed the proof of Theorem \ref{mainfiniterank} since the previous calculations were made under the assumptions
(\ref{hypo1}) and (\ref{hypo2}). In other words, we proved that an element $\mathcal P$ of the target space satisfying  (\ref{hypo1}) and (\ref{hypo2})
is in the range of $\Phi _n$. On the other hand, in section \ref{section explicit}, we proved that these properties are satisfied by the elements of the range of $\Phi _n$. Since these hypotheses are clearly open in the target space, we infer that the range of $\Phi _n$ is open.

\subsubsection{The mapping $\Phi _n$ is closed.}\label{closed}
Let $(u^\e )$ be a sequence of $\mathcal V_{(d_1,\dots,d_n)}$ such that
$\Phi_n(u^\e ):=\mathcal P^\e $ converges to some $\mathcal P$ in $\Omega _n\times  \prod_{r=1}^n\mathcal B_{d_r}$ as $\e $ goes to $0$.
In other words, 
$$\mathcal P^\e =\left((s_r^\e )_{1\le r\le 2q}, (\Psi_r^\e)_{1\le r\le  2q}\right)
 \longrightarrow \mathcal P=\left ((s_r)_{1\le r\le 2q}, (\Psi_r)_{1\le r\le 2q} \right) $$
in $\Omega_n\times \prod_{j=1}^n\mathcal B_{d_j}$ as $\e\to 0$. We have to find $u$ such that $\Phi(u)=\mathcal P$. Since 
$$\Vert u^\e \Vert _{H^{1/2}} \simeq {\rm Tr}(H_{u^\e }^2)=\sum_{r=1}^{2q} d_r (s_r^\e)^2+\sum _{j=1}^q (s _{2j-1}^\e )^2$$
is bounded, we may assume, up to extracting a subsequence,  that $u^\e $ is weakly convergent to some $u$ in $H^{1/2}$. Moreover,
the rank of $H_u$ is at most $N=q+\sum_{r=1}^{2q}d_r$.

Denote by $u^\e _j$ the orthogonal projection of $u^\e $ onto $\ker (H_u^2-(s _{2j-1}^\e)^2 I)$, $j=1,\dots ,q$, and by $(u^\e _k)' $  the orthogonal projection of $u^\e $ onto $\ker (K_u^2-(s _{2k}^\e)^2 I)$, $k=1,\dots , q$. Since all these functions are bounded in $L^2_+$, we may assume that, for the weak convergence in $L^2_+$,
$$ u^\e _j\rightharpoonup v_j\ ,\ (u^\e _k)' \rightharpoonup v'_k\ .$$
Taking advantage of the strong convergence of $u^\e $ in $L^2_+$ due to the Rellich theorem, we can pass to the limit in
$$(u^\e \vert u^\e _j)=(\tau ^\e _j)^2 \ ,\ (u^\e \vert (u^\e _k)')=(\kappa _k^\e)^2\ ,$$
and obtain, thanks to the explicit expressions (\ref{tau}), (\ref{kappa}) of $\tau_j^2, \kappa _k^2$ in terms of the $s_r$, 
$$(u\vert v_j)=\tau _j^2>0\ ,\ (u\vert v_k')=\kappa _k^2>0\ ,$$
in particular $v_j\ne 0, v_k'\ne 0$ for every $j,k$.

On the other hand, passing to the limit in
\begin{eqnarray*}
s ^\e _{2j-1}u^\e _{j}&=&\Psi_{2j}^\e\  H_{u^\e }u^\e _j\ ,\ H_{u^\e }^2(u^\e _j)=(s ^\e _{2j-1})^2u^\e _j\ ,\\
K_{u^\e }(u^\e _k)'&=&s ^\e _{2k}\Psi_{2k}^\e\ (u^\e _k)'\ ,\ K_{u^\e }^2(u^\e _k)'=(s ^\e _{2k})^2(u^\e _k)'\ ,\\
u^\e &=&\sum _{j=1}^q u^\e _j =\sum _{k=1}^q (u^\e_k)' \ ,
\end{eqnarray*}
we obtain
\begin{eqnarray*}
s  _{2j-1}v_j&=&\Psi_{2j}\  H_{u }v _j\ ,\ H_{u }^2(v _j)=s _{2j-1}^2v _j\ ,\\
K_{u }v_k'&=&s_{2k}\Psi_{2k} \ v _k' \ ,\ K_u^2(v'_k)=s_{2k}^2v'_k\ ,\\
u &=&\sum _{j=1}^q v _j =\sum _{k=1}^q v_k' \ .
\end{eqnarray*}
This implies that $u\in\mathcal V_{(d_1,\dots,d_n)}$, $v_j=u_j, v'_k=u'_k$, and $\Phi(u)=\mathcal P$. The proof of Theorem \ref{mainfiniterank} is thus 
complete in the case $n=2q$, under the assumption that $\mathcal V_{(d_1,\dots,d_n)}$ is non empty. 

\subsection{$\mathcal V_{(d_1,\dots,d_n)}$ is non empty, $n$ even}

Let $n$ be a positive even integer. The aim of this section is to prove that $\mathcal V _{(d_1,\dots,d_n)}$ is not empty for any multi-index $(d_1,\dots,d_n)$ of non 
negative integers.

The preceding section implies that, as soon as $\mathcal V _{(d_1,\dots,d_n)}$ is non empty, it is homeomorphic to $\Omega_n\times \prod_{j=1}^n\mathcal B_{d_j}$, via the explicit formula (\ref{luminy}) of Theorem \ref{Phibij}. We argue by induction on the integer  $d_1+\dots
+d_n$. In  the generic case consisting of simple eigenvalues (see \cite{GG1}), we proved that for any positive integer $q$, $\mathcal V _{(0,\dots,0)}(=\mathcal V_{\rm gen}(2q))$ is non empty. As a consequence, to any given sequence $((s_r),(\Psi_r))\in \Omega_{2q}\times \T^{2q}$ corresponds a unique $u\in \mathcal V_{(0,\dots,0)}$, the $s_{2j-1}^2$ being the simple eigenvalues of $H_u^2$ and the $s_{2j}^2$ the simple eigenvalues of $K_u^2$. This gives the theorem in the case $(d_1,\dots, d_n)=(0,\dots,0)$ for every $n$, which is one of the main theorems of \cite{GG2}.
Let us turn to the induction argument, which  is clearly a consequence of the following lemma.

\begin{lemma}\label{non empty}
Let $n=2q$, $(d_1,\dots,d_n)$ and $1\le r\le n-1$.
Assume $$\mathcal V_{(d_1,\dots,d_r,0,0, d_{r+1},\dots, d_n)}\text{ is non empty, }$$ then $$\mathcal V _{(d_1,\dots,d_{r-1},d_r+1,d_{r+1},\dots,d_n)}\text{ is non-empty.}$$
\end{lemma}

\begin{proof}
We  consider the case $r=1$.  The proof in the case $r$ odd follows the same lines. Write $m_j:=d_{2j-1}+1$ and $\ell_k=d_{2k}+1$. From the assumption, $$\mathcal V:=\mathcal V_{(d_1,0,0, d_{2},\dots, d_n)}\text{ is non empty, }$$ hence $\Phi$ establishes a diffeomorphism from $\mathcal V$ into $$\Omega_{n+2}\times\mathcal B_{d_1}\times \mathcal B_0\times \mathcal B_0\times \prod_{r=2}^{n}\mathcal B_{d_r}\ .$$ 
Therefore, given $\rh >\si _2>\rh _3>\si _3>\dots >\rh _{q+1}>\si _{q+1}>0$, and $\Psi _1,\theta _1,\f _2, \Psi _4, \dots ,\Psi _{n+2}$, for every $\eta >0$, for every $\e >0$ small enough,  we define
$u^\e$ to be the inverse image by $\Phi $ of
$$\left ((\rh +\e , \rh ,  \rh -\eta \e , \si _2,\rh _3,\si _3,\dots ,\rh _{q+1},\si _{q+1}), (\Psi _1, \expo_{-i\t _1}, \expo_{-i\f _2}, \Psi _4,\dots ,\Psi _{n+2})\right )\ .$$
By making $\e $ go to $0$, we are going to construct $u$ in $\mathcal V _{(d_1+1,\dots,d_n)}$, such that  $\rho_1(u)=\rh$ is of multiplicity $m_1+1=d_1+2$, $\rho_{j}(u)=\rho_{j+1}$, $j=2,\dots ,q,$ is of multiplicity $m_{j}$ and $\sigma_k(u)=\sigma_{k+1}$ for $k=1,\dots ,q$, is of multiplicity $\ell_{k}$.
\s
First of all, observe that $u^\e $ is bounded in $H^{1/2}_+$, since its norm is equivalent to ${\rm Tr}(H_{u^\e }^2)$. Hence, by the Rellich theorem, up to extracting a subsequence, $u^\e $ strongly converges in $L^2_+$ to some $u\in H^{1/2}_+$. Similarly, the orthogonal projections $u^\e _j$ and $(u^\e_k)'$ are bounded in $L^2_+$, hence are weakly convergent to $v_j$, $v'_k$. Arguing as in the previous subsection, we have
\beno
(u\vert v_1)&=&\lim _{\e \rightarrow 0}\Vert u_1^\e \Vert ^2=\lim _{\e \rightarrow 0}\frac{(\rh +\e )^2-\rh ^2}{(\rh +\e )^2-(\rh -\eta \e )^2}
\frac{\prod _{k\ge 2}((\rh +\e )^2-\si _k^2)}{\prod _{k\ge 3}((\rh+\e )^2-\rh _k^2)},\\
(u\vert v_2)&=&\lim _{\e \rightarrow 0}\Vert u_2^\e \Vert ^2=\lim _{\e \rightarrow 0}\frac{(\rh -\eta \e )^2-\rh ^2}{(\rh -\eta\e )^2-(\rh + \e )^2}
\frac{\prod _{k\ge 2}((\rh -\eta \e )^2-\si _k^2)}{\prod _{k\ge 3}((\rh-\eta \e )^2-\rh _k^2)},\\
(u\vert v_j)&=&\lim _{\e \rightarrow 0}\Vert u_j^\e \Vert ^2\\
&=&\lim _{\e \rightarrow 0}\frac{\rh _j^2-\rh ^2}{(\rh_j^2-(\rh -\eta \e )^2)(\rh _j^2-(\rh +\e )^2)}
\frac{\prod _{k\ge 2}(\rh _j^2-\si _k^2)}{\prod _{k\ge 3, k\ne j}(\rh _j^2-\rh _k^2)}, j\ge 3,
\eeno
and 
\beno
(u\vert v'_1)&=&\lim _{\e \rightarrow 0}\Vert (u_1^\e)' \Vert ^2\\&=&\lim _{\e \rightarrow 0}(\rh ^2-(\rh +\e )^2)(\rh^2-(\rh -\eta \e )^2)
\frac{\prod _{k\ge 3}(\rh ^2-\rh _k^2)}{\prod _{k\ge 2}(\rh^2-\si _k^2)},\\
(u\vert v'_k)&=&\lim _{\e \rightarrow 0}\Vert (u_k^\e)' \Vert ^2\\
&=&\lim _{\e \rightarrow 0}\frac{(\si_k ^2-(\rh +\e )^2)(\si _k^2-(\rh -\eta \e )^2)}{\si _k^2-\rh ^2}
\frac{\prod _{j\ge 3}(\si_k ^2-\rh _j^2)}{\prod _{j\ge 2, j\ne k}(\si _k^2-\si _j^2)},\ k\ge 2\ .
\eeno
In view of these identities, we infer that $v_j, j\ge 1$ and $v'_k,k\ge 2$ are not $0$. Passing to the limit into the identities
\begin{eqnarray*}
s ^\e _{2j-1}u^\e _{j}&=&\Psi_{2j}^\e\  H_{u^\e }u^\e _j\ ,\ H_{u^\e }^2(u^\e _j)=(s ^\e _{2j-1})^2u^\e _j\ ,\\
K_{u^\e }(u^\e _k)'&=&s ^\e _{2k}\Psi_{2k}^\e\ (u^\e _k)'\ ,\ K_{u^\e }^2(u^\e _k)'=(s ^\e _{2k})^2(u^\e _k)'\ ,\\
u^\e &=&\sum _{j=1}^q u^\e _j =\sum _{k=1}^q (u^\e_k)' \ ,
\end{eqnarray*}
we obtain
\begin{eqnarray*}
s  _{2j-1}v_j&=&\Psi_{2j}\  H_{u }v _j\ ,\ H_{u }^2(v _j)=s _{2j-1}^2v _j\ ,\\
K_{u }v_k'&=&s_{2k}\Psi_{2k} \ v _k' \ ,\ K_u^2(v'_k)=s_{2k}^2v'_k\ ,\\
u &=&\sum _{j=1}^q v _j =\sum _{k=1}^q v_k' \ ,
\end{eqnarray*}
hence 
$$\dim E_u(\rh _j)\ge m_j\ ,\ j\ge 3, \ \dim F_u(\si _k)\ge \ell _k\ ,\ k\ge 2\ .$$
In order to conclude that $u\in \mathcal V_{(d_1+1,d_2,\dots ,d_n)}$, it remains to prove that 
$$\dim E_u(\rh )\ge m_1+1\ .$$
We use the explicit formulae obtained in section \ref{section explicit}. We set $$\Psi_1(z)={\rm e}^{-i\f_1}\chi_1(z)\ .$$
We start with $\det \mathcal C(z)$ defined by (\ref{c}). Notice that elements $c _{11}(z)$ and $c_{12}(z)$ in formulae (\ref{c}) are of order $\e ^{-1}$, hence we compute 
\beno
&&\lim _{\e \rightarrow 0}2\e \det \mathcal C(z)=\\
&=& \left|\begin{array}{lll}1-z\expo_{-i\t _1}\Psi_1& \displaystyle {-\frac{1-z\expo _{-i(\t _1+\f _2)} }{
\eta }}& 0\dots \\
\displaystyle {\frac{\rh -z\sigma _k\Psi_1 \Psi_{2k}}{\rh ^2-\si _k ^2}}&  \displaystyle {\frac{\rh -z\sigma _k\expo _{-i\f _2} \Psi_{2k}}{\rh ^2-\si _k ^2}}& \displaystyle {\frac{\rh _3-z\sigma _k \Psi_5\Psi_{2k}}{\rh _3^2-\si _k ^2}},\dots  , k\ge 2 \end{array}\right|\ .
\eeno 
Let us add $\xi (z)$ times the first column to the second column in the above determinant, with
$$\xi (z):=  \frac{1-z\expo _{-i(\t _1+\f _2)} }{\eta (1-z\expo_{-i\t _1}\Psi_1(z))}\ .$$
We get
\beno    \lim _{\e \rightarrow 0}2\e \det \mathcal C(z)=(1-z\expo_{-i\t _1}\Psi_1)\det \left( \begin{array}{ll}  \zeta _k(z), &  \displaystyle {\frac{\rh _\ell -z\sigma _k \Psi_{2\ell -1}\Psi_{2k}}{\rh _\ell ^2-\si _k ^2}}, k\ge 2 , \ell \ge 3\end{array}\right )\ ,
\eeno
with 
\beno
 \zeta _k(z)&=&\frac{(1-z\expo _{-i(\t _1+\f _2)})
(\rh -z\sigma _k\Psi_1\Psi_{2k})   }{\eta (1- z\expo_{-i\t _1}\Psi_1)(\rho ^2-\si _k^2)}+\frac{ (1-z\expo_{-i\t _1}\Psi_1)(\rh -z\sigma _k\expo _{-i\f _2} \Psi_{2k})}{ (1- z\expo_{-i\t _1}\Psi_1)(\rho ^2-\si _k^2)}  \\
&=& \left (\frac 1\eta +1\right )\left (1-q(z)z\right )\frac{\rh -z\si _k\expo_{-i\psi }\chi_\psi(z)\Psi_{2k}(z)}{ (1- z\expo_{-i\t _1}\Psi_1)(\rho ^2-\si _k^2)} \  ,
\eeno
where
\beno
q(z)&:=&\frac{\expo_{-i\t _1}(\eta \Psi_1(z)+\expo_{-i\f _2})}{1+\eta }\ ,\\
\chi _\psi(z):&=&\frac{\chi_1(z)z-{\rm e}^{i\t_1}\frac{\eta{\rm e}^{i\f_1}+{\rm e}^{i\f_2}\chi_1(z)}{1+\eta}}{1-q(z)z}\ ,\ \psi :=\t_1+\f _1+\f _2+\pi \ .
\eeno
We know that $\chi_1$ is a Blaschke product of degree $m_1-1$. Let us verify that it is possible to choose $\varphi_2$ so that $\chi_\psi$ is a Blaschke product of degree $m_1$.  We first claim that it is possible to choose $\varphi_2$ so that $1-q(z)z\neq 0$ for $|z|\le 1$. Write $\alpha:=\frac 1{1+\eta}$, $\psi_1:=\varphi_1+\theta_1$ and $\psi_2:=\varphi_2+\theta_1$ and assume $1-q(z)z=0$.
Then
\begin{equation}\label{chi1}
1=(1-\alpha){\rm e}^{-i\psi_1}\chi_1(z)z+\alpha{\rm e}^{-i\psi_2}z\ .
\end{equation}
First notice that this clearly imposes $\vert z\vert =1$. Furthermore, this implies  equality in the Minkowski inequality, therefore there exists $\lambda>0$ so that $\chi_1(z)=\lambda{\rm e}^{-i(\psi_2-\psi_1)}$ and, eventually, that $\chi_1(z)={\rm e}^{-i(\psi_2-\psi_1)}$ since $|\chi_1(z)|=1$. Inserting this in equation (\ref{chi1}) gives $z={\rm e}^{i\psi_2}$ so that $\chi_1({\rm e}^{i\psi_2})={\rm e}^{-i(\psi_2-\psi_1)}$. If this equality holds true for any choice of $\psi_2$, by analytic continuation inside the unit disc, we would have
 $$\chi_1(z)=\frac{{\rm e}^{i\psi_1}}z$$ which is not possible since $\chi _1$ is  a holomorphic function in the unit disc. Hence, one can choose $\psi_2$, hence $\varphi_2$, in order to have $1-q(z)z\neq 0$ for any $|z|\le 1$. It implies that $\chi_\psi$ is a holomorphic rational function in the unit disc. Moreover, one can easily check that  it has modulus one on the unit circle, hence it is a Blaschke product. Finally, its  degree is $\deg(\chi _1)+1=m_1$.

Summing up,
$$\lim _{\e \rightarrow 0}2\e \det \mathcal C(z)=\left (1+\frac 1\eta \right )(1-q(z)z) \det((\tilde c_{k\ell })_{2\le k,\ell \le q+1})$$
where, for $k\ge 2$, $\ell \ge 3$,
\beno
\tilde c_{k2}&=& \frac{\rh -z\si _k\expo_{-i\psi }\chi _\psi(z)\Psi_{2k}(z)}{ (\rho ^2-\si _k^2)}\\
\tilde c_{k\ell }&=&c_{k\ell }= \frac{\rh _\ell -z\sigma _k \Psi_{2\ell -1}(z)\Psi_{2k}(z)}{\rh _\ell ^2-\si _k ^2}\ .
 \eeno
 
 Next, we perform the same calculation with the numerator of $u_j^\e (z)$, $j= 1,2$. Recall that 
 $$ u_j^\e (z)=\Psi_{2j-1}(z)\frac{\det \mathcal C_j(z)}{\det \mathcal C(z)}$$ where 
 $\mathcal C_j(z)$ denotes the matrix deduced from $(c_{k\ell }(z))_{1\le k,\ell \le q+1}$ by replacing the column $j$ by
\beno
{\bf 1}:=\left (\begin{array}{l} 1 \\ .\\ .\\ .\\ 1\end{array}\right )\ .
\eeno 
 We compute
 \beno &&\lim _{\e \rightarrow 0} 2\e\det \mathcal C_1(z)=\\
&=& \left|\begin{array}{lll}0& \displaystyle {-\frac{1-z\expo _{-i(\t _1+\f _2)} }{
\eta }}& 0\dots \\
\displaystyle {\bf 1}&  \displaystyle {\frac{\rh -z\sigma _k\expo _{-i\f _2} \Psi_{2k}}{\rh ^2-\si _k ^2}}& \displaystyle {\frac{\rh _3-z\sigma _k\Psi_5\Psi_{2k}}{\rh _3^2-\si _k ^2}},\dots  , k\ge 2 \end{array}\right|\\
&=& \frac{1-z\expo _{-i(\t _1+\f _2)} }\eta \det\left({\bf 1}, (c_{k\ell })_{k\ge 2, \ell \ge 3}\right)\eeno
and 
 \beno &&\lim _{\e \rightarrow 0}2\e \det \mathcal C_2(z)=\\
&=& \left|\begin{array}{lll} \displaystyle {1-z\expo _{-i\t _1} \Psi_1}&0& 0\dots \\
\displaystyle {\frac{\rh -z\sigma _k\Psi_1\Psi_{2k}}{\rh ^2-\si _k ^2}}&  \displaystyle {\bf 1}& \displaystyle {\frac{\rh _3-z\sigma _k\Psi_5\Psi_{2k}}{\rh _3^2-\si _k ^2}},\dots  , k\ge 2 \end{array}\right|\\
&=& (1-z\expo _{-i\t _1} \Psi_1)\det\left({\bf 1}, (c_{k\ell })_{k\ge 2, \ell \ge 3}\right))\eeno

Hence we have, for the weak convergence in $L^2_+$, 
 \beno v_1&:=&\lim _{\e \rightarrow 0}u_1^\e=\Psi_1\frac{(1-z\expo _{-i(\t _1+\f _2)} )}{(1+\eta)(1-q(z)z)}\cdot\frac{ \det\left(({\bf 1}
, (c_{k\ell })_{k\ge 2, \ell \ge 3}\right))}{
\det \left(( \tilde c_{k\ell })_{2\le k,\ell \le q+1}\right )}\\
 v_2&:=&\lim _{\e \rightarrow 0}u_2^\e=\eta{\rm e}^{-i\varphi_2}\frac{{(1- z\expo_{-i\t _1}\Psi_1)}}{(1+\eta)(1-q(z)z)}\cdot\frac{ \det\left(({\bf 1}
, (c_{k\ell })_{k\ge 2, \ell \ge 3}\right))}{\det \left( ( \tilde c_{k\ell })_{2\le k,\ell \le q+1}\right )}\ .
 \eeno
Furthermore,  if $D_1$ denotes the normalized denominator of $\Psi_1$, we have
 \begin{eqnarray*}
  H_{u^\e}^2\left(\frac {z^a}{D_1(z)}\frac{u_1^\e}{\Psi_1}\right)&=&(\rh+\e)^2\frac {z^a}{D_1(z)}\frac{u_1^\e}{\Psi_1}, \ 0\le a \le m_1 -1,\\ 
  H_{u^\e}^2(u_2^\e)&=&(\rh-\eta\e)^2u_2^\e,\\
\end{eqnarray*}
  Passing to the limit in these identities as $\e $ goes to $0$,  we get
  \begin{eqnarray*}
  H_{u}^2\left(\frac {z^a}{D_1(z)}\frac{v_1}{\Psi_1}\right)&=&\rh^2\frac {z^a}{D_1(z)}\frac{v_1}{\Psi_1}, \  0\le a\le m_1-1, \\ 
  H_{u}^2(v_2)&=&\rh^2v_2\ .
\end{eqnarray*}

  It remains to prove that the dimension of the vector space generated by $$\frac {z^a}{D_1(z)}\frac{v_1}{\Psi_1}, \ 0\le a\le m_1-1, v_2, $$  is $m_1+1$. From the expressions of $v_1$ and $v_2$, it is equivalent to prove that  the dimension of the vector space spanned by $$\frac{z^a}{D_1(z)}(1-{\rm e}^{-i\psi_2}z), \ 0\le a\le m_1-1,\ (1-{\rm e}^{-i\psi_1}z\chi_1(z))$$
  is $m_1+1$. Indeed, we claim that our choice of $\psi_2$ implies that this family is free. Assume that for some $\lambda_a$, $0\le a \le m_1-1$ we have
  $$\sum_{a=0}^{m_1-1}\lambda_a\frac{z^a}{D_1(z)}=\frac{1-{\rm e}^{-i\psi_1}z\chi_1(z)}{1-{\rm e}^{-i\psi_2}z}$$
  then, as the left hand side is a holomorphic function in $\overline \D$, it would imply $\chi_1({\rm e}^{i\psi_2})={\rm e}^{i(\psi_1-\psi_2)}$ but $\psi_2$ has been chosen so that this does not hold. 
  Eventually, we have constructed $u$ in $\mathcal V _{(d_1+1,\dots,d_n)}$. An analogous procedure would allow to construct $u$ in $\mathcal V _{(d_1,\dots, d_{r-1}, d_r+1,\dots,d_n)}$  for any $r$ odd. The case $r$ even can be handled similarly, by collapsing two variables $\sigma $ and one variable $\rh $. 
\end{proof}

\subsection{The case $n$ odd.}

 The proof of the fact that  $\Phi_n$ is one-to-one is the same as in the case $n$ even. One has to prove that $\Phi_n$ is onto. We shall proceed by approximation from the case $n$ even. We define $q=\frac {n+1}2$.

\noindent Let 
$$\mathcal P=((\rho _1,\sigma _1, \dots ,\rho _q), (\Psi_r)_{1\le r\le n})$$
be an arbitrary element of $\Omega _{n}\times \prod_{r=1}^n\mathcal B_{d_r}$. We look for $u\in \mathcal V _{(d_1,\dots,d_n)}$ such that
$\Phi_n(u)=\mathcal P$. Consider, for every $\e $ such that $0<\e < \rho _q$,
$$\mathcal P_\e=((\rho _1,\sigma _1, \dots ,\rho _q,\e), ((\Psi_r)_{1\le r\le n}, 1))\in \Omega _{n+1}\times  \prod_{r=1}^{n+1}\mathcal B_{d_r}$$ 
with $d_{n+1}:=0$ -- we take $\Psi_{2q}=1\in \mathcal B_0$. From Theorem \ref{mainfiniterank}, we get $u_\e\in \mathcal V_{(d_1,\dots,d_{n+1})}$ such that
$\Phi (u_\e)=\mathcal P_\e$. 
As before, we can prove by a compactness argument that a subsequence of $u_\e$ has a limit $u\in \mathcal V _{(d_1,\dots,d_n)}$ as $\e$ tends to $0$ with $\Phi_n(u)=\mathcal P$. We leave the details to the reader.
\end{proof}
\subsection{$\mathcal V _{(d_1,\dots,d_n)}$ is a manifold}
Let $d=n+2\sum _rd_r$. We consider the map
$$\Theta:=\left\{\begin{array}{lll} \Omega _{n}\times \prod_{r=1}^n\mathcal B_{d_r}  &\longrightarrow  &\mathcal V(d)\\
((s_r),(\Psi_r))&\longmapsto & u\end{array} \right. $$
where $u$ is given by the explicit formula obtained in section \ref{section explicit}.
This map is well defined and $\mathcal C^\infty$ on $\Omega _{n}\times \prod_{j=1}^n\mathcal B_{d_j}$. Moreover, from the previous section, it is a homeomorphism onto its range $\mathcal V _{(d_1,\dots,d_n)}$. In order to prove that $\mathcal V _{(d_1,\dots,d_n)}$ is a manifold of
dimension  $2n+2\sum_{j=1}^nd_j$, it is enough to check that the differential of $\Theta $ is injective at every point. From Lemma \ref{Phismooth}, near every point $u_0\in \mathcal V _{(d_1,\dots,d_n)}$, there exists a smooth function $\tilde \Phi _n$, defined on a neighborhood $V$ on $u_0$ in $\mathcal V(d)$, such that $\tilde \Phi _n$  
coincides with $\Phi _n$ on $V\cap \mathcal V _{(d_1,\dots,d_n)}$. Consequently, $\tilde \Phi _n\circ \Theta $ is the identity on a neighborhood of
$\mathcal P_0:= \Phi _n(u_0)$. In particular, the differential of $\Theta $ at $ \mathcal P_0$ is injective.

\subsection{Proof of Theorem \ref{realmain} in the finite rank case}
Denote by $L^2_{+,r}$ the real subspace of $L^2_+$ made of functions with real Fourier coefficients. If $u\in (VMO_+\setminus \{ 0 \}) \cap L^2_{+,r}$, then $H_u$ acts on $L^2_{+,r}$ as a compact self adjoint operator, which is unitarily equivalent  to $\Gamma _c$ if $u=u_c$. Consequently, for every Borel real valued function $f$, $f(H_u^2)$ acts on
$L^2_{+,r}$. In particular, the orthogonal projections $u_j$, $u'_k$ belong to $L^2_{+,r}$. Therefore, for every $r$,  the Blaschke product $\Psi _r(u)$ belongs to $L^2_{+,r}$, which means that its coefficients are real, in particular $\psi _r\in \{ 0,\pi \} $. Moreover, by Proposition \ref{action}, for every $r$, there exists bases of $E_u(s_r)\cap L^2_{+,r}$ and $F_u(s_r)\cap L^2_{+,r}$ on which the respective actions of $H_u$ and $K_u$ are described by  matrices of the type $\varepsilon _r s_r A$, where $\varepsilon _r=\expo _{-i\psi _r}=\pm 1$ and
$$A=\left (\begin{array}{lcccl}   0& \ldots & \ldots& 0& 1\\ 0&\ldots&\ldots&1&0\\
\vdots& \ldots& \diagup&\ldots&\vdots \\ 
0&1&\ldots& \ldots&0\\
1&0&\ldots &\ldots&0
  \end{array}\right )\ ,$$
  being of dimension $d_r+1$ or $d_r$.
By an elementary observation, the eigenvalues of $A$ are $\pm 1$, with equal multiplicities if the dimension of $A$ is even, and where the multiplicity of $1$ is one unit greater than the multiplicity of $-1$ if the dimension of $A$ is odd. Consequently, if we denote by $(\lambda _j),
(\mu _k)$ the respective sequences of non zero eigenvalues of $H_u$ and of $K_u$ on $L^2_{+,r}$, repeated according to their multiplicities, and 
ordered following
$$\vert \lambda _1\vert \ge \vert \mu _1\vert \ge \vert \lambda _2\vert \ge \dots \ ,$$
each $s_r$ correspond to a maximal string with consecutive equal terms, of length $2d_r+1$, where $d_r$ is the degree of $\Psi _r$. 
Moreover, we have
\begin{eqnarray*}
&\vert  \# \{ j: \lambda _j=s_r \} - \# \{ j: \lambda _j=-s_r \} \vert \le 1\ ,\\
 &\vert  \# \{ k: \mu _k=s_r \} - \# \{ k: \mu _k=-s_r \} \vert \le 1.
 \end{eqnarray*}
This is the Megretskii--Peller--Treil condition. Moreover, according to the parity of $r$ and $d_r$, if one of the above integers is $0$, the other one is $1$, and the eigenvalue with the greatest multiplicity is then $\varepsilon _rs_r$.
\s
Conversely, given two sequences $(\lambda _j), (\mu _k)$ satisfying the assumptions of Theorem \ref{realmain}, the above considerations imply that the set of solutions $u$ to the inverse spectral problem is exactly
$$\Phi ^{-1}\left ( \{ (s_r)\} \times \varepsilon_1\mathcal B_{d_1,r}^\sharp  \times \dots \times  \varepsilon_n\mathcal B_{d_n,r}^\sharp     \right )$$
where $\mathcal B_{d,r}^\sharp $ denotes the set of Blaschke products of degree $d$, with real coefficients and with angle $0$, which is diffeomorphic to $\R ^d$ in view of the result of Appendix B.
Notice that the explicit formula (\ref{luminy}) allows to check that $u$ belongs to $L^2_{+,r}$. In view of Theorem \ref{Phitopodif}, we conclude that the set of solutions $u$ to the inverse spectral problem  is diffeomorphic to $\R ^M$, with 
$M=\sum _rd_r$.

\section{Extension to compact Hankel operators}\label{compact}
 
In this section, we prove the parts of Theorems \ref{Phibij}, \ref{Phitopodif} corresponding to infinite rank Hankel operators.
Given an arbitrary sequence $(d_r)_{r\ge 1}$ of nonnegative integers, we set 
$$\mathcal V_{(d_r)_{r\ge 1}}:=\Phi ^{-1}\left (\Omega _\infty \times \prod _{r=1}^\infty \mathcal B_{d_r}\right )\ . $$ 
 \begin{theorem}\label{mainCompact}
The mapping
\begin{eqnarray*}
\Phi : \mathcal V_{(d_r)_{r\ge 1}} &\longrightarrow &  \Omega _{\infty}\times \prod _{r=1}^\infty \mathcal B_{d_r}\\
u&\longmapsto& ((s_r)_{r\ge 1}, (\Psi_r)_{ r\ge 1})
\end{eqnarray*}
is a homeomorphism.
\end{theorem}
Before giving the proof of this theorem, notice that it implies Theorem \ref{realmain} in the infinite case, following the same considerations as in the finite rank case above.
\begin{proof}
The fact that $\Phi $ is one-to-one follows from an explicit formula analogous to the one obtained in the finite rank case,
see section \ref{section explicit}. However, in this infinite rank situation, we have to deal with the continuity of infinite rank matrices
on appropriate $\ell ^2$ spaces. 
\s
Indeed, we still have
$$u=\sum_{j=1}^\infty \Psi_{2j-1} h_j$$
where $\mathcal H(z):=(h_j(z))_{j\ge 1}$ satisfies, for every $z\in \D $, 
\begin{equation}\label{eqHinfini}
\mathcal H(z)=\mathcal F(z)+z\mathcal D(z)\mathcal H(z)
\end{equation}
with
\begin{eqnarray*}
\mathcal F(z)&:=&\left (\frac{\tau _j^2}{\rho_j}\right )_{j\ge 1}\ ,\\
\mathcal D(z)&:=&\left (\frac{\tau _j^2}{\rho_j}\sum _{k=1}^\infty \frac {\kappa _k^2\sigma _k\Psi _{2k}(z)\Psi _{2\ell -1}(z)}{(\rho _j^2-\sigma _k^2)(\rho _\ell ^2-\sigma _k^2)}\right )_{j,\ell\ge 1}\ .
\end{eqnarray*}
Notice that  the coefficients of the infinite matrix $\mathcal D (z)$ depend holomorphically on $z\in \D$. 
We are going to prove that, for every $z\in \D $, $\mathcal D(z)$ defines a contraction  on the space $\ell ^2_\tau $ of sequences
$(v_j)_{j\ge 1}$ satisfying
$$\sum _{j=1}^\infty \frac{\vert v_j\vert ^2}{\tau _j^2}<\infty \ .$$
From the maximum principle, we may  assume that $z$ belongs to the unit circle. Then $z$ and $\Psi _r(z)$ have modulus $1$. We then compute $\mathcal D(z)\mathcal D(z)^*$, where the adjoint is taken  for the inner product associated to $\ell ^2_\tau $. We get, using identities (\ref{sommedoubletau}), (\ref{sommesimplekappa}) and (\ref{sommedoublekappa}),
\beno 
\left [\mathcal D(z)\mathcal D(z)^*\right ] _{jn}&=&\frac{\tau _j^2}{\rho_j\rho _n}\sum _{k,\ell ,m}
\frac {\kappa _k^2\sigma _k\Psi _{2k}(z)\tau _\ell ^2\kappa _m ^2\si _m\overline {\Psi _{2m}(z)}}{(\rho _j^2-\sigma _k^2)(\rho _\ell ^2-\sigma _k^2)(\rho _n^2-\sigma _m^2)(\rho _\ell ^2-\sigma _m^2)}\\
&=& \frac{\tau _j^2}{\rho_j\rho _n}\sum _{k}\frac{\kappa _k^2\si _k^2}{(\rho _j^2-\sigma _k^2)(\rho _n^2-\sigma _k^2)}\\
&=& -\frac{\tau _j^2}{\rho_j\rho _n}+\delta _{jn}\ .
\eeno
Since, from the identity (\ref{sommenu}) in Appendix A, 
$$\sum _{j=1}^\infty \frac {\tau _j^2}{\rh _j^2}\le 1\ ,$$
we conclude that $\mathcal D(z)\mathcal D(z)^*\le I$ on $\ell ^2_\tau $, and consequently that
$$\Vert \mathcal D(z)\Vert _{\ell ^2_\tau \rightarrow \ell ^2_\tau }\le 1\ .$$
From the Cauchy inequalities, this implies
$$\Vert \mathcal D^{(n)}(0)\Vert _{\ell ^2_\tau \rightarrow \ell ^2_\tau }\le n!\ .$$
Coming back to equation (\ref{eqHinfini}), we observe that $\mathcal H(0)=\mathcal F(0)\in \ell ^2_\tau $, and that, for every $n\ge 0$,
$$\mathcal H^{(n+1)}(0)=(n+1)\sum _{p=0}^n \begin{pmatrix}n\\p\end{pmatrix}\mathcal D^{(p)}(0)\mathcal H^{(n-p)}(0)\ .$$
By induction on $n$, this determines $\mathcal H^{(n)}(0)\in \ell ^2_\tau $, whence the injectivity of $\Phi $.
\s
Next, we prove that $\Phi $ is onto. We pick an element $$\mathcal P\in \Omega _{\infty}\times\prod _{r=1}^\infty \mathcal B_{d_r}$$ and we construct $u\in VMO_+$ so that $\Phi(u)=\mathcal P$.
Set $$\mathcal P=((\rho _1,\sigma _1, \rho _2,\dots), (\Psi _r)_{r\ge 1})$$
and consider, for any integer $N$,
$$\mathcal P_N:=((\rho _1,\sigma _1, \dots,\rho_N,\sigma_N), (\Psi_r)_{1\le r\le 2N})$$ in 
$$ \Omega_{2N}\times\prod_{j=1}^{2N} \mathcal B_{d_j}\ .$$
From Theorem \ref{mainfiniterank}, there exists $u_N\in \mathcal V _{(d_1,\dots,d_{2N})}$ with $\Phi(u_N)=\mathcal P_N$. As $u_N$ is bounded in $L^2_+$, there exists a subsequence converging weakly  to some $u$ in $L^2_+$. Let $u_{N,j}$ and $u'_{N,k}$ denote the orthogonal projections of $u_N$ respectively on $E_{u_N}(\rh _j)$ and on $F_{u_N}(\si _k)$ so that we have the orthogonal decompositions,
$$u_N=\sum_{j=1}^Nu_{N,j}=\sum_{k=1}^N u'_{N,k}\ .$$
After a diagonal extraction procedure, one may assume that $u_{N,j}$ and $u'_{N,k}$ converge weakly  in $L^2_+$ respectively to some $v^{(j)}$ and to some ${v'}^{(k)}$.
In fact one may assume that $u_N$ converges strongly to $u$ in $L^2_+$. The proof is along the same lines as the one developed for Proposition 2 in \cite{GG3}, and is based on the Adamyan-Arov-Krein (AAK)theorem \cite{AAK}, \cite{Pe2}. Let us recall the argument. 
\s
First we recall that the  AAK theorem  states that the $(p+1)$-th singular value of a Hankel operator, as the distance of this operator to operators of rank at most $p$,  is exactly achieved by some Hankel operator of rank at most $p$, hence, with a rational symbol. We refer to part (2) of the theorem in Appendix C. We set, for every $m\ge 1$,
$$p_m=m+\sum _{r\le 2m} d_r\ .$$
With the notation of Appendix C, one easily checks that, for every $m$,
$$\underline s_{p_{m-1}}(u)>\underline s_{p_m}(u)=\rho _{m+1}(u)\ .$$
By part (1) of the AAK theorem in Appendix C, for every $N$ and every $m=1,\dots, N$, there exists a rational symbol $u_{N}^{(m)}$, defining a Hankel operator of rank $p_m$,  namely $u_{N}^{(m)}\in \mathcal V(2p_m)\cup {\mathcal V}(2p_m-1)$, such that 
$$\Vert H_{u_N}-H_{u_{N}^{(m)}}\Vert =\rho_{m+1}(u_N)=\rho _{m+1}.$$
In particular, we get
$$\Vert u_N-u_{N}^{(m)}\Vert_{L^2}\le \rho_{m+1}.$$
On the other hand, one has $$\Vert H_{u_{N}^{(m)}}\Vert \ge \frac{1}{\sqrt {p_m}}(Tr(H^2_{u_{N}^{(m)}}))^{1/2}\ge \frac 1{\sqrt {p_m}}\Vert u_{N}^{(m)}\Vert_{H^{1/2}_+}.$$  Hence, for fixed $m$, the sequence $(u_{N}^{(m)})_N$ is bounded in $H^{1/2}_+$.
Our aim is to prove that the sequence $(u_N)$ is precompact in $L^2_+$.
We show that, for any $\varepsilon>0$ there exists a finite sequence $v_k\in L^2_+$, $1\le k\le M$ so that $$\{u_N\}_N\subset\bigcup_{k=1}^M B_{L^2_+}(v_k,\varepsilon).\ $$
Let $m$ be fixed such that $$\rho_{m+1}\le \varepsilon/2.$$ Since the sequence $(u_{N}^{(m)})_N$ is uniformly bounded in $H^{1/2}_+$,  it is precompact in $L^2_+$, hence there exists $v_k\in L^2_+$, $1\le k\le M$, such that $$\{u_{N}^{(m)}\}_N\subset\bigcup_{k=1}^M B_{L^2_+}(v_k,\varepsilon/2)\ .$$ Then, for every $N$ there exists some $k$ such that
$$\Vert u_N-v_k\Vert_{L^2}\le \rho_{m+1}+\Vert u_{N}^{(m)}-v_k\Vert_{L^2}\le \varepsilon.$$
Therefore $\{ u_N \}$ is precompact in $L^2_+$ and, since $u_N$ converges weakly to $u$, it converges strongly to $u$ in $L^2_+$. Since $\Vert H_{u_N}\Vert =\rh _1$ is bounded, we infer the strong convergence of  operators,
$$\forall h\in L^2_+, H_{u_N}(h)\td_p,\infty H_u(h)\ .$$
We now observe that if $\rho^2$ is an eigenvalue of $H_{u_N}^2$ of multiplicity $m$ then $\rho^2$ is an eigenvalue of $H_u^2$ of multiplicity at most  $m$.
Let $(e_N^{(l)})_{1\le l\le m}$ be an orthonormal family of eigenvectors of $H_{u_N}^2$ associated to the eigenvector $\rho^2$. Let $h$ be in $L^2_+$ and write
$$
h=\sum_{l=1}^m (h\vert  e_N^{(l)})e_N^{(l)}+h_{0,N}$$ where $h_{0,N}$ is the orthogonal projection of $h$ on the orthogonal complement of $E_{u_N}(\rho)$ so that 
\begin{eqnarray*}
\Vert (H_{u_N}^2-\rho^2 I)h\Vert^2&=&\Vert (H_{u_N}^2-\rho^2 I)h_{0,N}\Vert^2\\
&\ge & d_{\rho^2}\Vert h_{0,N}\Vert ^2=d_{\rho^2}(\Vert h\Vert^2-\sum_{l=1}^m \vert (h\vert e_N^{(l)})|^2)\ ,
 \end{eqnarray*}
here $d_{\rho^2}$ denotes the distance to the other eigenvalues of $H_{u_N}^2$.
 By taking the limit as $N$ tends to $\infty$ one gets
 $$\Vert (H_{u}^2-\rho^2 I)h\Vert^2\ge d_{\rho^2}(\Vert h\Vert^2-\sum_{l=1}^m|(h\vert e^{(l)})|^2)$$
 where $e^{(l)}$ denotes a weak limit of $e_N^{(l)}$.  
 Assume now that the dimension of $E_u(\rho)$ is larger than $m+1$ then we could construct $h$ orthogonal to $(e^{(1)},\dots e^{(m)})$ with $H_u^2(h)=\rho^2h$, a contradiction. The same argument allows to obtain that if $\rho^2$ is not an eigenvalue of $H_{u_N}^2$, $\rho^2$ is not an eigenvalue of $H_u^2$.
 
 We now argue as in section \ref{closed} above. We may assume, up to extracting a subsequence, that $u_{N,j}$ weakly converges to $v_j$,
 and that $u_{N,k}'$ weakly converges to $v'_k$ in $L^2_+$, with the identities
 $$\rh _jv_j=\Psi _{2j-1}H_u(v_j)\ ,\ H_u^2(v_j)=\rh _j^2v_j\ ,\ K_u(v'_k)=\si _k\Psi _{2k}v'_k\ ,\ K_u^2(v'_k)=\si _k^2v'_k\ .$$
 and  
 $$(u\vert v_j)=\tau _j^2\ ,\ (u\vert v'_k)=\kappa _k^2\ .$$
 This already implies that $v_j$, $v'_k$ are not $0$, and hence, in view of Lemmas \ref{crucialHuGeneral} and \ref{crucialKuGeneral}, that 
 $$\dim E_u(\rh _j)= m_j\ ,\ \dim F_u(\si _k)= \ell _k\ .$$
  We infer that $u\in \mathcal V_{(d_r)_{r\ge 1}}$ and that $\rho _j=s_{2j-1}(u)\ ,\ \si _k=s_{2k}(u)\ .$ It remains to identify $v_j$ with the orthogonal projection $u_j$  of $u$ onto $E_u(\rh _j)$,
 and $v'_k$ with the orthogonal projection $u'_k$ of $u$ onto $F_u(\si _k)$. The strategy of passing to the limit, as $N$ tends to infinity,  in the decompositions
 $$u_N=\sum _{j=1}^N u_{N,j}=\sum _{k=1}^Nu'_{N,k}\ $$
 is not easy to apply because of infinite sums. Hence we argue as follows. From the identity
 $$\Vert u_{N,j}\Vert^2=(u_N\vert u_{N,j})$$
 we get
 $$\Vert v_j\Vert ^2\le (u\vert v_j)=(u_j\vert v_j)\ ,$$
 and, by the Cauchy-Schwarz inequality, since $v_j\ne 0$,
 $$\Vert v_j\Vert \le \Vert u_j\Vert\ .$$
 On the other hand, we know from the general formulae of Appendix A that
 $$\Vert u_j\Vert ^2=\tau _j^2\ .$$
 Since $\tau _j^2=(u_j\vert v_j)$, we get $\Vert u_j\Vert \le \Vert v_j\Vert $ and finally infer
 $$(u_j\vert  v_j)=\Vert v_j\Vert^2= \Vert u_j\Vert ^2\ ,$$ 
 hence 
$$\Vert v_j-u_j\Vert^2=0\ .$$
Similarly, $v'_k=u'_k$. This completes the proof of the surjectivity.
\s
The continuity of $\Phi $ follows as in section \ref{continuity}. As for the continuity of $\Phi ^{-1}$, we argue exactly as for surjectivity
above, except that we have to prove the convergence of $u_N$ to $u$ in $VMO_+$. This can be achieved exactly as in the proof of Proposition 2 of \cite{GG3} : the Adamyan-Arov-Krein theorem allows to reduce to the following statement : if $w_N\in \mathcal V(2p)\cup \mathcal V(2p-1)$ 
strongly converges to $w\in \mathcal V(2p)\cup \mathcal V(2p-1)$, then the convergence takes place in $VMO $ --- in fact in $C^\infty $. See Lemma 3 of \cite{GG3}.
\end{proof}
For future reference, we state a similar result in the case of Hilbert--Schmidt operators. We set $$\mathcal V^{(2)}_{(d_r)_{r\ge 1}}:=\mathcal V_{(d_r)_{r\ge 1}}\cap H^{1/2}_+\ , $$
and 
$$\Omega _\infty ^{(2)}((d_r)_{r\ge 1}) :=\{ (s_r)_{r\ge 1}\in \Omega _\infty : \sum _{r=1}^\infty(d_r+1)s_r^2<\infty \} \ ,$$
endowed with the topology induced by the above weighted $\ell ^2$ norm.
 \begin{theorem}\label{HilbertSchmidt}
The mapping
\begin{eqnarray*}
\Phi : \mathcal V^{(2)}_{(d_r)_{r\ge 1}} &\longrightarrow &  \Omega ^{(2)}_{\infty}((d_r)_{r\ge 1})\times \prod _{r=1}^\infty \mathcal B_{d_r}\\
u&\longmapsto& ((s_r)_{r\ge 1}, (\Psi_r)_{ r\ge 1})
\end{eqnarray*}
is a homeomorphism.
\end{theorem}
The proof is essentially the same as the one of Theorem \ref{mainCompact}, except that the argument based on the AAK theorem is simplified by the identity
$${\rm Tr}(H_u^2)=\sum _{r=1}^\infty  d_r s_r^2(u)+\sum _{j=1}^\infty s _{2j-1}^2(u)\ ,$$
which provides bounds in $H^{1/2}_+$, hence strong convergence in $L^2_+$, and finally strong convergence in $H^{1/2}_+$. We leave the easy details to the reader.
\section{Evolution under the cubic Szeg\H{o} flow} \label{szegoflow}
\subsection{The theorem}
In this section, we prove the following result.
\begin{theorem}\label{evolszegotexte}
Let $u_0\in H^{1/2}_+$ with $$\Phi(u_0)=((s_r),(\Psi_r)).$$
The solution of $$i\partial_t u=\Pi(|u|^2u),\; u(0)=u_0$$ is characterized by
$$\Phi(u(t))=((s_r),(\expo_{i(-1)^rs_r^2t}\Psi _r))\ .$$
\end{theorem}
\begin{remark} It is in fact possible to define the  flow of the cubic Szeg\H{o}  on $BMO_+=BMO(\T )\cap L^2_+$, see \cite{GK}.  The above theorem then extends to the case of  an initial datum $u_0$ in $VMO_+$ .
\end{remark}
\begin{proof}
In view of the continuity of the flow map on $H^{1/2}_+$,  see \cite{GG1}, we may assume that $H_{u_0}$ is of finite rank. Let $u$  be the corresponding solution of the cubic Szeg\H{o} equation. 
Let $\rho$ be a singular value of $H_u$ in $\Sigma_H(u)$ such that $m:=\dim E_u(\rho)=\dim F_u(\rho)+1$ and denote by $u_\rho$ the orthogonal projection of $u$ on $E_u(\rho)$.
Hence, $u_\rho=1\!{\rm l}_{\{ \rho^2\}}(H_u^2)(u)$. Let us differentiate this equation with respect to time. 
Recall \cite{GG1}, \cite{GG4} that 
\begin{equation}\label{Laxpair}
\frac{dH_u}{dt}=[B_u,H_u] \text{ with }B_u=\frac i2H_u^2-i T_{|u|^2}\ .
\end{equation}
Here $T_b $ denotes the Toeplitz operator of symbol $b$,
\begin{equation}\label{Toeplitz}
T_b(h)=\Pi (bh)\ ,\ h\in L^2_+\ ,\ b\in L^\infty \ .
\end{equation}
Equation (\ref{Laxpair}) implies, for every Borel function $f$,
$$\frac{df(H_u^2)}{dt}=-i[T_{\vert u\vert ^2},f(H_u^2)]\ .$$
We get from this Lax pair structure
\begin{eqnarray*}
\frac{d u_\rho}{dt}&=& -i[T_{\vert u\vert ^2},1\!{\rm l}_{\{ \rho^2\}}(H_u^2)](u)+1\!{\rm l}_{\{ \rho^2\}}(H_u^2)\left (\frac{du}{dt}\right )\\
&=&  -i[T_{\vert u\vert ^2},1\!{\rm l}_{\{ \rho^2\}}(H_u^2)](u)+1\!{\rm l}_{\{ \rho^2\}}(H_u^2)\left (-iT_{\vert u\vert ^2}u\right )\ ,
\end{eqnarray*}
and eventually
\begin{equation}\label{deriveurho}
\frac{d u_\rho}{dt}=-iT_{\vert u\vert ^2}u_\rh \ .
\end{equation}

On the other hand, differentiating the equation
$$\rho u_\rho=\Psi H_u(u_\rho)$$ one obtains
$$
\rho \frac{d u_\rho}{dt}=\dot \Psi H_u(u_\rho)+\Psi\left([B_u,H_u](u_\rho)+H_u\left( \frac{d u_\rho}{dt}\right)\right )$$
Hence, using the expression (\ref{deriveurho}), we get
$$-i\rho T_{\vert u\vert ^2}(u_\rho)=\dot \Psi H_u(u_\rh )+\Psi \left (-iT_{\vert u\vert ^2}H_u(u_\rh ) +i\rh ^2H_u(u_\rh )\right )\ ,$$
hence
$$-i[T_{\vert u\vert ^2},\Psi ] H_u(u_\rh )=(\dot \Psi +i\rh ^2\Psi )H_u(u_\rh )\ .$$
We claim that the left hand side of this equality is zero. Assume this claim proved, we get, as $H_u(u_\rho)$ is not identically zero, that 
$\dot \Psi +i\rh ^2\Psi =0$, whence 
$$\Psi (t)=\expo_{-it\rh ^2}\Psi (0)\ .$$

It remains to prove the claim.
We first prove that, for any $p\in\D$ such that $\chi_p$ is a factor of $\chi$,
$$[T_{\vert u\vert ^2},\chi_p](e)=0$$
  for any $e\in E_u(\rho)$ such that $\chi _pe\in E_u(\rh )$.
Recall that
$$\chi _p(z)=\frac{z-p}{1-\overline pz}\ .$$
For any $L^2$ function $f$, $$\Pi(\chi_pf)-\chi_p\Pi(f)=K_{\chi _p}(g)=(1-\vert p\vert ^2)H_{1/(1-\overline pz)}(g)\ ,$$
where $\overline {(I-\Pi )f}=Sg\ .$ Consequently, the range of $[\Pi ,\chi _p]$ is one dimensional, directed by  $\frac 1{1-\overline pz}$. 
In particular, $[T_{\vert u\vert ^2},\chi_p](e)$ is proportional to $\frac 1{1-\overline pz}$. On the other hand,
\begin{eqnarray*}
([T_{\vert u\vert ^2},\chi_p](e)\vert 1)&=&(T_{|u|^2}(\chi_p e)-\chi_pT_{|u|^2}(e))\vert 1)\\
&=&(\chi_p(e)\vert H_u^2(1))-(\chi_p\vert 1)(e\vert H_u^2(1))\\
&=&(H_u^2(\chi_p(e))\vert 1)-(\chi_p\vert 1)(H_u^2(e)\vert 1)=0\ .
\end{eqnarray*}
This proves that $[T_{\vert u\vert ^2},\chi_p](e)=0$.

For the general case, we write $\Psi=\expo_{-i\psi}\chi_{p_1}\dots\chi_{p_{m-1}}$ and 
$$
[T_{\vert u\vert ^2},\Psi]H_u(u_\rho)=\expo_{-i\psi}\sum_{j=1}^{m-1}\prod_{k=1}^{j-1}\chi_{p_k}[T_{\vert u\vert ^2},\chi_{p_j}]\prod_{k=j+1}^{m-1}\chi_{p_k}H_u(u_\rho)=0\ .
$$
It remains to consider the evolution of the $\Psi _{2k}$'s. Let $\sigma$ be a singular value of $K_u$ in $\Sigma_K(u)$ such that $\dim F_u(\sigma)=\dim E_u(\sigma)+1$ and denote by $u'_\sigma$ the orthogonal projection of $u$ onto $F_u(\si )$. Recall \cite{GG4} that
$$\frac{dK_u}{dt}=[C_u,K_u]\text{ with }C_u=\frac i2K_u^2-i T_{|u|^2} \ .$$
As before, we compute the derivative in time of $u'_\sigma=1\!{\rm l}_{\{\sigma^2\}}(K_u^2)(u)$, and get
\begin{equation}\label{deriveuprimesigma}
 \frac{d u'_\sigma}{dt}=-iT_{\vert u\vert ^2}u'_\sigma\ .
\end{equation}
On the other hand, differentiating the equation
$$K_u(u'_\sigma)=\sigma \Psi u'_\sigma$$
one obtains
$$
-i[T_{\vert u\vert ^2}, \Psi ]u'_\si =(\dot \Psi -i\si ^2\Psi )u'_\si \ .
$$
As before, we prove that the left hand side of the latter identity is $0$, by checking that, for every factor $\chi _p$ of $\Psi $, for any $f\in F_u(\si )$
such that $\chi _pf\in F_u(\si )$, 
$$([T_{\vert u\vert ^2},\chi _p](f)\vert 1)=0\ .$$
The calculation leads to
\beno
([T_{\vert u\vert ^2},\chi _p](f)\vert 1)&=&(H_u^2(\chi _pf)-(\chi _p\vert 1)H_u^2(f)\vert 1)\\
&=&((\chi _p-(\chi _p\vert 1))f\vert u)(u\vert 1),
\eeno 
where we have used (\ref{Ku2}). Now $(\chi _p-(\chi _p\vert 1))f\in F_u(\si )$ is orthogonal to $1$, hence, from Proposition \ref{action}, it belongs to $E_u(\si )$, hence it is orthogonal to $u$. This completes the proof.

\end{proof}

\subsection{Application: traveling waves revisited}

As an application of Theorem \ref{Phibij} and of the previous section, we revisit the traveling waves of the cubic Szeg\H{o} equation. These are the solutions of the form
 $${u(t,e^{ix})=e^{-i\omega t}u_0(e^{i(x -ct)})\ ,\ \omega,c\in\R\ .}$$
 For $c=0$, it is easy to see \cite{GG1} that this condition for $u_0\in H^{1/2}_+$ corresponds to finite Blaschke product. The problem of characterizing traveling waves with  $c\ne 0$ is more delicate, and was solved in \cite{GG1} by the following result.

\begin{theo} \cite{GG1}
A function $u$ in $H^{1/2}_+$ is a traveling wave with $c\neq 0$ and $\omega\in \R$ if and only if  there exist non negative integers $\ell$ and $N$, $0\le \ell\le N-1$,  $\alpha\in \R$  and a complex number $p\in\C$ with $0<|p|<1$ so that
$$u(z)=\frac{\alpha z^\ell}{1-pz^N}$$
\end{theo}

Here we give an elementary proof of this theorem.
\begin{proof}
The idea is to keep track of the Blaschke products associated to $u$ through  the following unitary transform on $L^2(\T )$,
$$\tau _\alpha f(\expo_{ix}):=f(\expo_{i(x-\alpha)})\ ,\ \alpha \in \R .$$
Since $\tau _\alpha $ commutes to $\Pi $, notice that 
$$\tau _\alpha (H_u(h))=H_{\tau _\alpha (u)}(\tau _\alpha (h))\ .$$
Consequently, $\tau _\alpha $ sends $E_u(\rh )$ onto $E_{\tau _\alpha (u)} (\rh )$, and 
$$\tau _\alpha (u_\rh )=[\tau _\alpha (u)]_\rh \ .$$
Applying $\tau _\alpha $ to the identity
$$\rh u_\rh =\Psi_\rh H_u(u_\rh )\ ,$$
we infer
$$\rh [\tau _\alpha (u)]_\rh =\tau _\alpha (\Psi_\rh )H_{\tau _\alpha (u)}\left ([\tau _\alpha (u)]_\rh \right )\ ,$$
and similarly 
$$\rh [\expo_{-i\beta }\tau _\alpha (u)]_\rh =\expo_{-i\beta }\tau _\alpha (\Psi_\rh )H_{\expo_{-i\beta }\tau _\alpha (u)}\left ([\expo_{-i\beta }\tau _\alpha (u)]_\rh \right )\  .$$
This leads, for every $\rh \in \Sigma _H(u)$, to
$$\Psi _\rh (\expo_{-i\beta }\tau _\alpha (u))=\expo_{-i\beta }\tau _\alpha (\Psi _\rh (u))\ .$$
Applying this identity to $u=u_0$, $\alpha =ct$ and $\beta =\omega t$, and comparing with Theorem \ref{evolszego}, we conclude
$$\expo_{-it\rh ^2}\Psi _\rh (u_0)=\expo_{-i\omega t }\tau _{ct}(\Psi _\rh (u_0))\ .$$
Writing 
$$\Psi _\rh (u_0)={\rm e}^{-i\varphi}\prod_{1\le j\le m-1} \chi_{p_j},$$
we get, for every $t$,
$$\expo_{-it\rh ^2}\prod_{1\le j\le m-1} \chi_{p_j}=\expo_{-it(\omega +c(m-1))}\prod_{1\le j\le m-1} \chi_{\expo_{ict}p_j}\ .$$
This imposes, since $c\ne 0$, 
$$\rho^2=\omega +(m-1)c\ ,\ p_j=0\ ,$$
for every $\rho \in \Sigma _H(u_0)$. In other words, $\Psi _\rh (u_0)(z)=\expo_{-i\varphi } z^{m-1}\ .$
\s
We repeat the same argument for $\si \in \Sigma _K(u)$,  with $\ell=\dim F_u(\sigma)=\dim E_u(\sigma)+1$ and $$K_u(u'_\sigma)=\sigma \Psi_\sigma u'_\sigma\ ,$$
using this time
$$\tau _\alpha (K_u(h))=\expo _{i\alpha }K_{\tau _\alpha (u)}(\tau _\alpha (h))\ .$$
We get 
$$\sigma^2=\omega-\ell c \ ,$$
and 
$$\Psi _\si (u_0)(z)=\expo_{-i\theta}z^{\ell -1}\ .$$
If we assume that there exists at least two elements $\rh _1>\rh _2$ in $\Sigma _H(u_0)$, 
with $m_j=\dim E_{u_0}(\rho_j)$ for $j=1,2$, from Proposition \ref{rigidity}, there is at least one element $\sigma _1$ in $\Sigma _K(u_0)$, satisfying 
$$\rho _1 > \sigma_1  > \rho _2.$$
Set $\ell _1:=\dim F_{u_0}(\si _1)$, we get
$$(m_1-1)c > -\ell _1 c > (m_2-1)c$$ which is impossible since $m_1,\ell_1, m_2$ are positive integers.
Therefore,  there is only one element $\rho$ in $\Sigma _H(u_0)$, with $m=\dim E_{u_0}(\rho)$ and at most one element $\sigma$ in $\Sigma _K(u_0)$, of multiplicity $\ell$. 
Applying the results of section \ref{section explicit}, we obtain
$$u_0(z)=\frac {(\rh ^2-\si ^2)\expo_{-i\varphi }}{\rh }\frac{z^{m-1}}{1-\frac{\sigma }\rh {\rm e}^{-i(\varphi+\theta)} z^{\ell +m-1}}\ .$$
This completes the proof.
\end{proof}

\subsection{Application to almost periodicity}
As a second application of our main result, we prove that the solutions of the Szeg\H{o} equation are almost periodic. Let us recall a definition. Let $X$ be a Banach space. A function
$$f: \R \longrightarrow X$$
is almost periodic  if it is the uniform limit of quasi-periodic functions, namely finite linear combinations of functions 
$$t\longmapsto \expo_{i\omega t}x\ ,$$
where $x\in X$ and $\omega \in \R $. Of course, from the explicit formula obtained in Theorem \ref{Phibij} and from the evolution under the cubic Szeg\H{o} flow, for any $u_0\in \mathcal V(d)$, the solution $u(t)$ is quasi-periodic. This is also  a consequence of the results of \cite{GG4}.
It remains to consider data in $H^{1/2}_+$ corresponding to infinite rank Hankel operators. We are going to use Bochner's criterion, see chapters 1, 2 of \cite{LZ}, namely that $f\in C(\R ,X)$ is almost periodic if and only if it is bounded and the set of functions
$$f_h:t\in \R \longmapsto f(t+h)\in X\ ,\  h\in \R \ ,$$
is relatively compact in the space of bounded continuous functions valued in $X$.

Let $u_0\in \mathcal V^{(2)}_{(d_r)_{r\ge 1}}$. Set 
$$\Phi (u_0)=((s_r)_{r\ge 1}, (\Psi _r)_{r\ge 1})\ .$$
Then, from Theorem \ref{evolszego}, 
$$\Phi (u(t))= ((s_r)_{r\ge 1}, (\expo_{i(-1)^rs_r^2t}\Psi _r)_{r\ge 1})\ .$$
By Theorem \ref{HilbertSchmidt}, it is enough to prove that the set of functions
$$t\in \R \longmapsto \Phi (u(t+h))\in \Omega ^{(2)}_\infty \times \prod _{r=1}^\infty \mathcal B_{d_r}$$
is relatively compact in $C(\R , \Omega ^{(2)}_\infty \times \prod _{r=1}^\infty \mathcal B_{d_r})$. This is equivalent to the relative compactness
of the family $(\expo_{i(-1)^rs_r^2h})_{r\ge 1}$ in $(\S ^1)^\infty $, $h\in \R $, which is trivial.

\section{Evolution under the Szeg\H{o} hierarchy} \label{szegohier}

The Szeg\H{o} hierarchy was introduced in \cite{GG1} and used in \cite{GG2} and \cite{GG4}. 
In \cite{GG2}, it was used to identify the symplectic form on the generic part of $\mathcal V(d)$.
Similarly, our purpose in this section is  to establish preliminary formulae, towards the identification of the symplectic form 
on $\mathcal V_{(d_1,\dots,d_n)}$ in the next section.
\s
For the convenience of the reader, we recall the main properties of the hierarchy. For $y>0$ and $u\in H^{\frac 12}_+$, we set
$$J^y(u)= ((I+yH_u^2)^{-1}(1)\vert 1)\ .$$
Notice that the connection with the Szeg\H{o} equation is made by
$$E(u)=\frac 14(\pa _y^2J^y_{\vert y=0}-(\pa _yJ^y_{\vert y=0})^2)\ .$$
Thanks to formula (\ref{J}) in Appendix A, $J^y(u)$ is a function of the singular values $s_r(u)$.
For every $s>\frac 12$, $J^y $ is a smooth real valued function on $H^s_+$, and its Hamiltonian vector field  is given by
$$X_{J^y}(u)=2iy w^yH_uw^y\ ,\ w^y:=(I+yH_u^2)^{-1}(1)\ ,$$
which is a Lipschitz vector field on bounded subsets of $H^s_+$. 
By the Cauchy--Lipschitz theorem, the evolution equation 
\begin{equation}\label{hierarchy}
\dot u=X_{J^y}(u)
\end{equation}
admits local in time solutions for every initial data in $H^s_+$ for $s>1$, and the lifetime is bounded from below if the 
data are bounded in $H^s_+$. We recall that this evolution equation admits a Lax pair structure (\cite{GG4}). 
\begin{theorem}\label{Laxhier}
For every $u\in H^s_+$, we have
\begin{eqnarray*}
H_{iX_{J^y}(u)}&=&H_u F_u^y+F_u^yH_u\ ,\\
K_{iX_{J^y}(u)}&=&K_uG_u^y+G_u^yK_u\ ,\\
 G_u^y(h)&:=&-yw^y\, \Pi (\overline {w^y}\, h)+y^2H_uw^y\, \Pi (\overline {H_uw^y}\, h)\ ,\\
 F_u^y(h)&:=&G_u^y(h)-y^2(h\vert H_uw^y)H_uw^y\ .
\end{eqnarray*}
If $u\in C^\infty (\mathcal I ,H^s_+)$ is a solution of equation (\ref{hierarchy}) on a time interval $\mathcal I$, then
\begin{eqnarray*}
\frac{dH_u}{dt}&=&[B_u^y,H_u]\ ,\ \frac{dK_u}{dt}=[C_u^y,K_u]\ ,\\
B_u^y&=&-iF_u^y\ ,\ C_u^y=-iG_u^y\ .
\end{eqnarray*}
\end{theorem}
In particular, $\Sigma _H(u_0)=\Sigma _H(u(t))$ and $\Sigma _K(u_0)=\Sigma _K(u(t))$ for every $t$, therefore $J^y(u(t))$ is a constant $J^y$. We now state the main result of this section.
\begin{theorem}\label{hierarchyevol}
Let $u_0\in H^s_+\ ,\ s>1,$  with $$\Phi(u_0)=((s_r),(\Psi_r)).$$
The solution of  $$\dot u=X_{J^y}(u)\ ,\ u(0)=u_0\ ,$$  is characterized by
$$\Phi(u(t))=((s_r),(\expo_{i\omega _rt}\Psi _r))\ ,\ \omega _r:=(-1)^{r-1}\frac {2yJ^y}{1+ys_r^2}\ .$$
\end{theorem}
\begin{proof}
Let $\rh \in \Sigma _H(u_0)$. Denote by  $u_\rho$ the orthogonal projection of $u$ on $E_u(\rho)$.
Hence, $u_\rho=1\!{\rm l}_{\{ \rho^2\}}(H_u^2)(u)$. Let us differentiate this equation with respect to time. 
We get from the Lax pair structure
\begin{eqnarray*}
\frac{d u_\rho}{dt}&=& [B_u^y,1\!{\rm l}_{\{ \rho^2\}}(H_u^2)](u)+1\!{\rm l}_{\{ \rho^2\}}(H_u^2)[B_u^y,H_u](1)\\
&=& B_u^y(u_\rho)-1\!{\rm l}_{\{ \rho^2\}}(H_u^2)(H_u(B_u^y(1)))\ .
\end{eqnarray*}
Since $B_u^y(1)=iyJ^yw^y$, and since  $1\!{\rm l}_{\{ \rho^2\}}(H_u^2)(H_uw^y)=\frac 1{1+y\rho^2}u_\rho$, we get

\begin{equation}\label{deriveurhoJ}
\frac{d u_\rho}{dt}=B_u^y(u_\rho)+i\frac{yJ^y}{1+y\rho^2}u_\rho\ .
\end{equation}

On the other hand, differentiating the equation
$$\rho u_\rho=\Psi  H_u(u_\rho)$$
one obtains
$$
\rho \frac{d u_\rho}{dt}=\dot \Psi H_u(u_\rho)+\Psi\left([B_u^y,H_u](u_\rho)+H_u\left( \frac{d u_\rho}{dt}\right)\right)$$
Hence, using the expression (\ref{deriveurhoJ}), we get
$$\rho\left (B_u^y(u_\rho)+i\frac{yJ^y}{1+y\rho^2}u_\rho\right )=\left (\dot \Psi-i\frac{yJ^y}{1+y\rho^2}\Psi \right )H_u(u_\rh )+\Psi B_u^yH_u(u_\rho)\ ,$$
hence
$$[B_u^y,\Psi]H_u(u_\rho)=\left (\dot \Psi-2i\frac {yJ^y}{1+y\rho^2}\Psi \right )H_u(u_\rh ).$$
It remains to prove  that the left hand side of this equality is zero.
We first show that, for any $p\in\D$ such that $\chi_p$ is a factor of $\chi$, for every $e\in E_u(\rho)$ such that $\chi _pe\in E_u(\rh )$,
$[B_u^y,\chi_p](e)=0$. We write
\begin{eqnarray*}
i[B_u^y,\chi_p](e)&=&-yw^y\,( \Pi (\overline {w^y}\, \chi_pe)-\chi_p \Pi (\overline {w^y}\, e))\\
&+&y^2H_uw^y\, (\Pi (\overline {H_uw^y}\, \chi_p e)-\chi_p\Pi (\overline {H_uw^y}\,  e))\\
&-&y^2((\chi_p e\vert H_uw^y)H_uw^y-\chi_p(e\vert H_uw^y)H_uw^y) )\\
\end{eqnarray*}
We already used that, for any function $f\in L^2$, $\Pi(\chi_pf)-\chi_p\Pi(f)$ is proportional to $\frac 1{1-\overline pz}$. 
Hence, we obtain
\begin{eqnarray*}
i[B_u^y,\chi_p](e)&=&-yw^y \frac{c}{1-\overline pz}+y^2H_uw^y\frac{\tilde c}{1-\overline pz}\\
&-&y^2((\chi_p e\vert H_uw^y)H_uw^y-\chi_p(e\vert H_uw^y)H_uw^y) )\\
\end{eqnarray*}
with 
$$c=( \Pi (\overline {w^y}\, \chi_pe)-\chi_p \Pi (\overline {w^y}\, e)\vert 1)=( \chi_pe\vert w^y)-(\chi_p\vert 1)(e\vert w^y)$$ and 
$$\tilde c=(\Pi (\overline {H_uw^y}\, \chi_p e)-\chi_p\Pi (\overline {H_uw^y}\,  e)\vert 1)=( \chi_pe\vert H_u(w^y))-(\chi_p\vert 1)(e\vert H_u(w^y))\ .$$
Now,  for any $v\in E_u(\rho)$
$$( v\vert w^y)=( v\vert 1\!{\rm l}_{\{ \rho^2\}}(H_u^2)(w^y))=\frac 1{1+y\rho^2}( v\vert 1)$$
hence $c=0$.
On the other hand,
\begin{eqnarray*}
(v\vert H_u w^y)&=&(v\vert 1\!{\rm l}_{\{ \rho^2\}}(H_u^2)(H_u(w^y)))=\frac 1{1+y\rho^2}( v\vert u_\rho)\\
&=&\frac 1{1+y\rho^2}(v\vert H_u(1))=\frac 1{1+y\rho^2}(1\vert H_u(v)) \ .
\end{eqnarray*}
We infer
$$i[B_u^y,\chi_p](e)=C(z)\frac 1{1+y\rho^2}y^2H_uw^y$$
where 
\begin{eqnarray*}
C(z)&=&\frac 1{1-\overline pz}\left ((1\vert H_u(\chi_p e))-(\chi_p\vert 1)(1\vert H_u(e)\right)
-(1\vert H_u(\chi_p e)+\chi_p(1\vert H_u(e))\\
&=&(1\vert H_u(\chi_p e))(\frac 1{1-\overline pz}-1)+(1\vert H_u(e))(\chi_p+\frac{p}{1-\overline pz})\\
&=&\frac z{1-\overline pz}(\overline p(1\vert H_u(\chi_p e))+(1\vert H_u(e)))\ .
\end{eqnarray*}
We claim that $H_u(e)=\chi_pH_u(\chi_p e)$. Indeed, from the assumption $e\in E_u(\rh )$ and $\chi _pe\in E_u(\rh )$, we can write
 $e=fH_u(u_\rh )$ with $\Pi (\Psi \overline f)=
\Psi \overline f$ and $\Pi (\Psi \overline {\chi _pf})=
\Psi \overline {\chi _pf}$. From Lemma \ref{crucialHuGeneral}, we infer
$$H_u(\chi _pe)=\rh \Psi \overline {\chi _pf}H_u(u_\rh )\ ,\ H_u(e)=\rh \Psi \overline fH_u(u_\rh )\ .$$
This proves the claim. Since $(1\vert \chi_p)=-\overline p$, we conclude that $C(z)=0$. Hence $[B^y_u,\chi_p](e)=0$. Arguing as in the previous section, we conclude that $[B_u,\chi]H_u(u_\rho)=0$.
\s
It remains to consider the  other eigenvalues. Let $\sigma\in \Sigma _K(u_0)$.  Denote by $u'_\sigma$ the orthogonal projection of $u$ on $F_u(\sigma )$. We compute  the derivative of $u'_\sigma=1\!{\rm l}_{\{\sigma^2\}}(K_u^2)(u)$ as before.
From the Lax pair formula, we get
\begin{eqnarray*}
\frac{d u'_\sigma}{dt}&=& [C_u^y,1\!{\rm l}_{\{\sigma^2\}}(K_u^2)](u)+1\!{\rm l}_{\{\sigma^2\}}(K_u^2)[B^y_u,H_u](1)\\
&=& C_u^y(u'_\sigma)+1\!{\rm l}_{\{\sigma^2\}}(K_u^2)(B_u^y(u)-C_u^y(u)-H_u(B_u^y(1)))\\
&=&C_u^y(u'_\sigma)+1\!{\rm l}_{\{\sigma^2\}}(K_u^2)(iy^2(u\vert H_uw^y)H_uw^y+iyJ^y H_u w^y)\\
&=&C_u^y(u'_\sigma)+iy1\!{\rm l}_{\{\sigma^2\}}(K_u^2)H_uw^y
\end{eqnarray*}
since
$(B_u^y-C_u^y)(h)=iy^2(h\vert H_uw^y)H_uw^y$ and $-yH^2_uw^y=w^y-1$ so that $(u\vert-yH_uw^y)=(-yH_u^2w^y\vert 1)=J^y-1$.

We claim that 
\begin{equation}\label{ProjSigma}
1\!{\rm l}_{\{\sigma^2\}}(K_u^2)(H_uw^y)=\frac{J^y}{1+y\sigma^2} u'_\sigma\ .
\end{equation}
Using  $K_u^2=H_u^2-(\cdot\vert u)u$ one gets, for any $f\in L^2_+$
\begin{equation}\label{sandrine}
(I+yH_u^2)^{-1}f=(I+yK_u^2)^{-1}f-y((I+yH_u^2)^{-1}f\vert u)(I+yK_u^2)^{-1}u\ .
\end{equation}
Applying formula (\ref{sandrine}) to $f=u$, we get
$$
H_u w^y=(I+yH_u^2)^{-1}(u)=(I+yK_u^2)^{-1}(u)-y((I+yH_u^2)^{-1}(u)\vert u)(I+yK_u^2)^{-1}(u)\ ,$$ hence
\begin{equation}\label{Hu(w)}
H_u w^y=J^y(I+yK_u^2)^{-1}(u) \ .
\end{equation}
Formula (\ref{ProjSigma}) follows by taking the orthogonal projection on $F_u(\sigma )$.
Using Formula (\ref{ProjSigma}), we get
\begin{equation}\label{deriveUsigma}
\frac{d u'_\sigma}{dt}=C_u^y(u'_\sigma)+iy\frac{J^y}{1+y\sigma^2} u'_\sigma\ .
\end{equation}
On the other hand, differentiating the equation
$$K_u(u'_\sigma)=\sigma \Psi  u'_\sigma $$ 
one obtains
$$
[C_u^y,K_u](u'_\sigma)+K_u\left(\frac{d u'_\sigma}{dt}\right)=\sigma \dot \Psi u'_\sigma+\sigma \Psi  \frac{d u'_\sigma}{dt}.$$
From identity (\ref{Hu(w)}), we use the expression of $\frac{d u'_\sigma}{dt}$ 
obtained above to get
$$\left (\dot \Psi+2i\frac{yJ^y}{1+\sigma^2 y}\Psi \right )u'_\sigma=\sigma[C_u^y,\Psi](u'_\sigma)\ .$$
The result follows once we prove that $[C_u^y,\Psi](u'_\sigma)=0$.

From the arguments developed before, it is sufficient to prove that 
$[C_u^y, \chi_p](f)=0$ for any $f\in F_u(\sigma )$ such that $\chi _pf\in F_u(\sigma )$. 
As before
$$[C_u^y, \chi_p](f)=i\frac{c}{1-\overline pz} yw^y-iy^2H_uw^y\frac{\tilde c}{1-\overline pz}$$
where 
$$c=(( \chi_p-( \chi_p\vert 1))f\vert w^y)$$ and 
$$\tilde c=((\chi_p -( \chi_p\vert 1))f\vert H_u w^y).$$
Notice that $w^y=1-yH_uw^y$, hence $c=-y\tilde c$. 
Let us first prove that $\tilde c=0$. Using formula (\ref{Hu(w)}), 
\begin{eqnarray*}
\tilde c&=&( \chi_p f\vert 1\!{\rm l}_{\{\sigma^2\}}(K_u^2)H_u w^y)-( \chi_p\vert 1)(f\vert  1\!{\rm l}_{\{\sigma^2\}}(K_u^2)H_u w^y)\\
&=&\frac{J^y}{1+y\sigma^2}(( \chi_p -( \chi_p\vert 1))f\vert u)= 0\ ,
\end{eqnarray*}
since, as we already observed at the end of the proof of Theorem \ref{evolszegotexte},
 $$F_u(\sigma )\cap 1^\perp =E_u(\sigma )=F_u(\sigma )\cap u^\perp \ . $$
This completes the proof.
\end{proof}
We close this section by stating a corollary which will be useful for describing the symplectic form on $\mathcal V_{(d_1,\dots,d_n)}$.
\begin{corollary}\label{XJy}
On $\mathcal V_{(d_1,\dots,d_n)}$, we have 
\begin{equation}\label{decXJy}
X_{J^y}=\sum _{r=1}^n (-1)^{r}\frac{2yJ^y}{1+ys_r^2}\frac{\partial}{\partial \psi _r}\ .
\end{equation}
The vector fields $X_{J^y}, y\in \R _+,$ generate an integrable sub-bundle of rank  $n$ of the tangent bundle of  $\mathcal V_{(d_1,\dots,d_n)}$. The  leaves of the corresponding foliation are the isotropic tori 
$$\mathcal T ((s_1,\dots ,s_n),(\Psi _1,\dots ,\Psi _n)):=\Phi ^{-1}\left (\{ (s_1,\dots ,s_n)\} \times \S^1\Psi _1\times \dots \times \S^1\Psi _n  \right )\ ,$$
where $(s_1,\dots ,s_n)\in \Omega $ and $(\Psi _1,\dots ,\Psi _n)\in \mathcal B^\sharp _{d_1}\times \dots \times \mathcal B^\sharp _{d_n}$ are given.
\end{corollary}
\begin{proof}
For every $y\in \R _+$,  Theorem  \ref{hierarchyevol} can be rephrased as the following identities for Lie derivatives along $X_{J^y}$.
$$X_{J^y}(s_r)=0\ ,\ X_{J^y}(\chi _r)=0\ ,\ X_{J^y}(\psi _r)=(-1)^{r}\frac{2yJ^y}{1+ys_r^2}\ ,\ r=1,\dots ,n\ .$$
This implies identity (\ref{decXJy})  on $\mathcal V_{(d_1,\dots ,d_n)}$.
Given $n$ positive numbers $y_1>\dots >y_n$, the matrix 
$$\left (\frac{1}{1+y_\ell s_r^2}\right )_{1\le \ell ,r\le n}$$
is invertible. This implies that, for every $u\in \mathcal V_{(d_1,\dots,d_n)}$, the vector subspace of $T_u\mathcal V_{(d_1,\dots,d_n)}$ spanned by 
the $X_{J^y}(u), y\in \R _+$ is exactly
$${\rm span} \left (\frac{\partial}{\partial \psi _r}, r=1,\dots ,n\right )\ .$$
The integrability follows, as well as the identification of the leaves, while the isotropy of the tori comes from the identity
$$\{ J^y, J^{y'}\} =0$$
which was proved in \cite{GG1} and is also a consequence of identity \ref{J(x)Appendix} and of the conservation of the $s_r$'s 
along the Hamiltonian curves of $J^y$, as stated in Theorem \ref{hierarchyevol}. 
\end{proof}
 In the next section, we identify the tori  $\mathcal T ((s_r),(\Psi _r))$ above as classes of some special unitary equivalence for the pair of operators $(H_u,K_u)$. 
 
 \section{Invariant tori of the Szeg\H{o} hierarchy and unitary equivalence of  pairs of Hankel operators} \label{invtori}
In this section, we identify the sets of symbols $u\in VMO_+\setminus \{ 0\} $ having the same list of singular values $(s_r)$ and the same list $(\chi _r)$ of monic Blaschke products, for the pair $(H_u,K_u)$. In view of Theorem \ref{Phitopodif}, these sets are tori. Moreover, $VMO_+\setminus \{ 0\} $ is the disjoint union of these tori, and, from sections \ref{szegoflow} and \ref{szegohier}, the Hamilton flows of the Szeg\H{o} hierarchy act on them. We prove that they are classes of some specific unitary equivalence between the pairs $(H_u,K_u)$, which we now define.
\begin{definition}\label{equivalence}
Given $u,\tilde u\in VMO_+\setminus \{ 0\} $, we set $u\sim \tilde u$ if there exist unitary operators $U, V$ on $L^2_+$ such that
$$H_{\tilde u}=UH_uU^*\ ,\ K_{\tilde u}=VK_uV^*\ ,$$
and there exists  a Borel function $F:\R _+\rightarrow \S^1$ such that
$$U(u)=F(H_{\tilde u}^2)\tilde u\ ,\ V(u)=F(K_{\tilde u}^2)\tilde u\ ,\ U^*V=\overline F(H_u^2)F(K_u^2)\ .$$
\end{definition}
It is easy to check that the above definition gives rise to an equivalence relation.
\begin{theorem}\label{specialisospectral}
Given $u,\tilde u\in VMO_+\setminus \{ 0\} $, the following assertions are equivalent.
\begin{enumerate}
\item $u\sim \tilde u$.
\item $\forall r\ge 1, s_r(u)=s_r(\tilde u) \ \textrm{and}\  \exists \gamma _r\in \T : \Psi_r(\tilde u)=\expo_{i\gamma _r}\Psi _r(u)\ .$
\end{enumerate}
\end{theorem}
\begin{proof}
Assume that (1) holds. Then $H_{\tilde u}^2$ is unitarily equivalent to $H_u^2$, and $K_{\tilde u}^2$ is unitarily equivalent to $K_u^2$.
This clearly implies $\Sigma _H(\tilde u)=\Sigma _H(u)$ and $\Sigma _K(\tilde u)=\Sigma _K(u)$, so that $s_r(\tilde u)=s_r(u)$ for every $r$.
Let us show that, for every $r$, $\Psi _r(u)$ and $\Psi _r(\tilde u)$ only differ by a phase factor. Of course the only  cases to be addressed are $d_r\ge 1$. We start with $r=2j-1$. From the hypothesis, we infer
$$U(u_j)=U({\bf 1}_{\{ \rho _j^2\} }(H_u^2)(u))={\bf 1}_{\{ \rho _j^2\} }(H_{\tilde u}^2)(U(u))=F(\rho _j^2)\tilde u_j\ ,$$
and, consequently,
\begin{equation}\label{UH_u}
U(H_u(u_j))=H_{\tilde u}(U(u_j))=\overline F(\rho _j^2)H_{\tilde u}(\tilde u_j)\ .
\end{equation}
Next we take advantage of the identity
$$U^*V=\overline F(H_u^2)F(K_u^2)\ ,$$
by evaluating $U^*S^*U$ on the closed range of $H_u$. We compute
\beno
U^*S^*UH_u&=&U^*S^*H_{\tilde u}U=U^*K_{\tilde u}U=U^*VK_uV^*U\\
&=&\overline F(H_u^2)F(K_u^2)K_u\overline F(K_u^2)F(H_u^2)=\overline F(H_u^2)F(K_u^2)^2K_uF(H_u^2)\\
&=&\overline F(H_u^2)F(K_u^2)^2S^*\overline F(H_u^2)H_u\ ,
\eeno
and we conclude, on $\overline {{\rm Ran}(H_u)}$, 
\begin{equation}\label{U*S*U}
U^*S^*U=\overline F(H_u^2)F(K_u^2)^2S^*\overline F(H_u^2)\ .
\end{equation}
For simplicity, set $D:=D_{2j-1}$ and $d:=d_{2j-1}$. Recall from Proposition \ref{action} that a basis of $E_u(\rho _j)$ is 
$$\left (\frac{z^a}{D}H_u(u_j)\ ,\ a=0,\dots ,d\right ) ,$$
and a basis of $F_u(\rho _j)=E_u(\rho _j)\cap u^\perp $ is 
$$\left (\frac{z^b}{D}H_u(u_j)\ ,\ b=0,\dots ,d-1\right ) .$$
For $a=1,\dots ,d_{2j-1}$, we infer
$$U^*S^*U\left (\frac{z^a}{D}H_u(u_j)\right )=\frac{z^{a-1}}{D}H_u(u_j)\ ,$$
or
$$U\left (\frac{z^a}{D}H_u(u_j)\right )=(S^*)^{d-a}U\left (\frac{z^d}{D}H_u(u_j)\right )\ ,\ a=0,\dots ,d\ .$$
This implies, for $a=0,\dots ,d-1$, that the right hand side belongs to $F_{\tilde u}(\rho _j)$. On the other hand, if $P\in \C [z]$ has degree at most $d$,
one easily checks that
$$S^*\left (\frac{P}{\tilde D}H_{\tilde u}(\tilde u_j)\right )=P(0)K_{\tilde u}(\tilde u_j)+ R\ ,\ R\in F_{\tilde u}(\rho _j)\ .$$
Notice that the right hand side belongs to $F_{\tilde u}(\rho _j)$ if and only if $K_{\tilde u}(\tilde u_j)\in F_{\tilde u}(\rho _j)$ or $P(0)=0$. Assume for a while that $K_{\tilde u}(\tilde u_j)$ does not belong to $F_{\tilde u}(\rho _j)$. Then, writing
$$U\left (\frac{z^d}{D}H_u(u_j)\right )=\frac{P}{\tilde D}H_{\tilde u}(\tilde u_j)\ ,$$
and using that, for $a=0,\dots ,d-1$, 
$$(S^*)^{d-a}\left (\frac{P}{\tilde D}H_{\tilde u}(\tilde u_j)\right )\in F_{\tilde u}(\rho _j)\ ,$$
we infer $P(0)=0$, and, by iterating this argument,  that $P$ is divisible by $z^d$, in other words,
$$U\left (\frac{z^d}{D}H_u(u_j)\right )=c\frac{z^d}{\tilde D}H_{\tilde u}(\tilde u_j)\ ,$$
for some $c\in \C $, and conclude
$$U\left (\frac{z^a}{D}H_u(u_j)\right )=c\frac{z^a}{\tilde D}H_{\tilde u}(\tilde u_j)\ ,\ a=0,\dots ,d\ .$$
Comparing to formula (\ref{UH_u}) for $U(H_u(u_j))$, we obtain
$$ cD(z)=\overline F(\rho _j^2)\tilde D(z)\ .$$
Since $D(0)=1=\tilde D(0)$, we conclude $c=\overline F(\rho _j^2)$, $D=\tilde D$, and finally $$\Psi _{2j-1}(\tilde u)=\overline F(\rho _j^2)^2\Psi _{2j-1}(u)\ .$$
We now turn to study the special case $K_{\tilde u}(\tilde u_j)\in F_{\tilde u}(\rho _j)$. This reads
$$0=(K_{\tilde u}^2-\rho _j^2I)K_{\tilde u}(\tilde u_j)=K_{\tilde u}((H_{\tilde u}^2-\rho _j^2I)\tilde u_j-\Vert \tilde u_j\Vert ^2\tilde u)=-\Vert \tilde u_j\Vert ^2K_{\tilde u}(\tilde u)\ .$$
In other words, this imposes $K_{\tilde u}(\tilde u)=0$, or $\tilde u=\rho \tilde \Psi $, where $\tilde \Psi $ is a Blaschke product of degree $d$. Making $V^*$ act on the identity $K_{\tilde u}(\tilde u)=0$, we similarly conclude $ u=\rho  \Psi $, where $\Psi $ is a Blaschke product of degree $d$, so what we have to check is that $\Psi $ and $\tilde \Psi $ only differ by a phase factor. In this case, $S^*$ sends $E_u(\rho )={\rm Ran}(H_u)$ into $F_u(\rho )$, so that (\ref{U*S*U}) becomes, on ${\rm Ran}(H_u)$,
$$U^*S^*U=S^*\ .$$
In other words, the actions of $S^*$ on $W:={\rm span}\left (\frac{z^a}{D}, a=0,\dots ,d\right )$ and on $\tilde W:={\rm span}\left (\frac{z^a}{\tilde D}, a=0,\dots ,d\right )$ are conjugated. Writing
$$D(z)=\prod _{p\in \mathcal P}(1-\overline pz)^{m_p}\ ,$$
where $\mathcal P$ is a finite subset of $\D \setminus \{ 0\}$, and $m_p$ are positive integers, and using the elementary identities 
$$(S^* -\overline pI)\left(\frac 1{(1-\overline pz)^k}\right )=S^* \left (\frac 1{(1-\overline pz)^{k-1}}\right )\ ,$$
one easily checks that the eigenvalues of $S^*$ on $W$ are precisely the $\overline p$'s, for $p\in \mathcal P$, and $0$,  with the corresponding algebraic multiplicities $m_p$ and 
$$m_0=1+d-\sum _{p\in \mathcal P}m_p\ .$$
We conclude that $D=\tilde D$, whence the claim.
\s
Next, we study the case $r=2k$. Then
$$V(u'_k)=F(\sigma _k^2)\tilde u'_k\ ,\ V(K_u(u'_k))=\overline F(\sigma _k^2)K_{\tilde u}(\tilde u'_k)\ .$$
Denote by $P_u$ the orthogonal projector onto $\overline {{\rm Ran}(H_u)}$, and compute
\beno
V^*P_{\tilde u}SVK_u&=&V^*P_{\tilde u}SK_{\tilde u}V=V^*(H_{\tilde u}-(\tilde u\vert\, .\, )P_{\tilde u}(1)) V\\
&=&V^*U(H_u-(U^*(\tilde u)\vert \, .\, )U^*(P_{\tilde u}(1)))U^*V\\
&=&\overline F(K_u^2)F(H_u^2)(H_u-(\overline F(H_u^2)u\vert \, .\, )F(H_u^2)P_u(1))\overline F(H_u^2)F(K_u^2)\\
&=&\overline F(K_u^2)F(H_u^2)^2(H_u-(u\vert \, .\, )P_u(1))F(K_u^2)\\
&=& \overline F(K_u^2)F(H_u^2)^2P_uSK_uF(K_u^2)= \overline F(K_u^2)F(H_u^2)^2P_uS\overline F(K_u^2)K_u\ ,
\eeno
so that, on $\overline {{\rm Ran}(K_u)}$, 
\begin{equation}\label{V*PSV}
V^*P_{\tilde u}SV=\overline F(K_u^2)F(H_u^2)^2P_uS\overline F(K_u^2)\ .
\end{equation}
For simplicity again, set $D:=D_{2k}$ and $d:=d_{2k}$. Recall from proposition \ref{action} that a basis of $F_u(\sigma _k)$ is 
$$\left (\frac{z^a}{D}u'_k\ ,\ a=0,\dots ,d\right ) ,$$
and a basis of $E_u(\sigma _k)=F_u(\sigma _k)\cap u^\perp $ is 
$$\left (\frac{z^a}{D}u'_k\ ,\ a=1,\dots ,d\right ) \ .$$
Applying identity (\ref{V*PSV}) to $\frac{z^a}{D}u'_k$ for $a=0,\dots ,d-1$, we infer
$$V\left (\frac{z^a}{D}u'_k\right )=(P_{\tilde u}S)^a V\left (\frac 1{ D} u'_k\right )\ , \ a=0,\dots ,d\ .$$
In particular, the right hand side belongs to $E_{\tilde u}(\sigma _k)$ for $a=1,\dots ,d$. On the other hand, if $Q\in \C [z]$ has degree at most $d$, 
$$P_{\tilde u}S\left (\frac{Q}{\tilde D}\tilde u'_k\right )= \gamma P_{\tilde u}SK_{\tilde u}(\tilde u'_k) +R\ ,\ R\in E_{\tilde u}(\sigma _k)\ ,$$
where $\gamma \sigma _k \expo_{-i\tilde \psi _{2k}}$ is the coefficient of $z^d$ in $Q$.  Therefore the left hand side belongs to $E_{\tilde u}(\sigma _k)$ if and only if
\beno
0&=&\gamma (H_{\tilde u}^2-\sigma _k^2I)P_{\tilde u}SK_{\tilde u}(\tilde u'_k)\\
&=&\gamma (H_{\tilde u}^2-\sigma _k^2I)(H_{\tilde u}(\tilde u'_k)-\Vert \tilde u'_k\Vert ^2P_{\tilde u}(1))\\
&=&\gamma \sigma _k^2\Vert \tilde u'_k\Vert ^2P_{\tilde u}(1)\ ,
\eeno
which is impossible. We conclude that $\gamma =0$, which means that the degree of $Q$ is at most $d-1$. Iterating this argument, we infer 
$$V\left (\frac 1{ D} u'_k\right )=c\frac{1}{\tilde D}\tilde u'_k\ ,$$
for some $c\in \C $, and finally
$$V\left (\frac{z^a}{D}u'_k\right )=c\frac{z^a}{\tilde D}\tilde u'_k\ ,\ a=0,\dots d\ .$$
Comparing to the above formula for $V(u'_k)$, we obtain 
$$cD=F(\sigma _k^2)\tilde D\  ,$$
thus, since $D(0)=1=\tilde D(0)$, we have $D=\tilde D$, $c=F(\sigma _k^2)$, and finally $\Psi _{2k}(\tilde u)=F(\sigma _k^2)^2\Psi _{2k}(u)\ .$
\s
Assume that (2) holds. Define $F:\Sigma _H(u)\cup \Sigma _K(u)\rightarrow \S^1$ by
$$F(\rho _j^2)=\expo_{-i\frac{\gamma _{2j-1}}2}\ ;\ F(\sigma _k^2)=\expo_{i\frac{\gamma _{2k}}2}\ ,$$
and, if necessary,  we define $F(0)$ to be any complex number of modulus $1$. Next we define $U$ on the closed range of $H_u$, which is the closed orthogonal sum of $E_u(s_r)$. Thus we just have to define $$U:E_u(s_r)\rightarrow E_{\tilde u}(s_r).$$ 
\s
If $r=2j-1$, we set
\begin{equation}\label{Uj}
U\left (\frac{z^a}{D_{2j-1}}H_u(u_j)\right )=\overline F(\rho _j^2) \frac{z^a}{D_{2j-1}}H_{\tilde u}(\tilde u_j)\ ,\ a=0,\dots ,d_{2j-1}\ .
\end{equation}
If $r=2k$ and $d_{2k}\ge 1$, we set
\begin{equation}\label{Uk}
U\left (\frac{z^b}{D_{2k}}u'_k\right )=F(\sigma _k^2) \frac{z^b}{D_{2k}}\tilde u'_k\ ,\ b=1,\dots ,d_{2k}\ .
\end{equation}
Using (\ref{Uj}) we obtain
\beno
U(u_j)&=&\frac 1{\rho _j}U(\Psi _{2j-1}(u)H_u(u_j))=\frac 1{\rho _j}\overline F(\rho _j^2)\Psi _{2j-1}(u)H_{\tilde u}(\tilde u_j)\\
&=&\frac 1{\rho _j}F(\rho _j^2)\Psi _{2j-1}(\tilde u)H_{\tilde u}(\tilde u_j)=F(\rho _j^2)\tilde u_j\ . 
\eeno
Consequently, we get
$$U(u)=\sum _jU(u_j)=\sum _jF(\rho _j^2)\tilde u_j=F(H_{\tilde u}^2)\tilde u\ .$$
A similar argument combined to Proposition \ref{action} leads to 
$$UH_u=H_{\tilde u}U.$$
Next, we prove that $U$ is unitary. It is enough to prove that every map $U:E_u(s_r)\rightarrow E_{\tilde u}(s_r)$ is unitary, or that the Gram matrix of a basis of $E_u(s_r)$ is equal to the Gram matrix of its image. We first deal with $r=2j-1$. Equivalently, we prove that, for $a,b=0,\dots ,d_{2j-1}-1$,
$$\left (\frac{z^a}{D_{2j-1}}H_u(u_j)\vert \frac{z^b}{D_{2j-1}}H_u(u_j)\right )=\left (\frac{z^a}{D_{2j-1}}H_{\tilde u}({\tilde u}_j)\vert \frac{z^b}{D_{2j-1}}H_{\tilde u}({\tilde u}_j)\right )\ .$$
We set 
$$\zeta _{a-b}:=\left (\frac{z^a}{D_{2j-1}}H_u(u_j)\vert \frac{z^b}{D_{2j-1}}H_u(u_j)\right ), \ a,b=0,\dots ,d_{2j-1}-1\ ,$$
and we notice that $\zeta _{-k}=\overline \zeta _k\ ,\ k=-d_{2j-1},\dots ,d_{2j-1}\ .$ We drop the subscript $2j-1$ for simplicity and we set 
$$D(z):=1+\overline a_1z+\dots +\overline a_d z^d.$$
As $\Psi H_u(u_j)$ is orthogonal to $\frac {z ^a}{D} H_u(u_j)$ for $a=0,\dots, d-1$, and $\Vert H_u(u_j)\Vert^2 =\rho_j^2\tau_j^2$, we obtain the system
\begin{equation}\label{system}
\left\{\begin{array}{ll}
\zeta_{d-b}+a_1\zeta_{d-b-1}+\dots+a_d\zeta_{-b}=0\ , \ b=0,\dots, d-1\ ,\\
\zeta_0+a_1\zeta_{-1}+\dots +a_d\zeta_{-d}=\rho_j^2\tau_j^2\ .\end{array}\right.
\end{equation}
\begin{lemma}\label{Solsystem}
Let $a_1,\dots a_d$ be complex numbers such that the polynomial
$ z^d+a_1z^{d-1}+\dots+a_d $ has all its roots in $\D$. Then the system (\ref{system}) has at most one solution $\zeta_k$, $k=-d\dots, d$ with $\overline \zeta_k=\zeta_{-k}$.
\end{lemma}
Assume for a while that this lemma is proved. Since $\tau_j^2$  can be expressed in terms of the $(s_r)$'s --- see (\ref{tau}), we infer that 
$U: E_u(\rho_j)\rightarrow E_{\tilde u}(\rho_j)$ is unitary. Similarly, one proves that the Gram matrix of the basis 
$$\frac {z^a}{D_{2k}}u'_k, \ a=0,\dots,d_{2k}$$ of $F_u(\sigma_k)$ only depends on the $(s_r)$'s and on $D_{2k}$. In particular, $$U:E_u(\sigma_k)\rightarrow E_{\tilde u}(\sigma_k)$$ is unitary and finally is unitary from the closed range of $H_u$ onto the closed range of $H_{\tilde u}$.
\s
Next, we construct $V$ on the closed range of $H_u$ which is the orthogonal sum of the $F_u(\sigma)$ for $\sigma \in \Sigma_H\cup \Sigma_K$.
Thus we just have to define $V:F_u(\sigma)\rightarrow E_{\tilde u}(\sigma)$ for $\sigma  \in \Sigma_H\cup \Sigma_K$.
\s
If $r=2j-1$, we set
\begin{equation}\label{Vj}
V\left (\frac{z^a}{D_{2j-1}}H_u(u_j)\right )=\overline F(\rho _j^2) \frac{z^a}{D_{2j-1}}H_{\tilde u}(\tilde u_j)\ ,\ a=1,\dots ,d_{2j-1}\ .
\end{equation}
If $r=2k$ and $d_{2k}\ge 1$, we set
\begin{equation}\label{Vk}
V\left (\frac{z^b}{D_{2k}}u'_k\right )=F(\sigma _k^2) \frac{z^b}{D_{2k}}\tilde u'_k\ ,\ b=0,\dots ,d_{2k}\ .
\end{equation}
Similarly, if $0\in\Sigma_K$, we define $V(u'_0)=F(0)\tilde u'_0$.
Using (\ref{Vk}) we get $V(u'_k)=F(\sigma _k^2)\tilde u'_k.$ 
Consequently, 
$$V(u)=V(u'_0)+\sum _kV(u'_k)=F(K_{\tilde u}^2)\tilde u\ .$$
A similar argument combined with Proposition \ref{action} leads to 
$$VK_u=K_{\tilde u}V.$$
Using again Lemma \ref{Solsystem}, $V$ is unitary from the closed range of $H_u$ onto the closed range of $H_{\tilde u}$.
\s
Now we define $U$ and $V$ on the kernel of $H_u$ which is either $\{0\}$ or an infinite dimensional separable Hilbert space. From Corollary \ref{kernel} of Appendix A, the cancellation of $\ker H_u$ only depends on the $s_r$'s. Therefore, $\ker H_u$ and $\ker H_{\tilde u}$ are isometric. We then define $U=V$ from $\ker H_u$ onto $\ker H_{\tilde u}$ to be any unitary operator.
\s
It remains to prove that $U^*V= \overline F(H_u^2)F(K_u^2)$. On $\ker H_u$, it is trivial since $U^*V=I=\overline F(0) F(0)$.
Similarly, it is trivial on vectors 
$$\frac {z^a}{D_{2k}}u'_k \ ,a=1,\dots , d_{2k}\ .$$
It remains to prove the equality for $u'_0$, $u'_k$. We write 
\beno
U^*V(u'_k)&=&F(\sigma_k^2)U^*(\tilde u'_k)=F(\sigma_k^2)U^*\left(\kappa_k^2\sum_j\frac{\tilde u_j}{\rho_j^2-\sigma_k^2}\right)\\
&=&F(\sigma_k^2)\sum_j\overline F(\rho_j^2)\kappa_k^2 \frac{u_j}{\rho_j^2-\sigma_k^2}= F(\sigma_k^2)\overline F(H_u^2)u'_k\\
&=& \overline F(H_u^2)F(K_u^2)(u'_k)\ .
\eeno 
A similar arguments holds for $U^*V(u'_0)$.
\s 
It remains to prove Lemma \ref{Solsystem}. It is sufficient to prove that the only solution of the homogeneous system
\begin{equation}\label{systemHomogene}
\left\{\begin{array}{ll}
\zeta_{d-b}+a_1\zeta_{d-b-1}+\dots+a_d\zeta_{-b}=0\ , \ b=0,\dots, d-1\ ,\\
\zeta_0+a_1\zeta_{-1}+\dots +a_d\zeta_{-d}=0\ ,\end{array}\right.
\end{equation} with $\overline \zeta_k=\zeta_{-k} \ , k=0,\dots, d$, is the trivial solution $\zeta =0$.

 We proceed by induction on $d$. For $d=1$, the system reads
$$ \left\{\begin{array}{ll}
\zeta_{1}+a_1\zeta_{0}=0\ , \ ,\\
\zeta_0+a_1\overline\zeta_{1}=0 \ .\end{array}\right.$$
Since $|a_1|<1$, this trivially implies $\zeta_0=\zeta_1=0$.
\s
For a general $d$, we plug the expression
$$\zeta_d=-(a_1\zeta_{d-1}+\dots +a_d\zeta_{0})$$ into the last equation. We get
\begin{equation}\label{eq1}
\zeta_0+b_1\overline\zeta_{1}+\dots +b_{d-1}\overline\zeta_{d-1}=0
\end{equation} with 
$$b_k=\frac{a_k-a_d\overline a_{d-k}}{1-|a_d|^2}\ , k=1,\dots, d-1\ .$$
Notice that from Proposition \ref{Od} of Appendix B, $|a_d|<1$ and the polynomial
$z^{d-1}+b_1z^{d-2}+\dots+b_{d-1}$ has all its roots in $\D$.
For $b=1,\dots ,d-1$, we multiply by $a_d$ the conjugate of equation 
$$\zeta_{b}+a_1\zeta_{b-1}+\dots+a_d\zeta_{b-d}=0$$
and substract the result from equation
$$\zeta_{d-b}+a_1\zeta_{d-b-1}+\dots+a_d\zeta_{-b}=0\ .$$
This yields
$$\zeta_{d-b}+b_1\zeta_{d-b-1}+\dots+b_{d-1}\zeta_{1-b}=0\ .$$
Together with Equation (\ref{eq1}), this is exactly the system at order $d-1$ with coefficients $b_1,\dots, b_{d-1}$. By induction, we obtain 
$$\zeta_0=\zeta_1=\dots=\zeta_{d-1}=0$$ and finally $\zeta_d=0$.

This completes the proof.
 
\end{proof}

\section{The symplectic form on $\mathcal V_{(d_1,\dots,d_n)}$}\label{symplectic}

In this section, we prove the last part of Theorem \ref{evolszego}, namely that the symplectic form $\omega$ restricted to $\mathcal V_{(d_1,\dots,d_n)}$ is given by
\begin{equation}\label{omegasurV(d)}
\omega =\sum _{r=1}^n d\left (\frac{s_r^2}{2}\right )\wedge d\psi _r\ .
\end{equation}
Recall that the variable $\psi _r$ is connected to the Blaschke product $\Psi _r$ through the identity
$$\Psi _r=\expo_{-i\psi _r}\chi _r\  ,$$
where $\chi _r$ is a Blaschke product built with a monic polynomial. Given an integer $k$, we denote by $\mathcal B_k ^\sharp $ the submanifold
of $\mathcal B_k$ made with Blaschke products built with monic polynomials of degree $k$.

Let us first point out that  we get the following result as a corollary.
\begin{corollary}
The manifold $\mathcal V_{(d_1,\dots,d_n)}$ is an involutive submanifold of $\mathcal V(d)$, where $$d=2\sum _{r=1}^nd_r+n.$$
Moreover, $\mathcal V_{(d_1,\dots,d_n)}$ is the disjoint union of the symplectic manifolds
$$\mathcal W(\chi _1,\dots ,\chi _n):=\Phi ^{-1}(\Omega _n\times (\S ^1\chi _1\times \dots \times \S ^1\chi_n))\ ,$$
on which 
$$\left (\frac{s_r^2}{2}, \psi _r\right )_{1\le r\le n}$$
are action angle variables for the cubic Szeg\H{o} flow.
\end{corollary}

\begin{proof}
From the definition of an involutive submanifold, one has to prove that, at every point $u$ of  $\mathcal V_{(d_1,\dots,d_n)}$,  the tangent space $T_u\mathcal V_{(d_1,\dots,d_n)}$ contains its orthogonal  relatively to $\omega$. We use an argument of dimension.
Namely, one has 
\beno
\dim _\R (T_u\mathcal V_{(d_1,\dots,d_n)})^\perp&=&\dim _\R T_u\mathcal V(d)-\dim _\R T_u\mathcal V_{(d_1,\dots,d_n)}\\
&=&2d-(2n+2\sum _{r=1}^nd_r)=2\sum _{r=1}^nd_r\ .
\eeno
One the other hand, from equation (\ref{omegasurV(d)}),  the tangent space to the manifold
$$\mathcal F (u):=\Phi ^{-1}\left (\{ (s_r(u))\}\times \prod_{r=1}^n \expo_{-i\psi _r(u)}\mathcal B_{d_r}^\sharp \right )$$
 is clearly a subset of $(T_u\mathcal V_{(d_1,\dots,d_n)})^\perp$. Since its dimension equals $2\sum d_r$, we get the equality and hence the first result. The second result is an immediate consequence of the previous sections.
\end{proof}
\begin{remark}
\s
\begin{itemize}
\item As this is the case for any involutive submanifold of a symplectic manifold,  the subbundle $(T\mathcal V_{(d_1,\dots,d_n)})^\perp $ of 
$T\mathcal V_{(d_1,\dots,d_n)}$ is integrable. The leaves of the corresponding isotropic foliation are the  manifolds $\mathcal F(u)$ above. 
\item The Lagrangian tori of $\mathcal W(\chi _1,\dots ,\chi _n)$ corresponding to the above action angle variables are precisely the tori studied in
section \ref{invtori}.
\end{itemize}
\end{remark}
Now, we prove equality (\ref{omegasurV(d)}).
We first establish the following lemma, as a consequence of Theorem \ref{hierarchyevol}.
\begin{lemma}
On $\mathcal V_{(d_1,\dots ,d_n)}$, 
$$\omega=\sum _{r=1}^n d\left (\frac{s_r^2}{2}\right )\wedge d\psi _r+\tilde \omega\ .$$
where,  for any $1\le r\le n$, 
$$i_{\frac{\partial}{\partial\psi_r}} \tilde \omega=0\ .$$
\end{lemma}
\begin{proof}
Taking the interior product of both sides of identity (\ref{decXJy}) with the restriction of $\omega $ to $\mathcal V_{(d_1,\dots ,d_n)}$, we obtain
$$-d(\log J^y)=\sum _{r=1}^n (-1)^{r}\frac{2y}{1+ys_r^2}i_{\frac{\partial}{\partial \psi _r}}\omega \ .$$
On the other hand, from formula (\ref{J}) in Appendix A,
$$d(\log (J^y))=\sum _{r=1}^n (-1)^r \frac{2y}{1+ys_r^2}d\left (\frac{s_r^2}2\right )\ .$$
Identification of residues in the $y$ variables yields
$$d\left (\frac{s_r^2}2\right )=-i_{\frac{\partial}{\partial \psi _r}}\omega \  ,\ r=1,\dots ,n\ .$$
Since
$$i_{\frac{\partial}{\partial \psi _r}}\left (\sum _{r'=1}^n d\left (\frac{s_{r'}^2}{2}\right )\wedge d\psi _{r'}\right )=-d\left (\frac{s_r^2}2\right )\ ,$$
this completes the proof.
\end{proof}
Since $d\omega=0$, we have  $d\tilde \omega=0$. Combining this information with  $i_{\frac{\partial}{\partial\psi_r}} \tilde \omega=0\ ,$
we conclude that
$$\tilde \omega =\pi ^*\beta \ ,$$
where $\beta $ is a closed $2$-form  on $\Omega _n\times \prod _{r=1}^n \mathcal B_{d_r}^\sharp $, and
$$\pi (u):=((s_r(u))_{1\le r\le n},(\chi _r(u)_{1\le r\le n})\ .$$
In order to prove that $\tilde \omega =0$, it is therefore sufficient to prove that $\tilde \omega =0$ on the submanifold $$\mathcal V_{(d_1,\dots ,d_n),{\rm red}}:=\Phi ^{-1}\left (\Omega _n\times \prod _{r=1}^n \mathcal B_{d_r}^\sharp \right )\  $$
given by the equations $\psi _r=0\ ,\ r=1,\dots ,n$.
\begin{lemma}\label{restriction omega}
The restriction of $\omega $ to $\mathcal V_{(d_1,\dots,d_n),{\rm red}}$ is $0$.
\end{lemma}
\begin{proof} Consider the differential form $\alpha $ of of degree $1$ defined 
$$\langle \alpha (u), h\rangle :={\rm Im}(u\vert h)\ .$$
It is elementary to check that
$$\frac 12 d\alpha =\omega \ ,$$
hence  the statement is consequence of the fact  that the restriction of $\alpha $ to $\mathcal V_{(d_1,\dots,d_n),{\rm red}}$ is $0$.
Let us prove this stronger fact. By a density argument, we may assume that $n=2q$ is even, and that the Blaschke products $\chi _r(u)$ have only simple zeroes. Firstly, we describe the tangent space of $\mathcal V_{(d_1,\dots,d_n),{\rm red}}$ at a generic point. We use the notation of section \ref{section explicit}.
\begin{lemma}\label{espace tangent}
The tangent vectors to $\mathcal V_{(d_1,\dots,d_n),{\rm red}}$ at a generic point $u$ where every $\chi _r$  has
only simple zeroes are linear combinations with real coefficients of $u_j, u_jH_u(u_\ell ), 1\le j,\ell \le q $, and of the following functions,
for $ \zeta \in \C$ and $1\le j,k\le q$,
\begin{eqnarray*}
\dot u_{\chi _{2j-1}, \zeta }(z)&:=& \left (\overline \zeta \frac{z}{1-\overline pz}-\zeta \frac{1}{z-p}\right )u_j(z)H_u(u_j)(z)\ ,\ \chi _{2j-1}(p)=0\ ,\\
\dot u_{\chi _{2k},\zeta}(z)&:=&  \left (\overline \zeta \frac{z}{1-\overline pz}-\zeta \frac{1}{z-p}\right )zu'_k(z)K_u(u_k')(z)\ ,\ \chi _{2k}(p)=0\ .
\end{eqnarray*}
\end{lemma}
We assume  Lemma \ref{espace tangent}  and show how it implies Lemma \ref{restriction omega}.
Notice that 
\beno
(u\vert u_j)&=&\Vert u_j\Vert ^2\ ,\  (u\vert u_jH_u(u_j))=\Vert H_u(u_j)\Vert ^2\ ,\\
 (u\vert H_u(u_\ell )u_j)&=&(H_u(u_j)\vert H_u(u_\ell ))=0\ ,\ j\ne \ell \ ,
 \eeno
and therefore 
$\alpha (u)$ cancels on $u_j, u_jH_u(u_\ell ), 1\le j,\ell \le q $.  We now deal with  vectors $\dot u_{\chi _r,\zeta }$. 
\begin{eqnarray*}
(u\vert \dot u_{\chi _{2j-1},\zeta })&=& \zeta \left (u\Big \vert \frac{z}{1-\overline pz}u_jH_u(u_j)\right )-\overline \zeta \left (u\Big \vert \frac{u_j}{z-p}H_u(u_j)\right )\\&=& \zeta \left (H_u(u_j)\Big \vert \frac{z}{1-\overline pz}H_u(u_j)\right )-\overline \zeta \left (H_u^2(u_j)\Big \vert \frac{u_j}{z-p}\right )\ ,
\end{eqnarray*}
where we used that $u_j/(z-p)$ belongs to $L^2_+$. Since $H_u^2(u_j)=\rho _j^2u_j$ and $\rho _j^2\vert u_j\vert ^2=\vert H_u(u_j)\vert ^2$ on the 
unit circle, we infer
$$(u\vert \dot u_{\chi _{2j-1},\zeta })=\zeta \left (H_u(u_j)\Big \vert \frac{z}{1-\overline pz}H_u(u_j)\right )-\overline \zeta \left (\frac{z}{1-\overline pz}H_u(u_j)\Big \vert H_u(u_j)\right )$$
and consequently
$$\langle \alpha (u),\dot u_{\chi _{2j-1},\zeta }\rangle =2{\rm Im}\ \zeta\left (H_u(u_j)\Big \vert \frac{z}{1-\overline pz}H_u(u_j)\right )\ ,$$
which is $0$ for every $\zeta \in \C $ if and only if
$$\left (H_u(u_j)\Big \vert \frac{z}{1-\overline pz}H_u(u_j)\right )=0\ .$$
Let us prove this identity. Set
$$v:=\frac{z}{1-\overline pz}H_u(u_j)\ .$$
Notice that, since $\chi _{2j-1}(p)=0$,  $v\in E_u(\rh _j)$, and moreover
$$(v\vert 1)=v(0)=0\ .$$
Therefore
$$(H_u(u_j)\vert v)=(H_u(v)\vert u_j)=(H_u(v)\vert u)=(1\vert H_u^2(v))=\rho _j^2(1\vert v)=0\ .$$
We conclude that
$$\langle \alpha (u),\dot u_{\chi _{2j-1},\zeta }\rangle=0\ .$$
Similarly, we calculate
\begin{eqnarray*}
(u\vert \dot u_{\chi _{2k},\zeta })&=& \zeta \left (u\Big \vert \frac{z}{1-\overline pz}zu_k'K_u(u_k')\right )-\overline \zeta \left (u\Big \vert \frac{K_u(u_k')}{z-p}zu_k'\right )\\&=& \zeta \left (K_u(u_k')\Big \vert \frac{z}{1-\overline pz}K_u(u_k')\right )-\overline \zeta \left (K_u(u_k')\Big \vert \frac{K_u(u_k')}{z-p}\right )\ ,
\end{eqnarray*}
where we have used that $K_u(u_k')/(z-p)$ belongs to $L^2_+$. We conclude that
$$\langle \alpha (u),\dot u_{\chi _{2k},\zeta }\rangle =2{\rm Im}\ \zeta\left (K_u(u_k')\Big \vert \frac{z}{1-\overline pz}K_u(u_k')\right )\ ,$$
which is $0$ for every $\zeta \in \C $ if and only if
$$\left (K_u(u_k')\Big \vert \frac{z}{1-\overline pz}K_u(u_k')\right )=0\ .$$
Since $\vert K_u(u_k')\vert ^2=\sigma _k^2\vert u_k'\vert ^2$ on the unit circle, we are left to prove
$$\left (u_k'\Big \vert \frac{z}{1-\overline pz}u_k'\right )=0\ .$$
Set
$$w:=\frac{1}{1-\overline pz}u_k'\ .$$
We notice that $w\in F_u(\sigma _k)$, and that $zw \in F_u(\sigma _k)$. Moreover,
$$w=\frac{1}{1-\vert p\vert ^2}(\overline p\chi _p+1)u_k'\ ,$$
therefore, setting $\chi _{2k}:=g_k\chi _p$,
$$K_u(w)=\frac{\sigma _k}{1-\vert p\vert ^2}(pg_ku'_k+\chi _{2k}u'_k)=\frac{\sigma _kg_k}{1-\vert p\vert ^2}(p+\chi _p)u'_k=\sigma _kg_kzw\ .$$
In particular,
$$(K_u(w)\vert 1)=K_u(w)(0)=0\ .$$
We now conclude as follows,
$$\left (u_k'\Big \vert \frac{z}{1-\overline pz}u_k'\right )=(u_k'\vert zw)=(u\vert zw)=(K_u(w)\vert 1)=0\ .$$
This completes the proof up to the proof of lemma \ref{espace tangent}.
\end{proof}
Let us prove lemma \ref{espace tangent}. We are going to use formulae from section \ref{section explicit}, namely
$$u(z)=\sum _{j=1}^q \chi _{2j-1}(z)h_j(z)\ ,$$
where $\mathcal H(z)=(h_\ell (z))_{1\le \ell \le q}$ satisfies 
$$\mathcal C(z)\mathcal H(z)={\bf 1}\ ,\ \mathcal C(z)=\left (\frac{\rho _\ell -\sigma _k z\chi _{2k}(z)\chi _{2\ell -1}(z)}{\rho _\ell ^2-\sigma _k^2}\right )_{1\le k,\ell \le q}\ .$$
If we denote by $\dot {\  }$  the derivative with respect to one of the parameters $\rho _j, \sigma _k$ or one of the coefficients of the $\chi _r$,
we have 
$$\dot u(z)=\sum _{\ell =1}^q (\dot \chi _{2\ell -1}(z)h_\ell (z)+\chi _{2\ell-1}(z)\dot h_\ell (z))\ ,$$
with
$$\dot {\mathcal H}(z)=-\mathcal C(z)^{-1}\dot {\mathcal C}(z)\mathcal H(z)\ .$$ 
In the case of the derivative with respect to $\rho _j$, one gets
$$\dot h_\ell (z)=h_j(z)\sum _{k=1}^q\frac{(-1)^{k+\ell }\Delta _{k\ell}(z)(\rh _j^2+\si _k^2-2\si _k\rh _j z\chi _{2j-1}(z)\chi_{2k}(z))}
{\det \mathcal C(z) (\rh _j^2-\si _k^2)^2}\ ,$$
and therefore, from formula (\ref{u'DR}), 
\beno
\dot u(z)&=& h_j(z)\sum _{k=1}^qu'_k(z)\frac{ (\rh _j^2+\si _k^2-2\si _k\rh _j z\chi _{2j-1}(z)\chi_{2k}(z))}
{ (\rh _j^2-\si _k^2)^2} \\
&=& \frac {H_u(u_j)(z)}{\rho _j}\sum _{k=1}^q \frac{\rh _j^2+\si _k^2}{(\rh _j^2-\si _k^2)^2}u'_k(z)-2\rh _ju_j(z)\sum _{k=1}^q\frac{zK_u(u'_k)(z)}
{(\rh _j^2-\si _k^2)^2}\ .
\eeno 
Observe that 
$$zK_u(u'_k)(z)=[SS^*H_u(u'_k)](z)=H_u(u'_k)(z)-\kappa _k^2\ ,$$
and that $u'_k$ is a linear combination with real coefficients of $u_\ell $ in view of (\ref{urho}). We infer that, in this case, $\dot u$ is a linear combination with real coefficients of $u_j$ and $u_mH_u(u_\ell )$. 

In the case of the derivative with respect to $\si _k$, one  similarly gets
$$\dot h_\ell (z)=\sum _{j=1}^q h_j(z)\frac{(-1)^{k+\ell }\Delta _{k\ell}(z)\left(z\chi _{2j-1}(z)\chi_{2k}(z)(\rh _j^2+\si _k^2)-2\si _k\rh _j \right)}
{\det \mathcal C(z) (\rh _j^2-\si _k^2)^2}\ ,$$
and therefore, from formula (\ref{u'DR}), 
\beno
\dot u(z)&=& u'_k(z)\sum _{j=1}^qh_j(z)\frac{ z\chi _{2j-1}(z)\chi_{2k}(z)(\rh _j^2+\si _k^2)-2\si _k\rh _j }
{ (\rh _j^2-\si _k^2)^2} \\
&=& \frac{zK_u(u'_k)(z)}{\si _k}\sum _{j=1}^q \frac{(\rh _j^2+\si _k^2)u_j(z)}{(\rh _j^2-\si _k^2)^2}-2\si _ku'_k(z)\sum _{j=1}^q\frac {H_u(u_j)(z)}
{(\rh _j^2-\si _k^2)^2}\ ,
\eeno 
which is a linear combination with real coefficients of $u_j$ and $u_jH_u(u_\ell )$.

In the case of a derivative with respect to one of the zeroes of $\chi _{2j-1}$, we obtain a simpler identity,
$$\dot h_\ell (z)=\frac{\dot \chi _{2j-1}(z)}{\chi _{2j-1}(z)}u_j(z)\sum _{k=1}^q \frac{(-1)^{k+\ell}\Delta _{k\ell}(z)z\si _k\chi _{2k}(z)}{\det \mathcal C(z) (\rh _j^2-\si _k^2)}\ ,$$
and therefore, from formulae (\ref{u'DR}), (\ref{urho})  and (\ref{sommesimplekappa}),
\beno
\dot u(z)&=&\frac{\dot \chi _{2j-1}(z)}{\chi _{2j-1}(z)}u_j(z)+\sum _{\ell =1}^q\chi _{2\ell -1}(z)\dot h_\ell (z)\\
&=& \frac{\dot \chi _{2j-1}(z)}{\chi _{2j-1}(z)}\frac{u_j(z)H_u(u_j)(z)}{\tau _j^2}\ .
\eeno
In the case of a derivative with respect to one of the zeroes of $\chi _{2k}$, we obtain similarly,
$$\dot h_\ell (z)=\frac{\dot \chi _{2k}(z)}{\chi _{2k}(z)}\frac{(-1)^{k+\ell}\Delta _{k\ell}(z)zK_u(u'_k)(z)}{\det \mathcal C(z)}\ ,$$
and therefore
\beno
\dot u(z)&=&\sum _{\ell =1}^q\chi _{2\ell -1}(z)\dot h_\ell (z)\\
&=& \frac{\dot \chi _{2k}(z)}{\chi _{2k}(z)}\frac{u'_k(z)zK_u(u'_k)(z)}{\kappa _k^2}\ .
\eeno
The proof of Lemma \ref{espace tangent} is completed by observing that, since $\chi _r$ is a product of functions $\chi _p$ for $\vert p\vert <1$,
$\dot \chi _r/\chi _r$ is a sum of terms of the form
$$ \left (\overline \zeta \frac{z}{1-\overline pz}-\zeta \frac{1}{z-p}\right )$$
where $\zeta :=\dot p$.

\begin{appendix}

\section{Some Bateman-type formulae}
Let $u\in VMO_+$. Denote by $(\rho_j)$ the decreasing sequence of elements of $\Sigma _H(u)$, and by $(\sigma_k^2)$ 
the decreasing sequence of  elements of  $\Sigma_K(u)$. Recall that both sequences are either finite or infinite, with the same number of elements,
and we have
 $$\rho_1^2>\sigma _1^2>\rho_2^2>\dots $$
 Denote by $u_j$ the orthogonal projection of $u$ onto $E_u(\rho_j)=\ker(H_u^2-\rho_j^2I)$, and by $\tau_j$ the norm of $u_j$. Similarly, denote by $u'_k$, the orthogonal projection of $u$ onto $F_u(\sigma_k):=\ker(K_u^2-\sigma_k^2I)$, and by $\kappa_k$ the norm of $u'_k$.
In this section, we state and prove  several formulae connecting these sequences. These formulae are based on the special case of a general formula for the resolvent of a finite rank perturbation of an operator, which seems to be due to Bateman \cite{Ba} in the framework of Fredholm integral equations. Further references can be found in Chap. II, sect. 4.6 of  \cite{KK}, section 106 of \cite{AG} and \cite{Na}, from which we borrowed this information.
\begin{proposition}\label{formulae}
The following functions coincide respectively for $x$ outside the set $\{\frac 1{\rho_j^2}\}$ and outside the set $\{\frac 1{\sigma_k^2}\}$.
\begin{equation}\label{J(x)Appendix}
\prod_j \frac{1-x\sigma_j^2}{1-x\rho_j^2}=1+x\sum_j \frac{\tau_j^2}{1-x\rho_j^2}
\end{equation}
\begin{equation}\label{1/J(x)Appendix}
\prod_j\frac{1-x\rho_j^2}{1-x\sigma_j^2}=1-x\left (\sum_j\frac{\kappa_j^2}{1-x\sigma_j^2}\right )\ .
\end{equation}
Furthermore, 
\begin{equation}\label{sommenu}
1-\sum _j\frac{\tau _j^2}{\rho _j^2}=\prod _j\frac{\si _j^2}{\rh _j^2}\ ,
\end{equation}
and, if ${\displaystyle \prod _j\frac{\si _j^2}{\rh _j^2}=0}$, 
\begin{equation}\label{sommenurho}
\sum _j\frac{\tau _j^2}{\rho _j^4}=\frac 1{\rh _1^2}\prod _j\frac{\si _j^2}{\rh _{j+1}^2}\ .
\end{equation}
The $\tau_j^2$'s are given by 
\begin{equation}\label{tau}
\tau_j^2= \left(\rho_j^2-\sigma_j^2\right)\prod_{k\ne j}\frac{\rho_j^2-\sigma_k^2}{\rho_j^2-\rho_k^2}\ ,\end{equation}
and the $\kappa_j^2$'s by 
\begin{equation}\label{kappa}
\kappa_j^2=(\rho_j^2-\sigma_j^2)\prod_{k\ne j}\frac{\sigma_j^2-\rho_k^2}{\sigma_j^2-\sigma_k^2}.
\end{equation}
\end{proposition}
\begin{proof}
For $x\notin \{\frac 1{\rho_j^2}\}$, we set $$J(x):=((I-xH_u^2)^{-1}(1)\vert 1).$$
We claim that
\begin{equation}\label{J}
J(x)=\prod_j \frac{1-x\sigma_j^2}{1-x\rho_j^2}.
\end{equation}
Indeed, let us first assume that $H_u^2$ and $K_u^2$ are  of trace class and compute the trace of $(I-xH_u^2)^{-1}-(I-xK_u^2)^{-1}$. We  write
$$[(I-xH_u^2)^{-1}-(I-xK_u^2)^{-1}](f)=\frac x{J(x)}(f\vert (I-xH_u^2)^{-1}u)\cdot (I-xH_u^2)^{-1} u.$$
Consequently, taking the trace, we get
$${\rm Tr}[(I-xH_u^2)^{-1}-(I-xK_u^2)^{-1}]=\frac x{J(x)}\Vert (I-xH_u^2)^{-1}u\Vert^2.$$
Since, on the one  hand
$$\Vert (I-xH_u^2)^{-1} u\Vert^2= ((I-xH_u^2)^{-1}H_u^2(1)\vert 1)=J'(x)$$
and on the other hand
\begin{eqnarray*}
{\rm Tr}[(I-xH_u^2)^{-1}-(I-xK_u^2)^{-1}]&=&x{\rm Tr}[H_u^2(I-xH_u^2)^{-1}-K_u^2(I-xK_u^2)^{-1}]\\
&=&x \sum _j\left ( \frac{\rho_j^2}{1-\rho_j^2x}-\frac{\sigma_j^2}{1-\sigma_j^2x}\right )\ ,\end{eqnarray*}
where  we used  proposition \ref{rigidity}.
 On the other hand,
 \begin{equation}\label{traceformula}
  \sum _j\left ( \frac{\rho^2_j}{1-\rho_j^2x}-\frac{\sigma_j^2}{1-\sigma_j^2x}\right )=\frac{J'(x)}{J(x)}\ ,\ x\notin \left \{ \frac 1{\rho_j^2}, \frac 1{\sigma_j^2} \right \}\ .
 \end{equation}

 This gives equality (\ref{J}) for $H_u^2$ and $K_u^2$ of trace class. To extend this formula to compact operators, we remark that $\sum _j(\rho_j^2-\sigma_j^2)$ converges since $0\le \rho_j^2-\sigma_j^2\le \rho_j^2-\rho_{j+1}^2$ and $(\rho_j^2)$ tends to zero from the compactness of $H_u^2$. Hence the infinite product in formula (\ref{J}) and the above computation makes sense for compact operators.

On the other hand, for $x\notin \{\frac 1{\rho_j^2}\}$, if $\tau_j$ denotes the norm of $u_j$, 
\begin{eqnarray*}
J(x)&=&((I-xH_u^2)^{-1}(1)\vert 1)=1+x((I-xH_u^2)^{-1}(u)\vert u)\\
&=&1+x(\sum_j( I-xH_u^2)^{-1}(u_j)\vert u))=1+x\sum_j \frac{\tau_j^2}{1-x\rho_j^2}
\end{eqnarray*}
hence
\begin{equation}
\prod_j \frac{1-x\sigma_j^2}{1-x\rho_j^2}=1+x\sum_j\frac{\tau_j^2}{1-x\rho_j^2}\ .
\end{equation}
Passing to the limit as $x$ goes to $-\infty $ in (\ref{J(x)Appendix}), we obtain (\ref{sommenu}). If we assume that
the left hand side of (\ref{sommenu}) cancels, then (\ref{J(x)Appendix}) can be rewritten as
$$\prod_j \frac{1-x\sigma_j^2}{1-x\rho_j^2}=\sum_j\frac{\tau_j^2}{\rh _j^2(1-x\rho_j^2)}\ .$$
Multiplying by $x$ and passing to the limit as $x$ goes to $-\infty $ in this new identity, we obtain (\ref{sommenurho}).
Furthermore, we multiply  both terms of (\ref{J(x)Appendix}) by $(1-x\rho_j^2)$ and we let $x$ go to $1/\rho_j^2$. We get
$$\tau_j^2= \left(\rho_j^2-\sigma_j^2\right)\prod_{k\ne j}\frac{\rho_j^2-\sigma_k^2}{\rho_j^2-\rho_k^2} \ .$$
For Equality (\ref{1/J(x)}), we do almost the same analysis. First, we establish as above that
$$\frac{1}{J(x)}=1-x((I-xK_u^2)^{-1}(u)\vert u)=1-x\sum_k \frac{\kappa_k^2}{1-x\sigma_k^2}$$
where $\kappa_k^2=\Vert u'_k\Vert^2$. 

Identifying this expression with $$\frac 1{J(x)}=\prod_j\frac{1-x\rho_j^2}{1-x\sigma_j^2}$$ we get
$$\kappa_j^2=(\rho_j^2-\sigma_j^2)\prod_{k\ne j}\frac{\sigma_j^2-\rho_k^2}{\sigma_j^2-\sigma_k^2}.$$
\end{proof}

As a consequence of the previous lemma, we get the following couple of corollaries.
\begin{corollary}For any $k, r\ge 1$, we have
 \begin{equation}\label{sommesimpletau}\sum_j \frac{\tau_j^2}{\rho_j^2-\sigma_k^2}=1\end{equation}
  \begin{equation}\label{sommesimplekappa}\sum_j \frac{\kappa_j^2}{\rho_k^2-\sigma_j^2}=1\end{equation}
 \begin{equation}\label{sommedoubletau}\sum_j \frac{\tau_j^2}{(\rho_j^2-\sigma_k^2)(\rho_j^2-\sigma_r^2)}=\frac 1{\kappa_k^2}\delta_{kr}\end{equation}
  \begin{equation}\label{sommedoublekappa}\sum_j \frac{\kappa_j^2}{(\sigma_j^2-\rho_k^2)(\sigma_j^2-\rho_r^2)}=\frac {1}{\tau_k^2}\delta_{kr}\end{equation}
\end{corollary}

\begin{proof}
The  first two equalities (\ref{sommesimpletau}) and (\ref{sommesimplekappa}) are obtained by making $x=\frac 1{\sigma_k^2}$ and $x=\frac 1{\rho_k^2}$ respectively  in formula (\ref{J(x)Appendix}) and formula (\ref{1/J(x)Appendix}).
For equality (\ref{sommedoubletau}) in the case $k= r$, we first make the change of variable $y=1/x$ in formula (\ref{J(x)Appendix}) then differentiate both sides with respect to $y$ and make $y=\sigma_r^2$. 
Equality (\ref{sommedoublekappa}) in the case $k=r$ follows by differentiating equation (\ref{1/J(x)Appendix}) and making $x=\frac 1{\rho_m^2}$. 
Both equalities in the case $m\neq p$ follow directly respectively from equality (\ref{sommesimpletau}) and equality (\ref{sommesimplekappa}).
\end{proof}

\begin{corollary}\label{kernel}
The kernel of $H_u$ is $\{ 0\} $ if and only if
$$\prod _j\frac{\sigma _j^2}{\rh _j^2}=0\ ,\ \prod _j\frac{\sigma _j^2}{\rh _{j+1}^2}=\infty \ .$$
\end{corollary}
\begin{proof}
By the first part of theorem 4 in \cite{GG3} --- which is independent of multiplicity assumptions--- ,  the kernel of $H_u$ is $\{ 0\} $ if and only if $1\in \overline R\setminus R$, where $R={\rm Ran}(H_u)$
 denotes the range of $H_u$. On the other hand,
 $$u=\sum _j u_j=\sum _j \frac{H_u(H_u(u_j))}{\rh _j^2}\ ,$$
 hence the orthogonal projection of $1$ onto $\overline R$ is
 $$\sum _j \frac{H_u(u_j)}{\rh _j^2}\ .$$
 Consequently, $1 \in \overline R$ if and only if
 $$1=\sum _j \left \Vert \frac{H_u(u_j)}{\rh _j^2}\right \Vert ^2=\sum _j \frac{\tau _j^2}{\rh _j^2}\ .$$
 Moreover, if this is the case,
 $$1=\sum _j \frac{H_u(u_j)}{\rh _j^2}$$
 and $1\in R$ if and only if the series $\sum _j u_j/\rh_j^2$ converges, which is equivalent to
 $$\sum _j \frac{\tau _j^2}{\rh _j^4}<\infty \ .$$
 Hence $1\in \overline R \setminus R$ if and only if
 $$\sum _j \frac{\tau _j^2}{\rh _j^2}=1\ ,\ \sum _j \frac{\tau _j^2}{\rh _j^4}=\infty \ ,$$
 which is the claim, in view of identities (\ref{sommenu}) and (\ref{sommenurho}).
\end{proof}

\section{The structure of finite Blaschke products}

In this appendix, we describe the set $\mathcal B_d$  of Blaschke products of degree $d$. Every element of $\mathcal B_d$ can be written
$$\Psi =\expo_{-i\psi }\chi \ ,$$
where $\psi \in \T $ and  $\chi \in \mathcal B_d^\sharp $ is a Blaschke product of the form
$$\chi (z)=\frac {P(z)}{z^d \overline P\left (\frac 1z\right )}\ ,$$
where $P(z)=z^d+a_1z^{d-1}+\dots +a_d$ is a monic polynomial of degree $d$ with all its zeroes in the open unit disc $\D $. Conversely, if
$P$ is such a polynomial, then $\chi $ is a Blaschke product of degree $d$.
We denote by $\mathcal O_d$ the open subset of $\C ^d$ made of such $(a_1,\dots ,a_d)$. The following result is connected to the Schur--Cohn criterion \cite{Sc}, \cite{Co}, and is classical  in control theory, see 
{\it e.g.} \cite{He} and references therein. For the sake of completeness, we give a self contained proof.

\begin{proposition}\label{Od}
For every $d\ge 1$ and $(a_1,\dots ,a_d)\in \C ^d$, the following two assertions are equivalent.
\begin{enumerate}
\item $(a_1,\dots ,a_d)\in \mathcal O_d$\ .
\item $\vert a_d\vert <1$ and $$\left (\frac{a_k-a_d\overline a_{d-k}}{1-\vert a_d\vert ^2}\right )_{1\le k\le d-1}\in \mathcal O_{d-1}\ .$$
\end{enumerate}
In particular, for every $d\ge 0$, $\mathcal O_d$ is diffeomorphic to $\R ^{2d}$.
\end{proposition}
\begin{proof}
Consider the rational functions
$$\chi (z)=\frac{z^d+a_1z^{d-1}+\dots +a_d}{1+\overline a_1z+\dots +\overline a_dz^d}\ ,$$
and
$$\tilde \chi (z)=\frac{\chi (z)-\chi (0)}{1-\overline {\chi (0)}\chi (z)}=z\frac{z^{d-1}+b_1z^{d-2}+\dots +b_{d-1}}{1+\overline b_1z+\dots +\overline b_{d-1}z^{d-1}}\ ,\ b_k:=\frac{a_k-a_d\overline a_{d-k}}{1-\vert a_d\vert ^2}\ .$$
If (1) holds true, then $\chi \in \mathcal B_d$, which implies 
\begin{equation}\label{chi}
\forall z\in \D, \vert \chi (z)\vert <1 \ ,\ \vert \chi (\expo_{ix})\vert =1\ .\ 
\end{equation}
In particular, $\chi (0)=a_d\in \D $, and therefore the numerator and the denominator of $\tilde \chi $ have no common root. Moreover,
\begin{equation}\label{chitilde}
\forall z\in \D, \vert \tilde \chi (z)\vert <1 \ ,\ \vert \tilde \chi (\expo_{ix})\vert =1\ .
\end{equation}
This implies $\tilde \chi \in \mathcal B_d$, hence (2). Conversely, if (2) holds, then $\tilde \chi $ satisfies (\ref{chitilde}) and has degree $d$, hence 
$$\chi (z)=\frac{\tilde \chi (z)+a_d}{1+\overline a_d\tilde \chi(z)}$$
satisfies (\ref{chi}) and has degree $d$, whence (1). 
\s
The second statement follows from an easy induction argument on $d$, since $\mathcal O_1=\D $ is diffeomorphic to $\R ^2$.
\end{proof}

\section{Two results by Adamyan--Arov--Krein}

In this appendix, we recall the proof of two important results  by Adamyan--Arov--Krein, which have been used throughout
our paper. The proof is translated from \cite{AAK} into our representation of Hankel operators, and is given for the convenience of the reader.

\begin{theo}[Adamyan, Arov, Krein \cite{AAK}] 
Let $u\in VMO_+\setminus \{ 0\} $. Denote by $(\underline s_k(u))_{k\ge 0}$ the sequence of singular values of $H_u$, namely the eigenvalues of $\vert H_u\vert :=\sqrt{H_u^2}$, in decreasing order, and repeated according to their multiplicity.Let $k\ge 0, m\ge 1,$ such that
$$\underline s_{k-1}(u)>\underline s_k(u)=\dots =\underline s_{k+m-1}(u)=s>\underline s_{k+m}(u)\ ,$$
with the convention $\underline s_{-1}(u):=+\infty $. 
\begin{enumerate}
\item For every $h\in E_u(s)\setminus \{ 0\} $, there exists a polynomial $P\in \C _{m-1}[z]$ such that
$$\forall z\in \D\ ,\ \frac{sh(z)}{H_u(h)(z)}=\frac{P(z)}{z^{m-1}\overline P\left (\frac 1z\right )}\ .$$
\item There exists a rational function $r$ with no pole on $\overline \D$ such that ${\rm rk} (H_r)=k$ and
$$\Vert H_u-H_r\Vert =s\ .$$
\end{enumerate}
\end{theo}
\begin{proof}
We start with the case $k=0$. In this case the statement (2) is trivial, so we just have to prove (1). This is a consequence of the following lemma.
\begin{lemma}\label{tops}
Assume $s=\Vert H_u\Vert $. For every $h\in E_u(s)\setminus \{ 0\} $, consider the following inner outer decompositions,
$$h=ah_0\ ,\ s^{-1}H_u(h)=bf_0\ .$$
If $c$ is an arbitrary inner divisor of $ab$,  $ab=cc'$, then $ch_0\in E_u(s)$, with
\begin{equation} \label{ch0}
H_u(ch_0)=sc'f_0\ ,\ H_u(c'f_0)=sch_0\ .
\end{equation}
 In particular, $a, b$ are finite Blaschke products and 
\begin{equation}\label{abBla}
\deg (a)+\deg (b) +1\le \dim E_u(s)\ .
\end{equation}
Furthermore, there exists an outer function $h_0$ such that, if $m:=\dim E_u(s)$,
\begin{equation}\label{Eunorm}
E_u(s)=\C _{m-1}[z] h_0\ ,
\end{equation}
and there exists $\varphi \in \T $ such that, for every $P\in \C _{m-1}[z]$, 
\begin{equation}\label{Hunorm}
H_u(Ph_0)(z)=s\expo _{i\varphi}z^{m-1}\overline P\left (\frac 1z\right )h_0(z)\ .
\end{equation}
\end{lemma}
Let us prove this lemma. We will need a number of elementary properties of Toeplitz operators $T_b$ defined by equation (\ref{Toeplitz}), 
where $b$ is a function in $L^\infty _+:=L^2_+\cap L^\infty $, which we recall below. In what follows, $b$ denotes a function in $L^\infty _+$ and $u\in BMO_+$.
\begin{enumerate}
\item 
$$H_uT_b=T_{\overline b}H_u=H_{T_{\overline b}u}\ .$$
\item If $\vert b\vert \le 1$ on $\S ^1$, 
$$H_u^2\ge T_{\overline b}H_u^2T_b\ .$$
\item If $\vert b\vert =1$ on $\S ^1$, namely $b$ is an inner function,  
$$\forall f\in L^2_+\ ,\ f=T_bT_{\overline b}f  \Longleftrightarrow \Vert f\Vert =\Vert T_{\overline b}f\Vert \ .$$
\end{enumerate}
Indeed, (1) is just equivalent to the elementary identities
$$\Pi (u\overline b\overline h)=\Pi (\overline b\Pi (u\overline h))=\Pi ((\Pi (\overline bu)\overline h)\ .$$
As for (2),
we observe that $T_b^*=T_{\overline b}$ and 
$$\Vert T_{\overline b}h\Vert \le \Vert \overline bh\Vert \le \Vert h\Vert \ .$$
Hence, using (1),
$$(H_u^2h\vert h)-(T_{\overline b}H_u^2T_bh\vert h)=\Vert H_u(h)\Vert ^2-\Vert T_{\overline b}H_u(h)\Vert ^2\ge 0\ .$$
Finally, for (3) we remark that, if $b$ is inner,  $T_{\overline b}T_b=I$ and $T_bT_{\overline b}$ is the orthogonal projector onto the range of $T_b$, namely $bL^2_+$.  Since $\Vert T_{\overline b}f\Vert =\Vert T_bT_{\overline b}f\Vert $, (3) follows.
\s
Let us come back to the proof of Lemma \ref{tops}. Starting from
$$H_u(h)=sf\ ,\ H_u(f)=sh\ ,\ h=ah_0\ ,\ f=bf_0\ ,\ ab =cc'\ ,$$
we obtain, using property (1),
$$T_{\overline c'}H_u(ch_0)= H_u(cc'h_0)=T_{\overline b}H_u(h)=sf_0\ .$$
In particular, 
$$\Vert H_u(ch_0)\Vert \ge \Vert T_{\overline c'}H_u(ch_0)\Vert =s\Vert f_0\Vert =s\Vert f\Vert =s\Vert h\Vert =s\Vert ch_0\Vert \ .$$
Since $s=\Vert H_u\Vert$, all the above inequalities are equalities, hence $ch_0\in E_u(s)$, and, using (3), 
$$H_u(ch_0)=T_{c'}T_{\overline c'}H_u(ch_0)=sc'f_0\ .$$
The second identity in (\ref{ch0}) immediately follows.  Remark that, if $\Psi $ is an inner function of degree at least $d$, there exist $d+1$ linearly independent inner divisors of $\Psi $ in $L^\infty _+$. Then inequality (\ref{abBla}) follows. Let us come to the last part. Since $\dim E_u(s)=m$, there exists $h\in E_u(s)\setminus \{ 0\}$ such that the first $m-1$ Fourier coefficients of $h$ cancel, namely 
$$h=z^{m-1}\tilde h\ .$$
Considering the inner outer decompositions
$$\tilde h=ah_0\ ,\ H_u(h)=s bf_0\ ,$$
and applying the first part of the lemma, we conclude that $\deg (a)+\deg (b)=0$, hence, up to a slight change of notation, $a=b=1$,
and, for $\ell =0,1,\dots ,m-1$,
$$H_u(z^\ell h_0)=sz^{m-1-\ell}f_0\ ,\ H_u(z^{m-\ell-1} f_0)=sz^{\ell}h_0\ .$$
This implies
$$E_u(s)=\C _{m-1}[z] h_0=\C _{m-1}[z]f_0\ .$$
Since $\Vert h_0\Vert =\Vert h\Vert =\Vert f\Vert =\Vert f_0\Vert $, it follows that $f_0=\expo_{i\varphi} h_0$, and (\ref{Hunorm}) follows from the antilinearity of $H_u$. The proof of Lemma \ref{tops} is complete.
\s
Let us complete the proof of the theorem by proving the case $s<\Vert H_u\Vert $. The crucial new observation is the following.
\begin{lemma}\label{unimod}
There exists a function $\phi \in L^\infty $ such that $\vert \phi \vert =1$ on $\S ^1$, and such that the operators $H_u$ and $H_{s\Pi (\phi )}$ coincide on $E_u(s)$.
\end{lemma}
Let us prove this lemma. For every pair $(h,f)$ of elements of $E_u(s)$ such that  $H_u(h)=sf, H_u(f)=sh$, we claim that the function
$$\phi :=\frac{f}{\overline h}\ ,$$
 does not depend on the choice of the pair $(h,f)$. Indeed, it is enough to check that, if $(h',f')$ is another such pair,
 $$f\overline h'=f'\overline h\ .$$
 In fact, for every $n\ge 0$,
 \beno s(f\overline h'\vert z^n)&=&(H_u(h)\vert S^nh')=((S^*)^nH_u(h)\vert h')\\
 &=&(H_u(S^nh)\vert h')=(H_u(h')\vert S^nh)=s(f'\overline h\vert z^n)\ .
 \eeno
 Changing the role of $(h,h')$ and $(f,f')$, we get the claim. Finally the fact $\vert \phi \vert =1$ comes from applying the above identity to the pairs $(h,f)$ and $(f,h)$. Then we just have to check that, for every such pair,
 $$H_{s\Pi (\phi )}(h)=s\Pi (\Pi (\phi )\overline h)=s\Pi (\phi \overline h)=sf \ .$$
 This completes the proof of Lemma \ref{unimod}. 
 \s
 Let us come to part (2) of the Theorem. Introduce 
 $$v:=s\Pi (\phi )\ .$$
 We are going to show that $r:=u-v$ is a rational function with no pole on $\overline{\D}$, ${\rm rk}(H_r)=k$ and 
 $$\Vert H_u-H_r\Vert =s\ .$$
 Since, for every $h\in L^2_+$,
 $$H_v(h)=s\Pi (\phi \overline h)\ ,$$
 we infer $\Vert H_v\Vert \le s$, and  from $E_u(s)\subset E_v(s)$, we conclude 
 $$\Vert H_v\Vert =s\ .$$
 Because of (\ref{Ku}), $H_u$ and $H_v$ coincide on the smallest shift invariant closed subspace of $L^2_+$ containing $E_u(s)$.
 By Beurling's theorem \cite{B}, this subspace is $aL^2_+$ for some inner function $a$. 
 Then $H_r =0$ on $aL^2_+$, hence the rank of $H_r$ is at most the dimension of $(aL^2_+)^\perp $. 
 Since 
$$\Vert H_u-H_r\Vert =\Vert H_v\Vert =s<\underline s_{k-1}(u)\ ,$$
the rank of $H_r$ cannot be smaller than $k$, and the result will follow by proving that the dimension of $(aL^2_+)^\perp $ is  $k$.

We can summarize the above construction as
 $$H_{T_{\overline a}u}=H_uT_a=H_vT_a=H_{T_{\overline a}v}\ .$$
 The above Hankel operator is compact and its norm is at most $s$. In fact, if $H_u(h)=sf\ ,\ H_u(f)=sh$, with $f=a \tilde f$, it is clear from property (1) above that
 $$H_{T_{\overline a}u}(h)=s\tilde f\ ,\ H_{T_{\overline a}u}(\tilde f)=sh\ .$$
 In particular, 
$$\Vert  H_{T_{\overline a}u}\Vert =s\ .$$
Applying property (\ref{ch0}) from Lemma \ref{tops}, we conclude that there exists an outer function $h_0$ such that $$ch_0\in E_{T_{\overline a}u}(s)$$  for every inner divisor $c$ of $a$. Moreover, $a$ is a Blaschke product of finite degree $d$. Since $h_0$ is outer, it does not vanish at any point of $\D $, therefore, it is easy to find $d$ inner divisors $c_1,\dots ,c_d$ of $a$ such that $c_1h_0,\dots ,c_dh_0$ are linearly independent and generate a vector subspace $\tilde E$ satisfying
$$\tilde E\cap aL^2_+=\{ 0\} \ .$$
Consequently, we obtain
$$\tilde E\oplus E_u(s)\subset E_{T_{\overline a}u}(s),$$
whence
\begin{equation}\label{borneinfdim}
d':=\dim E_{T_{\overline a}u}(s)\ge d+m\ .
\end{equation}
On the other hand, by property (2) above, we have
$$H_{T_{\overline a}u}^2=T_{\overline a}H_u^2T_a \le H_u^2\ ,$$
therefore, from the min-max formula,
$$\forall n, \underline s_n(T_{\overline a}u)\le \underline s_n(u)\ .$$
In particular, from the definition of $d'$,
$$s=\underline s_{d'-1}(T_{\overline a}u)\le \underline s_{d'-1}(u)\ ,$$
which imposes, in view of the assumption,
$d'-1\le k+m-1$, in particular, in view of (\ref{borneinfdim}),
$$d\le k\ .$$
Finally, notice that, since $a$ has degree $d$, the dimension of  $(aL^2_+)^\perp $ is $d$.
Hence, by the min-max formula again, 
$$\underline s_d(u)\le \sup _{h\in aL^2_+\setminus \{ 0\} }\frac{\Vert H_u(h)\Vert }{\Vert h\Vert}\le s<\underline s_{k-1}(u)\ .$$
This imposes $d\ge k$, and finally
$$d=k\ ,\ d'=k+m\ ,$$
and part (2) of the theorem is proved. 

In order to prove part (1), we apply 
properties (\ref{Eunorm}) and (\ref{Hunorm}) of Lemma \ref{tops}. We describe elements of $E_{T_{\overline a}u}(s)$ as 
$$h(z)=Q(z)h_0(z)\ ,$$
where $h_0$ is outer and $Q\in \C _{k+m-1}[z]$. Moreover, if $h\in E_u(s)$, then $h=a\tilde h$, $H_u(h)=sa\tilde f$, where $\tilde h, \tilde f\in E_{T_{\overline a}u}(s)$.
This reads 
$$Q(z)=a(z)\tilde Q(z)\ ,$$
If we set
$$a(z)=\frac{z^k\overline D\left (\frac 1z\right )}{D(z)}\ ,$$
where $D\in \C _k[z]$ and $D(0)=1$, and $D$ has no zeroes in $\D $, this implies
$$Q(z)=z^k\overline D\left (\frac 1z\right )P(z)\ ,\ P\in \C _{m-1}[z]\ .$$
Moreover, 
$$H_{T_{\overline a}u}(h)(z)=s\expo_{i\varphi }z^{m+k-1}\overline Q\left (\frac 1z\right )h_0(z)=s\expo_{i\varphi } D(z)z^{m-1}\overline P\left ( \frac 1z\right )h_0(z)\ ,$$
and
$$H_u(h)(z)=a(z)H_{T_{\overline a}u}(h)(z)=s\expo_{i\varphi }  z^k\overline D\left (\frac 1z\right )z^{m-1}\overline P\left ( \frac 1z\right )h_0(z)\ .$$
Changing $P$ into $P\expo _{-i\varphi /2}$, this proves part (1) of the theorem.
\end{proof}

\end{appendix}

\end{document}